\documentclass{amsart}
\usepackage[english]{babel}

\usepackage[latin1]{inputenc}
\usepackage[T1]{fontenc}

\usepackage{amssymb}
\usepackage[arrow,matrix,cmtip]{xy}
\usepackage{paralist}

\title[Topological Invariants]{Topological Invariants of Anosov Representations}
\author{Olivier Guichard}
\address{CNRS, Laboratoire de Math\'ematiques d'Orsay, Orsay cedex, F-91405\\
Universit\'e Paris-Sud, Orsay cedex, F-91405}
\author{Anna Wienhard}
\address{Department of Mathematics, Princeton University\\
Fine Hall, Washington Road, Princeton, NJ 08540, USA}
\thanks{A.W. was partially supported by the National Science
  Foundation under agreement No. DMS-0604665, DMS-0803216 and DMS-0846408.
  O.G. was partially supported by the Agence Nationale de la
  Recherche under ANR's projects Repsurf (ANR-06-BLAN-0311) and
  ETTT (ANR-09-BLAN-0116-01) and by the National Science Foundation
  under agreement No. DMS-0635607.}
\date{\today}
\keywords{Surface Groups, Maximal Representations, Obstructions
  Classes, Representation Varieties, Anosov Representations} 
\subjclass[2000]{57M50, 20H10}

\numberwithin{equation}{section}

\theoremstyle{plain}

\newtheorem{thm}[equation]{Theorem}
\newtheorem{prop}[equation]{Proposition}
\newtheorem{lem}[equation]{Lemma}

\newtheorem{cor}[equation]{Corollary}
\newtheorem*{thm*}{Theorem}

\newtheorem{prop_intro}{Proposition}
\newtheorem{thm_intro}[prop_intro]{Theorem}
\newtheorem{cor_intro}[prop_intro]{Corollary}
\newtheorem{conj_intro}[prop_intro]{Conjecture}

\theoremstyle{definition}

\newtheorem{defi}[equation]{Definition}
\newtheorem{fact}[equation]{Fact}
\newtheorem{facts}[equation]{Facts}
\newtheorem{nota}[equation]{Notation}
\newtheorem{cond}[equation]{Condition}

\theoremstyle{remark}

\newtheorem{rem_intro}{Remark}
\newtheorem{remark}[equation]{Remark}

\newtheorem{examp_intro}[prop_intro]{Examples}

\newcommand{\ZZ}{\mathbf{Z}}
\newcommand{\RR}{\mathbf{R}}
\newcommand{\CC}{\mathbf{C}}

\newcommand{\HH}{\mathbb{H}}

\newcommand{\hH}{\mathcal{H}}
\newcommand{\xX}{\mathcal{X}}
\newcommand{\Ff}{\mathcal{F}}

\newcommand{\g}{\gamma}

\newcommand{\Ll}{\mathcal{L}}

\newcommand{\PSL}{\mathrm{PSL}}
\newcommand{\SL}{\mathrm{SL}}
\newcommand{\Sp}{\mathrm{Sp}}
\newcommand{\SO}{\mathrm{SO}}
\newcommand{\SOcon}{\mathrm{SO}_\circ}
\newcommand{\diag}{\mathrm{diag}} 
\newcommand{\Gk}{\mathrm{Sp}(2n, \RR)_{(k)}}

\renewcommand{\O}{\mathrm{O}}
\newcommand{\PGL}{\mathrm{PGL}}

\newcommand{\PU}{\mathrm{PU}}

\newcommand{\PSp}{\mathrm{PSp}}
\newcommand{\PSO}{\mathrm{PSO}}
\newcommand{\GL}{\mathrm{GL}}
\newcommand{\sym}{\mathrm{Sym}}
\newcommand{\rep}{\mathrm{Rep}}

\newcommand{\id}{\mathrm{Id}}
\newcommand{\h}{\mathrm{H}}

\newcommand{\map}{\mathcal{M}od}

\newcommand{\Span}{\mathrm{Span}}

\renewcommand{\hom}{\mathrm{Hom}}
\newcommand{\stab}{\mathrm{Stab}}

\newcommand{\moins}{\smallsetminus}
\newcommand{\cupprod}{\smallsmile}
\newcommand{\sep}{,\,}
\newcommand{\longto}{\longrightarrow}

\newcommand{\bqn}{\begin{equation*}}
\newcommand{\eqn}{\end{equation*}}

\newcommand{\bq}{\begin{equation}}
\newcommand{\eq}{\end{equation}}

\def\bal#1\eal{\begin{align}#1\end{align}}
\def\baln#1\ealn{\begin{align*}#1\end{align*}}

\newcommand{\homHano}{\hom_{H\textup{-Anosov}}}
\newcommand{\homGLano}{\hom_{\GL(n,\RR)\textup{-Anosov}}}
\newcommand{\hommax}{\hom_{\textup{max}}}
\newcommand{\homHit}{\hom_{\textup{Hitchin}}}
\newcommand{\repmax}{\rep_{\textup{max}}}
\newcommand{\repHit}{\rep_{\textup{Hitchin}}}
\newcommand{\hommaxzero}{\hom_{ \textup{max}, sw_1=0}}
\newcommand{\repmaxzero}{\rep_{ \textup{max}, sw_1=0}}

\newcommand{\FF}{\mathbf{F}}

\newcommand{\transpose}{\,{}^t\!}

\newcommand{\GenTor}{[\Sigma]}
\newcommand{\LZPlus}{L_{0+}}

\newcommand{\homMaxFZ}{\hom_{ \textup{max}, sw_1=0}}
\newcommand{\homMaxFourZero}{\hom_{ \textup{max}, sw_1=0}(
  \pi_1( \Sigma) , \Sp(4, \RR))}
\newcommand{\repMaxFourZero}{\rep_{ \textup{max}, sw_1=0}(
  \pi_1( \Sigma) , \Sp(4, \RR))}
\newcommand{\repmaxneqzero}{\rep_{ \textup{max}, sw_1 \neq 0}(
  \pi_1( \Sigma) , \Sp(4, \RR))}
\newcommand{\homMaxFour}{\hom_{ \textup{max}}(
  \pi_1( \Sigma) , \Sp(4, \RR))}
\newcommand{\homHitFour}{\hom_{ \textup{Hitchin}}(
  \pi_1( \Sigma) , \Sp(4, \RR))}
\newcommand{\repHitFour}{\rep_{ \textup{Hitchin}}(
  \pi_1( \Sigma) , \Sp(4, \RR))}
\newcommand{\homMaxLl}{\hom_{\textup{max},sw_1=0}^{\Ll, \gamma}(
  \pi_1( \Sigma) , \Sp(4, \RR))} 
\newcommand{\homMaxLlPlus}{\hom_{\textup{max},sw_1=0}^{\Ll_+, \gamma}(
  \pi_1( \Sigma) , \Sp(4, \RR))} 
\newcommand{\homMaxLlPlusNoGrp}{\hom_{\textup{max},sw_1=0}^{\Ll_+, \gamma}} 
\newcommand{\homMaxHatLlPlus}{\hom_{\textup{max},sw_1=0}^{\Ll_+, \hat{\gamma}}(
  \widehat{\pi_1(\Sigma)} , \Sp(4, \RR))} 
\newcommand{\homMaxPLnot}{\hom_{\textup{max},sw_1=0}^{\Ll_+, \gamma, \LZPlus}( \pi_1( \Sigma) ,
     \Sp(4, \RR))}

\begin{document}

\begin{abstract}
  We define new topological invariants for Anosov representations and
  study them in detail for maximal representations of the fundamental
  group of a closed oriented surface $\Sigma$ into the symplectic
  group $\Sp(2n,\RR)$.  In particular we show that the invariants
  distinguish connected components of the space of symplectic maximal
  representations other than Hitchin components.  Since the invariants
  behave naturally with respect to the action of the mapping class
  group of $\Sigma$, we obtain from this the number of components of
  the quotient by the mapping class group action.

  For specific symplectic maximal representations we compute the
  invariants explicitly. This allows us to construct nice model
  representations in all connected components. The construction of
  model representations is of particular interest for $\Sp(4,\RR)$,
  because in this case there are $-1-\chi(\Sigma)$ connected
  components in which all representations are Zariski dense and no
  model representations were known so far. Finally, we use the model
  representations to draw conclusions about the holonomy of symplectic
  maximal representations.
\end{abstract}

\maketitle

\section{Introduction}
\label{sec_intro}

Let $\Sigma$ be a closed oriented connected surface of negative Euler
characteristic, $G$ a connected Lie group.  The obstruction to lifting
a representation $\rho: \pi_1(\Sigma) \to G$ to the universal cover of
$G$ is a characteristic class of $\rho$ which is an element of
$\h^2(\Sigma; \pi_1(G)) \cong \pi_1(G)$.

When $G$ is compact it is a consequence of the famous paper of Atiyah
and Bott \cite{Atiyah_Bott} that the connected components of
\[
\hom(\pi_1(\Sigma), G)/G
\]
are in one-to-one correspondence with the elements of $\pi_1(G)$. When
$G$ is a complex Lie group the analogous result has been conjectured
by Goldman \cite[p.~559]{Goldman_components} and proved by Li
\cite[Theorem~0.1]{Li}.

When $G$ is a real non-compact Lie group, this correspondence between
connected components of $\hom(\pi_1(\Sigma), G)/G$ and elements of
$\pi_1(G)$ fails to hold. Obviously characteristic classes of
representations still distinguish certain connected components of
$\hom(\pi_1(\Sigma), G)/G$, but they are not sufficient to distinguish
all connected components.

\begin{examp_intro}\label{exintro}
  Here are some examples of this phenomenon:
  \begin{asparaenum}
  \item For $n\geq 3$, the characteristic class of a representation of
    $\pi_1(\Sigma)$ into $\SL(n,\RR)$ is an element of $\ZZ/2\ZZ$.
    But the space
    \[\hom(\pi_1(\Sigma) ,\SL(n,\RR))/\PGL(n,\RR)\]
    has three connected components \cite[Theorem~B]{Hitchin}.

  \item For representations of $\pi_1(\Sigma)$ into $\PSL(2,\RR)$ the
    Euler number does distinguish the $4g-3$ connected components
    \cite[Theorem~A]{Goldman_components}. For representations into
    $\SL(2,\RR)$ the Euler number is not sufficient to distinguish
    connected components, there are $2^{2g+1} + 2g-3$ components, and
    in particular there are $2^{2g}$ components of maximal (or
    minimal) Euler number, each of which corresponds to the choice of
    a spin structure on $\Sigma$.
  \item \label{exactcount} For representations of $\pi_1(\Sigma)$ into
    $\Sp(2n,\RR)$ the characteristic class which generalizes the Euler
    number is an element of $\h^2(\Sigma; \pi_1(\Sp(2n,\RR))) \cong
    \ZZ$. It is bounded in absolute value by $n(g-1)$. The subspace of
    representations where it equals $n(g-1)$ is called the space of
    maximal representations. This subspace decomposes into several
    connected components, $3\times 2^{2g}$ when $n\geq 3$
    \cite[Theorem~8.7]{GarciaPrada_Gothen_Mundet} and $(3 \times
    2^{2g} + 2g-4)$ when $n=2$ \cite[Theorem, p.~824]{Gothen}. The
    space of maximal representations and its connected components are
    discussed in detail in this article.
  \end{asparaenum}
\end{examp_intro}

We introduce new topological invariants for representations $\rho:
\pi_1(M) \to G$, whenever $\rho$ is an {\em Anosov
  representation}. Let us sketch the definition.  Let $M$ be a compact
manifold equipped with an Anosov flow that has an invariant volume
form.  A representation $\rho:\pi_1(M) \to G$ is said to be a
$(G,H)$-Anosov representation if the associated $G/H$-bundle over $M$
admits a section whose image is a hyperbolic set for the induced flow
on $\widetilde{M} \times_\rho G/H$ (see Section~\ref{sec_anosov} for
details). We call such a section an \emph{Anosov section}.

\begin{thm_intro}\label{thm_intro:invariants}
  Let $\rho:\pi_1(M) \to G$ be a $(G,H)$-Anosov representation. Then
  the Anosov section is unique and defines a canonical principal
  $H$-bundle over $M$.  Hence there is a well defined map
  \[
  \homHano( \pi_1(M), G) \longrightarrow \mathcal{B}_H(M),
  \]
  where $\homHano( \pi_1(M), G)$ denotes the subspace of
  $(G,H)$-Anosov representations and $\mathcal{B}_H(M)$ the set of
  gauge isomorphism classes of principal $H$-bundles over $M$.  This
  map is continuous and natural with respect to:
  \begin{itemize}
  \item taking covers of $M$,
  \item certain morphisms of pairs $(G,H) \to (G' ,H')$ (see
    Lemmas~\ref{lem_anosov_constructions} and \ref{lem:finite_cover}).
  \end{itemize}
\end{thm_intro}
As a consequence the topological type of the $H$-bundle canonically
given by the Anosov section gives rise to topological invariants of
$\rho$.  Some general properties of the invariants associated with
Anosov sections are discussed in Section~\ref{sec:char_class_in_gal}.

\subsection{Maximal representations into $\Sp(2n,\RR)$}
Our main focus lies on maximal representations into
$\Sp(2n,\RR)$. Maximal representations into $\Sp(2n,\RR)$ are
$(\Sp(2n,\RR), \GL(n,\RR))$-Anosov representations
\cite[Theorem~6.1]{Burger_Iozzi_Labourie_Wienhard}.  More precisely,
with respect to some hyperbolic metric on $\Sigma$, the geodesic flow
is an Anosov flow on the unit tangent bundle $T^1\Sigma$ and the
Liouville form is invariant.  The fundamental group $\pi_1(T^1\Sigma)$
is a central extension of $\pi_1(\Sigma)$, with natural projection
$\pi: \pi_1(T^1\Sigma) \to \pi_1(\Sigma)$. Let $\rho: \pi_1(\Sigma)
\to \Sp(2n,\RR)$ be a maximal representation, then the composition
$\rho\circ \pi: \pi_1(T^1\Sigma) \to \Sp(2n,\RR)$ is a $(\Sp(2n,\RR),
\GL(n,\RR))$-Anosov representation.  The topological invariants
obtained by Theorem~\ref{thm_intro:invariants} are the characteristic
classes of a $\GL( n, \RR)$-bundle over $T^1\Sigma$. We only consider
the first and second Stiefel-Whitney classes $sw_1(\rho\circ \pi) \in
\h^1(T^1\Sigma; \FF_2)$ and $sw_2(\rho\circ\pi) \in \h^2(T^1\Sigma;
\FF_2)$.
\begin{thm_intro}\label{thm_intro:restrictions}
  Let $\rho: \pi_1(\Sigma) \to \Sp(2n,\RR)$ be a maximal
  representation. Then the topological invariants $sw_1(\rho) =
  sw_1(\rho\circ \pi) \in \h^1(T^1\Sigma; \FF_2)$ and $sw_2(\rho)=
  sw_2(\rho\circ\pi) \in \h^2(T^1\Sigma; \FF_2)$ are subject to the
  following constraints:
  \begin{enumerate}
  \item The image of
    \[sw_1 : \hommax( \pi_1(\Sigma),G) \longrightarrow \h^1(T^1 \Sigma
    ; \FF_2 )\] is contained in one coset of $\h^1(\Sigma ; \FF_2 )$.
    \begin{itemize}
    \item For $n$ even, $sw_1( \rho)$ is in $\h^1( \Sigma; \FF_2)
      \subset \h^1( T^1\Sigma; \FF_2)$,
    \item for $n$ odd, $sw_1( \rho)$ is in $\h^1( T^1\Sigma; \FF_2)
      \moins \h^1( \Sigma; \FF_2) $.
    \end{itemize}
  \item The image of
    \[sw_2 : \hommax( \pi_1(\Sigma),G) \longrightarrow \h^2(T^1 \Sigma
    ; \FF_2 )\] lies in the image of $\h^2( \Sigma ; \FF_2 ) \to
    \h^2(T^1 \Sigma ; \FF_2)$.
  \end{enumerate}
\end{thm_intro}

\begin{rem_intro}
  The homomorphism $\h^1( \Sigma; \FF_2) \to \h^1( T^1\Sigma; \FF_2)$
  is induced by the natural fibration $T^1 \Sigma \to \Sigma$. The
  Gysin exact sequence (see Equation~\eqref{eq_gysin}) implies that it
  is injective and that is image is of index $2$.
\end{rem_intro}

In the case when $n=2$, that is, for maximal representation $\rho:
\pi_1(\Sigma) \to \Sp(4,\RR)$, let $\homMaxFourZero$ denote the
subspace of maximal representations where the first Stiefel-Whitney
class vanishes.  This means that the $\GL(2, \RR)$-bundle over $T^1
\Sigma$ admits a reduction of the structure group to $\GL^+(2,\RR)$,
equivalently the corresponding $\RR^2$-vector bundle is orientable.  A
reduction of the structure group to $\GL^+(2,\RR)$ gives rise to an
Euler class, but since an orientable bundle does not have a canonical
orientation this reduction is not canonical. To circumvent this
problem, we introduce an \emph{enhanced} representation space, which
involves the choice of a nontrivial element $\gamma \in
\pi_1(\Sigma)$. For pairs $(\rho, L_+)$ consisting of a maximal
representation with vanishing first Stiefel-Whitney class and an
\emph{oriented} Lagrangian $L_+\subset \RR^4$ which is fixed by
$\rho(\gamma)$, there is a well-defined Euler class (see
Section~\ref{sec:sp4case}).

\begin{thm_intro}\label{thm_intro:restricEuler}
  Let $\rho: \pi_1(\Sigma) \to \Sp(4,\RR)$ be a maximal representation
  with $sw_1(\rho) = 0$. Let $\gamma \in \pi_1(\Sigma) \moins \{1\}$
  and $L_+$ an oriented Lagrangian fixed by $\rho(\gamma)$.  Then the
  Euler class $e_\gamma(\rho, L_+) \in \h^2(T^1\Sigma; \ZZ)$ lies in
  the image of $\h^2(\Sigma; \ZZ) \to \h^2(T^1\Sigma;\ZZ)$.
\end{thm_intro}

For every possible topological invariant satisfying the above
constraints we construct explicit representations $\rho:\pi_1(\Sigma)
\to \Sp(2n,\RR)$ realizing this invariant (see
Section~\ref{sec:model_intro} and Section~\ref{sec:examples} for the
construction of the representations and Section~\ref{sec:computations}
for the calculation of the invariants). From this we deduce a lower
bound on the connected components of the space $\hommax(\pi_1(\Sigma),
\Sp(2n,\RR))$ of maximal representations.

\begin{prop_intro}\label{prop:lowerbound_intro} 
  Let $\hommax(\pi_1(\Sigma), \Sp(2n,\RR))$ be the space of maximal
  representations.
  \begin{enumerate}
  \item If $n\geq 3$ the space $\hommax(\pi_1(\Sigma), \Sp(2n,\RR))$
    has at least $3\times 2^{2g}$ connected components.
  \item The space $\hommax(\pi_1(\Sigma), \Sp(4,\RR))$ has at least $3
    \times 2^{2g} + 2g-4$ connected components.
  \end{enumerate}
\end{prop_intro}

Our method (so far) only gives a lower bound on the number of
connected components. To obtain an exact count of the number of
connected components using the invariants defined here, a closer
analysis for surfaces with boundary would be necessary (see
\cite{Goldman_components} for the case when $n=1$).

Fortunately, the correspondence between representations and Higgs
bundles allows to use algebro-geometric methods to study the topology
of $\rep(\pi_1(\Sigma), G):=\hom(\pi_1(\Sigma), G)/G$.  These methods
have been developed by Hitchin \cite{Hitchin_selfdual} and applied to
representations into Lie groups of Hermitian type in
\cite{Bradlow_GarciaPrada_Gothen, Bradlow_GarciaPrada_Gothen_survey,
  GarciaPrada_Gothen_Mundet, GarciaPrada_Mundet, Gothen} leading to
the exact count mentioned above in
Examples~\ref{exintro}.(\ref{exactcount}).

Combining Proposition~\ref{prop:lowerbound_intro} with this exact
count we can conclude that the invariants defined here distinguish
connected components other than Hitchin components.

More precisely, let $\homHit(\pi_1(\Sigma), \Sp(2n,\RR))$ be the space
of Hitchin representations; by definition it is the union of the
connected components of $\hom(\pi_1(\Sigma), \Sp(2n,\RR))$ containing
representations of the form $\phi_{irr} \circ \iota$ where $\iota:
\pi_1( \Sigma) \to \SL(2, \RR)$ is a discrete embedding and
$\phi_{irr}: \SL(2, \RR) \to \Sp(2n, \RR)$ is the irreducible
representation of $\SL(2, \RR)$ of dimension $2n$.  Hitchin
representations are maximal representations.

\begin{thm_intro}\label{thm_intro:symp_maximal}
  Let $n\geq 3$. Then the topological invariants of
  Theorem~\ref{thm_intro:restrictions} distinguish connected
  components of $\hommax \moins \homHit$. More precisely,
  \begin{multline*}
    \hommax(\pi_1(\Sigma), \Sp(2n,\RR)) \moins
    \homHit(\pi_1(\Sigma), \Sp(2n,\RR)) \\
    \xrightarrow{sw_1,sw_2} \h^1(T^1\Sigma; \FF_2) \times
    \h^2(T^1\Sigma; \FF_2)
  \end{multline*}
  induces a bijection from $\pi_0(\hommax \moins \homHit)$ to the set
  of pairs satisfying the constraints of
  Theorem~\ref{thm_intro:restrictions}.
\end{thm_intro}

It is easy to see that, when $n$ is even, the first Stiefel-Whitney
class of a Hitchin representation vanishes, i.e.\ one has the
inclusion $\homHit \subset \homMaxFZ$.

\begin{thm_intro}\label{thm_intro:n=2_components}
  \begin{asparaenum}
  \item The Euler class defines a map
    \[\homMaxFourZero \moins
    \homHit(\pi_1(\Sigma), \Sp(4,\RR)) \longrightarrow \h^2(T^1\Sigma;
    \ZZ)\] which induces a bijection from $\pi_0(\homMaxFZ \moins
    \homHit)$ to the image of $\h^2(\Sigma; \ZZ)$ in
    $\h^2(T^1\Sigma;\ZZ)$. In particular, the Euler class
    distinguishes connected components in $\homMaxFZ \moins \homHit$.

  \item The components of $\hommax \moins \homMaxFZ$ are distinguished
    by the first and second Stiefel-Whitney classes, i.e.\ the map
    \begin{multline*}
      \homMaxFour \moins \homMaxFourZero \\
      \xrightarrow{sw_1,sw_2} \big(\h^1(\Sigma; \FF_2)\moins\{0\}\big)
      \times \h^2(\Sigma; \FF_2)
    \end{multline*}
    induces a bijection at the level of connected components.
  \end{asparaenum}

\end{thm_intro}

\begin{rem_intro}\label{rem_intro_inv_for_Hitchin}
  Hitchin representations are not only $(\Sp(2n,\RR),
  \GL(n,\RR))$-Anosov representations, but $(\Sp(2n,\RR), A)$-Anosov
  representations, where $A$ is the subgroup of diagonal matrices
  \cite[Theorems~4.1 and 4.2]{Labourie_anosov}.  Applying
  Theorem~\ref{thm_intro:invariants} to the pair $(G,H) =
  (\Sp(2n,\RR), A)$ one can define first Stiefel-Whitney classes
  $sw_1^A(\rho)$ in $\h^1( T^1\Sigma; \FF_2)$; similarly to the above
  discussion these invariants are shown to belong to $\h^1( T^1\Sigma;
  \FF_2) \moins\h^1( \Sigma; \FF_2) $ and distinguish the $2^{2g}$
  connected components of $\homHit(\pi_1(\Sigma), \Sp(2n,\RR))$ (see
  Section~\ref{sec:invforHitchin}). In
  Section~\ref{sec:gener-hitch-comp} we discuss the case of the
  Hitchin component of a general split real simple Lie group (see
  Theorem~\ref{thm_Hitchin_cover} and
  Theorem~\ref{thm_Hitchinclassical}).
\end{rem_intro}
\begin{rem_intro}
  The existence of $2^{2g}$ Hitchin components is due to the center of
  $\Sp(2n,\RR)$: they all project to the same component in $\hom(
  \pi_1( \Sigma), \PSp(2n, \RR))$.  For $n$ odd, $n \geq 3$, the
  abundance of non-Hitchin connected components in the space of
  maximal representations is also explained by the center. However,
  for $n$ even, the abundance of non-Hitchin connected components
  pertains when we consider representations into the adjoint group. A
  precise statement is the following theorem (see
  Section~\ref{sec:cover_sp}).
\end{rem_intro}

\begin{thm_intro}\label{thm_intro_comp_PSp}
  If $n$ is odd and $n\geq 3$, then the space $\hommax(\pi_1(\Sigma),
  \PSp(2n ,\RR))$ has $3$ connected components.

  If $n$ is even and $n \geq 4$, then there are $2^{2g}+2$ connected
  components of $\hommax(\pi_1(\Sigma), \PSp(2n ,\RR))$ that lift to
  $\Sp(2n,\RR)$.
 
  There are $2^{2g}+2g-2$ connected components of
  $\hommax(\pi_1(\Sigma), \PSp(4 ,\RR))$ that lift to $\Sp(4,\RR)$.
\end{thm_intro}

This result gives lower bounds for the number of components of the
space $\hommax(\pi_1(\Sigma), \PSp(2n ,\RR))$ when $n$ is even.

\begin{rem_intro}
  According to \cite[Section~7]{Bradlow_GarciaPrada_Gothen_survey} the
  space $\hommax(\pi_1(\Sigma), \PSp(4 ,\RR))\cong
  \hommax(\pi_1(\Sigma), \SOcon(2,3))$ has $2^{2g+1}+ 4g-5$ connected
  components. In particular, there are $2^{2g} + 2g-3$ components that
  do not lift to $\Sp(4,\RR)$.
\end{rem_intro}

Maximal representations into covers of $\Sp(2n, \RR)$ are also Anosov,
hence we have the corresponding topological invariants. We describe
their properties in Section~\ref{sec:cover_sp}. We get in particular
the following

\begin{thm_intro}\label{thm_intro:coverSP}
  Let $\Gk$ be the connected $k$-fold cover of $\Sp(2n, \RR)$. Then
  the space $\hommax(\pi_1(\Sigma), \Gk)$ is nonempty if and only if
  $\chi(\Sigma)$ is a multiple of $k$. In that case
  \[ \hommax(\pi_1(\Sigma), \Gk) \longto \hommax(\pi_1(\Sigma),
  \Sp(2n, \RR)) \] is the trivial cover of degree $k^{2g}$. In
  particular
  \[ \# \pi_0 \big( \hommax(\pi_1(\Sigma), \Gk) \big) = k^{2g} \#
  \pi_0 \big ( \hommax(\pi_1(\Sigma), \Sp(2n, \RR))\big). \]
\end{thm_intro}

\subsection{The action of the mapping class group}
The first and second Stiefel-Whitney classes of a maximal
representation $\rho$ do not change if $\rho$ is conjugated by an
element $\Sp(2n,\RR)$. Thus, they give well defined functions:
\begin{equation}
  \label{eq:map}
  sw_i: \repmax(\pi_1(\Sigma), \Sp(2n,\RR)) \longrightarrow
  \h^i(T^1\Sigma; \FF_2).
\end{equation}
The mapping class group $\map(\Sigma)$ acts by precomposition on
$\repmax$; this action is properly discontinuous \cite[Theorem
1.0.2]{Labourie_energy} \cite[Theorem 1.1]{Wienhard_mapping} and by
Theorem~\ref{thm_intro:invariants} the map \eqref{eq:map} is
equivariant with respect to this action and the natural action of
$\map(\Sigma)$ on $\h^i(T^1\Sigma; \FF_2)$.

For the Euler class $e_\gamma$ (see
Theorem~\ref{thm_intro:restricEuler}) there is a corresponding
statement of equivariance for the subgroup of $\map(\Sigma)$ fixing
the homotopy class of $\gamma$.

This allows us to determine the number of connected components of the
space $\repmax(\pi_1(\Sigma), \Sp(2n,\RR))/\map(\Sigma)$.

\begin{thm_intro}\label{thm:components_quotient_intro}
  If $n\geq 3$, the space $\repmax(\pi_1(\Sigma),
  \Sp(2n,\RR))/\map(\Sigma)$ has $6$ connected components.

  The space $\repmax(\pi_1(\Sigma), \Sp(4,\RR))/\map(\Sigma)$ has
  $2g+2$ connected components.
\end{thm_intro}

\subsection{Model representations}\label{sec:model_intro}
Given two representations it is in general very difficult to determine
whether they lie in the same connected component or not.  The
invariants defined here can be computed rather explicitly and hence
allow us to decide in which connected component of
$\hommax(\pi_1(\Sigma), \Sp(2n,\RR))$ a specific representation lies
in.  This computability enables us to give particularly nice model
representations in all connected components.

An easy way to construct maximal representations $\rho:\pi_1(\Sigma)
\to \Sp(2n,\RR)$ is by composing a discrete embedding $\pi_1(\Sigma)
\to \SL(2,\RR)$ with a tight homomorphism of $\SL(2,\RR)$ into
$\Sp(2n,\RR)$ (see \cite{Burger_Iozzi_Wienhard_tight} for the notion
of tight homomorphism; here they can be characterized as the morphisms
inducing multiplication by $n$ at the level of fundamental groups).
The composition with the $2n$-dimensional irreducible representation
of $\SL(2,\RR)$ into $\Sp(2n,\RR)$ is called an {\em irreducible
  Fuchsian representation}.  Hitchin representations are precisely
deformations of such representations.  The composition with the
diagonal embedding of $\SL(2,\RR)$ into the subgroup $\SL(2,\RR)^n <
\Sp(2n,\RR)$ is called a {\em diagonal Fuchsian representation}.  The
centralizer of the image of the diagonal embedding of $\SL(2,\RR)$ is
isomorphic to $\O(n)$. Any representation can be twisted by a
representation into its centralizer, thus any diagonal Fuchsian
representation can be twisted by a representation $\pi_1(\Sigma) \to
\O(n)$, defining a {\em twisted diagonal representations}.  A
representation obtained by one of these constructions will be called a
{\em standard maximal representation} (see
Section~\ref{sec_standard}).

\begin{thm_intro}
  \label{thm_intro:symp_n>2}
  Let $n\geq 3$. Then every maximal representation $\rho:
  \pi_1(\Sigma) \to \Sp(2n,\RR)$ can be deformed to a standard maximal
  representation.
\end{thm_intro}
\begin{cor_intro}\label{cor_intro:deform_closed}
  Let $n\geq 3$. Then any maximal representation $\rho: \pi_1(\Sigma)
  \to \Sp(2n,\RR)$ can be deformed to a maximal representation whose
  image is contained in a proper closed subgroup of $\Sp(2n,\RR)$.
\end{cor_intro}

\begin{rem_intro}
  This conclusion can also be obtained from
  \cite[Section~5]{GarciaPrada_Gothen_Mundet} because the Higgs
  bundles for standard maximal representations can be described quite
  explicitly.
\end{rem_intro}
Our computations of the topological invariants in
Section~\ref{sec:computations} give more precise information on when a
maximal representation can be deformed to an irreducible Fuchsian or a
diagonal Fuchsian representation:

\begin{cor_intro}\label{cor_intro:deform}
  Let $\rho:\pi_1(\Sigma) \to \Sp(2n,\RR)$ be a maximal
  representation. Then $\rho$ can be deformed either to an irreducible
  Fuchsian representation or to a diagonal Fuchsian representation if
  \begin{enumerate}
  \item for $n= 2m$, $m>2$, $sw_1(\rho) = 0$ and $sw_2(\rho) = m
    \frac{\chi(\Sigma)}{2} \mod 2$,
  \item for $n= 2m+1$, $sw_2(\rho) = m \frac{\chi(\Sigma)}{2} \mod 2$.
  \end{enumerate}
\end{cor_intro}

The case of $\Sp(4,\RR)$ is different as
Theorem~\ref{thm_intro:symp_n>2} and
Corollary~\ref{cor_intro:deform_closed} do not hold anymore.  From the
count of the connected components of $\repmax(\pi_1(\Sigma),
\Sp(4,\RR))$ in \cite{Gothen}, one can conclude that there are $2g-3$
exceptional components not containing any standard maximal
representations.

To construct model representations in these components, we decompose
$\Sigma = \Sigma_l \cup \Sigma_r$ into two subsurfaces and define a
representation of $\pi_1(\Sigma)$ by amalgamation of an irreducible
Fuchsian representation of $\pi_1(\Sigma_l)$ with a deformation of a
diagonal Fuchsian representation of $\pi_1(\Sigma_r)$. We call these
representations {\em hybrid representations} (see
Section~\ref{sec:description_hybrid} for details).
  
We compute the topological invariants of these representations
explicitly.  Allowing the Euler characteristic of the subsurface
$\Sigma_l$ to vary between $3-2g$ and $-1$, we obtain $2g-3$ hybrid
representations which exhaust the $2g-3$ exceptional components of the
space of maximal representations into $\Sp(4,\RR)$. We conclude

\begin{thm_intro}
  \label{thm_intro:symp_n=2}
  Every maximal representation $\rho: \pi_1(\Sigma) \to \Sp(4,\RR)$
  can be deformed to a standard maximal representation or a hybrid
  representation.
\end{thm_intro}

\begin{rem_intro}
  To obtain Theorem~\ref{thm_intro:symp_n=2} it is essential that we
  are able to compute the topological invariants
  explicitly. Geometrically there is no obvious reason why different
  hybrid representations lie in different connected components.  In
  particular, our results on the topological invariants imply that
  similar constructions by amalgamation (see
  Section~\ref{sec:other_amalgams}) give representations which can be
  deformed to twisted diagonal representations.
\end{rem_intro}

From the computations of the topological invariants we also deduce the
following
\begin{thm_intro}
  \label{thmintro_Zdensity}
  Any representation in $\homMaxFourZero$ with Euler class not equal
  to $(g-1) \GenTor $ has Zariski dense image.
\end{thm_intro}

\begin{rem_intro} 
  Here the class $\GenTor \in \h^2 ( T^1 \Sigma; \ZZ)$ is the image of
  the orientation class in $\h^2( \Sigma; \ZZ)$ under the natural map
  $\h^2( \Sigma; \ZZ) \to \h^2( T^1 \Sigma; \ZZ)$; it is a torsion
  class of order $2g-2$ (see Appendix~\ref{subsec:cohomologysigma}).
\end{rem_intro}

\begin{rem_intro}
  A similar result is proved in
  \cite[Th.~1.1.(3)]{Bradlow_GarciaPrada_Gothen_inpreparation} using
  the theory of Higgs bundles.
\end{rem_intro}

\subsection{Holonomies of maximal representations}
A direct consequence of the fact that maximal representations
$\rho:\pi_1(\Sigma)\to \Sp(2n,\RR)$ are $(\Sp(2n,\RR),
\GL(n,\RR))$-Anosov representations is that the holonomy
$\rho(\gamma)$ is conjugate to an element of $\GL(n,\RR)$ for every
$\gamma \in \pi_1(\Sigma)$. More precisely $\rho(\gamma)$ fixes two
transverse Lagrangians, one, $L^s$, being attractive, the other being
repulsive.  From this it follows that the holonomy $\rho(\gamma)$ is
an element of $\GL(L^s)$ whose eigenvalues are strictly bigger than
one.

For representations in the Hitchin components we have moreover that
$\rho(\gamma) \in \GL(L^s)$ is a regular semi-simple element
\cite[Prop.~3.4]{Labourie_anosov},
\cite[Prop.~8]{Guichard_CompHit}. This does not hold for other
connected components of $\hommax(\pi_1(\Sigma), \Sp(2n,\RR))$. Using
the description of model representations in
Theorem~\ref{thm_intro:symp_n>2} and Theorem~\ref{thm_intro:symp_n=2}
we prove

\begin{thm_intro}
  \label{thm_intro:holonomy}
  Let $\hH$ be a connected component of \[\repmax(\pi_1(\Sigma),
  \Sp(2n,\RR)) \moins \repHit(\pi_1(\Sigma), \Sp(2n,\RR)),\] and let
  $\g \in \pi_1(\Sigma)\moins\{1\}$ be an element corresponding to a
  simple closed curve.  If $n=2$, the genus of $\Sigma$ is $2$ and
  $\g$ is separating, we require that $\hH$ is not the connected
  component determined by $sw_1 = 0$ and $e_{\gamma} = 0$.  Then there
  exist
  \begin{enumerate}
  \item a representation $\rho \in \hH$ such that the Jordan
    decomposition of $\rho(\g)$ in $\GL(L^s) \cong \GL(n,\RR)$ has a
    nontrivial parabolic component.
  \item a representation $\rho' \in \hH$ such that the Jordan
    decomposition of $\rho'(\g)$ in $\GL(L^s) \cong\GL(n,\RR)$ has a
    nontrivial elliptic component.
  \end{enumerate}
\end{thm_intro}

\subsection{Other maximal representations}
Maximal representations $\rho:\pi_1(\Sigma) \to G$ can be defined
whenever $G$ is a Lie group of Hermitian type, and they are always
$(G,H)$-Anosov representations \cite{Burger_Iozzi_Wienhard_anosov},
where $H$ is a specific subgroup of $G$ (see
Theorem~\ref{thm:maximal_anosov}).  When $G$ is not locally isomorphic
to $\Sp(2n,\RR)$ there is no analogue of Hitchin
representations\footnote{Hitchin components can be defined for any
  $\RR$-split semisimple Lie group (see \cite[Theorem~A]{Hitchin}) and
  the only simple $\RR$-split Lie groups of Hermitian type are the
  symplectic groups.} and we conjecture:

\begin{conj_intro}\label{conj_intro:components}
  Let $G$ be a simple Lie group of Hermitian type. If $G$ is not
  locally isomorphic to $\Sp(2n,\RR)$, then the topological invariants
  of Theorem~\ref{thm_intro:invariants} distinguish connected
  components of $\repmax(\pi_1(\Sigma), G)$.
\end{conj_intro}

If the real rank of $G$ is $n$, then there is an embedding of
$\mathfrak{sl}(2,\RR)^n$ into $\mathfrak{g}$, and it is unique up to
conjugation.  Thus, there is always a corresponding diagonal embedding
of $L$, a finite cover of $\PSL(2, \RR)$, into $G$, the centralizer of
which is a compact subgroup of $G$. In particular, one can always
construct twisted diagonal representations.

\begin{conj_intro}\label{conj:standard_intro}
  Let $G$ be of Hermitian type. If $G$ not locally isomorphic to
  $\Sp(2n,\RR)$, then every maximal representation $\rho:\pi_1(\Sigma)
  \to G$ can be deformed to a twisted diagonal representation.
\end{conj_intro}

If Conjecture~\ref{conj:standard_intro} holds the analogue of
Theorem~\ref{thm_intro:holonomy} will also hold.

\subsection{Comparison with Higgs bundle invariants}
We already mentioned that the correspondence between (reductive)
representations and Higgs bundles permits to use algebro-geometric
methods to study the structure of $\rep(\pi_1(\Sigma), G)$, and in
particular to count the number of connected components. Where these
methods have been applied to study representations into Lie groups of
Hermitian type, see \cite{Bradlow_GarciaPrada_Gothen,
  Bradlow_GarciaPrada_Gothen_survey, GarciaPrada_Gothen_Mundet,
  GarciaPrada_Mundet, Gothen}, the authors associate special vector
bundles to the Higgs bundles, whose characteristic classes give
additional invariants for maximal representations $\pi_1(\Sigma) \to
G$; then they show that for any possible value of the invariants the
corresponding moduli space of Higgs bundles is non-empty and
connected.

We conclude the introduction with several remarks concerning the
relation between the topological invariants defined here and the
invariants obtained via Higgs bundles:

\begin{asparaenum}
\item The Higgs bundle approach has the feature that the $L^2$-norm of
  the Higgs field gives a Morse-Bott function on the moduli space,
  which allows to perform Morse theory on the representation
  variety. In fact, for the symplectic structure on the moduli space,
  this function is the Hamiltonian of a circle action, so that its
  critical points are exactly the fixed points of this circle
  action. These fixed points are Higgs bundles of a very special type,
  called ``variations of Hodge structures''. Additional information
  arising from this framework allows one to read off the index of the
  critical submanifolds from the eigenvalues of the circle action on
  the tangent space at a fixed point.  This is used to give an exact
  count of the connected components in many cases, as well as to
  obtain further important information about the topology of the
  representation variety.  For more details on this strategy we refer
  the reader to Hitchin's article \cite{Hitchin_selfdual} as well as
  to the series of papers \cite{Bradlow_GarciaPrada_Gothen_survey,
    GarciaPrada_Gothen_Mundet, GarciaPrada_Mundet, Gothen}.

\item The invariants defined here can be computed for explicit
  representations; this is very difficult for the Higgs bundle
  invariants. The computability is essential in order to determine in
  which connected components specific representations lie.  This is of
  particular interest for the $2g-3$ exceptional connected components
  when $n=2$, because in these connected components no explicit
  representations were known before.

\item The Higgs bundle invariants depend on the choice of a complex
  structure $X$ on $\Sigma$ and on the choice of a square-root
  $K^{1/2}$ of the canonical bundle $K= T^{1,0*}X$ (i.e.\ by
  \cite[Prop.~3.2]{Atiyah_spin} such a square root corresponds to a
  spin structure on $\Sigma$ which is given by an element $v$ in
  $\h^1(T^1\Sigma;\FF_2) \moins \h^1(\Sigma;\FF_2)$, see
  \cite[p.~55]{Atiyah_spin}). More precisely the Higgs bundle
  corresponding to a representation into $\Sp(2n, \RR)$ is a pair $(E,
  \Phi)$ of a holomorphic vector bundle $E=V \oplus V^*$ ($V$ is of
  rank $n$) and $\Phi$ a holomorphic one-form with coefficients in the
  endomorphisms of $E$ of the form \[ \Phi = \left (
    \begin{array}{cc}
      0 & b \\ c & 0
    \end{array}
  \right)\] with $b: V^* \to V \otimes K$ and $c: V \to V^* \otimes K$
  being symmetric \cite[Eq.~(2.6)]{Gothen}. For maximal representation
  $c$ is an isomorphism \cite[Prop.~3.2]{Gothen}, hence $V\otimes
  K^{-1/2}$ becomes an $\O(n,\CC)$-bundle over $X$ whose
  Stiefel-Whitney classes $w_1(\rho,v)$ in $\h^1(X,\FF_2)$ and
  $w_2(\rho,v)$ in $\h^2(X,\FF_2)$ (and sometimes Euler class
  $e(\rho,v)$) are the Higgs bundle invariants. One can remove the
  dependency in the spin structure by considering $V$ as a
  $\O(n,\CC)$-bundle over $T\Sigma \moins \{0\}$, the complement of
  the zero section. The dependency in the complex structure can also
  be removed, because the invariants take values in a discrete set and
  Teichm\"uller space is connected.

  The Stiefel-Whitney classes defined here do not depend on any
  choices.  In particular, they are equivariant under the action of
  the mapping class group of $\Sigma$ and behave naturally with
  respect to taking finite index subgroups of $\pi_1(\Sigma)$. They
  are also natural with respect to tight homomorphisms.
\end{asparaenum}

For symplectic maximal representations there is a simple relation
between the invariants defined here and those defined using Higgs
bundles, although they live naturally in different cohomology groups.
\begin{prop_intro}\label{prop_intro:relations}
  Let $\rho: \pi_1(\Sigma) \to \Sp(2n,\RR)$ be a maximal
  representation. Then, for any choice of spin structure $v$, we have
  the following equality in $\h^i(T^1\Sigma; \FF_2)$:
  \begin{align*}
    sw_1(\rho) &= w_1(\rho,v) + n v\\
    sw_2(\rho) &= w_2(\rho,v) + sw_1(\rho) \cupprod v + (g-1)\mod 2.
  \end{align*}
\end{prop_intro}
When $n=2$, the suspected relation is
\[
e(\rho,v) = \varepsilon e(\rho) + (g-1)
\]
in $ \h^2( T^1 \Sigma; \ZZ)^{tor}$, where $\varepsilon$ depends on the
choices of orientation involved in the definition of $e(\rho)$.

The existence of such relations is not surprising since the invariants
arise from the same compact Lie group ($\O(n,\CC)$-bundles as well as
$\GL(n,\RR)$-bundles have, up to isomorphism, a unique $\O(n,
\RR)$-reduction, and the invariants are in fact the invariants of the
underlying $\O(n,\RR)$-bundle).  Nevertheless, it would be interesting
to provide a general proof for these relations for all maximal
representations. The relations in
Proposition~\ref{prop_intro:relations} are obtained from case by case
considerations for model representations.

\subsection{Structure of the paper}
In Section~\ref{sec:prelim} we recall the definition and properties of
Anosov representations and of maximal representations. Examples of
such representations are discussed in Section~\ref{sec:examples}. The
topological invariants are defined in Section~\ref{sec:invariants} and
computed for symplectic maximal representations in
Section~\ref{sec:computations}. Section~\ref{sec:mapgroup} discusses
the action of the mapping class group; in
Section~\ref{sec:HoloOfMaxRep} we derive consequences for the holonomy
of symplectic maximal representations.  In Appendix~\ref{sec:app_max}
we establish several important facts about positive curves and maximal
representations; in Appendix~\ref{sec:cohomology} we review some facts
in cohomology, in particular for the unit tangent bundle $T^1\Sigma$.

{\bf Acknowledgments.} The discussions leading to this paper started
during a workshop on ``Surface group representations'' taking place at
the American Institute for Mathematics in 2007. The authors thank the
University of Chicago, the Institut des Hautes \'Etudes Scientifiques
and Princeton University for making mutual intensive research visits
possible.

\section{Preliminaries}\label{sec:prelim}

This section introduces the notations and definitions necessary for
the rest of the paper.  Section~\ref{sec_anosov} gives the definition
of Anosov representations and their basic properties, in particular
the uniqueness of the \emph{Anosov section}
(Proposition~\ref{prop:uniqsigma}). Section~\ref{sec:prelim_maximal}
recalls the needed theory of maximal representations, especially that
they are Anosov (Theorem~\ref{thm:maximal_anosov}) and the gluing
property (Theorem~\ref{thm:gluing_maximal}).

\subsection{Anosov representations}
\label{sec_anosov}

Holonomy representations of locally homogeneous geometric structures
are very special.  The Thurston-Ehresmann theorem
\cite[p.~178]{Goldman_geometric} \cite[Theorem~2.1]{Bergeron_Gelander}
states that every deformation of such a representation can be realized
through deformations of the geometric structure.  The concept of
\emph{Anosov structures}, introduced by Labourie in
\cite[Section~2]{Labourie_anosov}, gives a dynamical generalization of
this, which is more flexible, but for which it is still possible to
obtain enough rigidity of the associated representations.

\subsubsection{Definition}
Let
\begin{itemize}
\item $M$ be a compact manifold with an Anosov flow $\phi_t$,
\item $G$ a connected semisimple Lie group and $(P^s, P^u)$ a pair of
  opposite parabolic subgroups of $G$,
\item $H = P^s \cap P^u$ their intersection, and
\item $\mathcal{F}^s = G/ P^s$ (resp. $\mathcal{F}^u = G/ P^u$) the
  flag variety associated with $P^s$ (resp. $P^u$).
\end{itemize}
There is a unique open $G$-orbit $\mathcal{X} \subset \mathcal{F}^s
\times \mathcal{F}^u$. We have $\mathcal{X} = G/H$ and as open subset
of $\mathcal{F}^s \times \mathcal{F}^u$ it inherits two foliations $
\mathcal{E}^s$ and $\mathcal{E}^u$ whose corresponding distributions
are denoted by $E^s$ and $E^u$, i.e.\ $(E^s)_{(f^s,f^u)} \cong
T_{f^{s}} \Ff^s$ and $(E^u)_{(f^s,f^u)} \cong T_{f^u} \Ff^u$.

\begin{defi}
  A flat $G$-bundle $\mathsf{P}$ over $(M, \phi_t)$ is said to have an
  $H$-reduction $\sigma$ that is \emph{flat along flow lines} if:
  \begin{itemize}
  \item $\sigma$ is a section of $\mathsf{P} \times_G \mathcal{X}$;
    i.e.\ $\sigma : M \to \mathsf{P} \times_G \mathcal{X}$ defines the
    $H$-reduction\footnote{For details on the bijective correspondence
      between $H$-reductions and sections of ${\mathsf{P} \times_G
        G/H}$ we refer the reader to \cite[Section 9.4]{Steenrod}.}
  \item the restriction of $\sigma$ to every orbit of $\phi_t$ is
    locally constant with respect to the induced flat structure on
    $\mathsf{P} \times_G \mathcal{X}$.
  \end{itemize}
\end{defi}

The two distributions $E^s$ and $E^u$ on $\mathcal{X}$ are
$G$-invariant and hence define distributions, again denoted $E^s$ and
$E^u$, on $\mathsf{P}\times_G \mathcal{X}$. These two distributions
are invariant by the flow, again denoted by $\phi_t$, that is the lift
of the flow on $M$ by the connection.

To every section $\sigma$ of $\mathsf{P} \times_G \mathcal{X}$ we
consider the two vector bundles $\sigma^* E^s$ and $\sigma^* E^u$ on
$M$ by pulling back to $M$ the vector bundles $E^s$ and $E^u$. If
furthermore $\sigma$ is flat along flow lines, so that it commutes
with the flow, then these two vector bundles $\sigma^* E^s$ and
$\sigma^* E^u$ are equipped with a natural flow.

\begin{defi}
  \label{defi:anosov}
  A flat $G$-bundle $\mathsf{P} \to M$ is said to be a
  \emph{($G,H)$-Anosov bundle} if:
  \begin{enumerate}
  \item \label{anosovpoint1} $\mathsf{P}$ admits an $H$-reduction
    $\sigma$ that is flat along flow lines, and
  \item \label{anosovpoint2} the flow $\phi_t$ on $\sigma^* E^s$
    (resp. $\sigma^* E^u$) is contracting (resp. dilating).
  \end{enumerate}
  We call $\sigma$ an {\em Anosov section} or an {\em Anosov
    reduction} of $\mathsf{P}$.
\end{defi}
By (\ref{anosovpoint2}) we mean that there exists a continuous family
of norms $(\| \cdot \|_m)_{m \in M} $ on $\sigma^* E^s$
(resp. $\sigma^* E^u$) and constants $A,a >0$ such that for any $e$ in
$(\sigma^* E^s)_m$ (resp. $(\sigma^* E^u)_m$) and for any $t>0$ one
has
\[
\| \phi_t e \|_{\phi_t m} \leq A \exp(-at) \| e \|_m \quad
\text{(resp. } \| \phi_{-t} e \|_{\phi_{-t} m} \leq A \exp(-at) \| e
\|_m \text{)}.
\]
Since $M$ is compact this definition does not depend on the norm $\|
\cdot \|$ or the parametrization of $\phi_t$.

\begin{defi}
  \label{defi:repano}
  A representation $\pi_1( M) \to G$ is said to be
  ($G,H$)\emph{-Anosov} (or $H$\emph{-Anosov} or simply \emph{Anosov})
  if the corresponding flat $G$-bundle $\mathsf{P}$ is a
  $(G,H)$-Anosov bundle.
\end{defi}

\begin{remark}\label{rem:parabolic}
  Note that the terminology for Anosov representations is not
  completely uniform. A $(G,H)$-Anosov representation is sometimes
  called a $(G,\mathcal{X})$-Anosov representation or an Anosov
  representation with respect to the parabolic subgroup $P^s$ or
  $P^u$.
\end{remark}

\subsubsection{Properties}
From now on we assume that the Anosov flow $\phi_t$ has an invariant
volume form.  For the definition of topological invariants of Anosov
representations in Section~\ref{sec:invariants} the following
proposition will be crucial.

\begin{prop}
  \label{prop:uniqsigma}
  Let $ \mathsf{P} \to M$ be an Anosov bundle, then there is a
  \emph{unique} section $\sigma : M \to \mathsf{P} \times_G
  \mathcal{X}$ such that properties (\ref{anosovpoint1}) and
  (\ref{anosovpoint2}) of Definition~\ref{defi:anosov} hold.  In
  particular, a $(G,H)$-Anosov bundle admits a \emph{canonical}
  $H$-reduction.
\end{prop}

To prove Proposition~\ref{prop:uniqsigma} we will use the following
classical fact.

\begin{fact}
  \label{fact:attracL}
  Suppose that $g\in G$ has a fixed point $f^s\in \Ff^s$ and a fixed
  point $f^u \in \Ff^u$ such that
  \begin{enumerate}
  \item $f^s$ and $f^u$ are in general position, i.e.\ $(f^s, f^u)$
    belongs to $\mathcal{X} \subset \Ff^s\times \Ff^u$; and
  \item the (linear) action of $g$ on the tangent space $T_{f^s}
    \Ff^s$ is contracting, the action on $T_{f^u} \Ff^u$ is expanding.
  \end{enumerate}
  Then $f^s$ is the only attracting fixed point of $g$ in $\Ff^s$. Its
  attracting set is the set of all flags that are in general position
  with $f^u$, and $f^u$ is the only repelling fixed point for $g$ in
  $\Ff^u$.
\end{fact}

\begin{proof}[Proof of Proposition~\ref{prop:uniqsigma}]
  Let $\sigma$ be an Anosov section of the flat bundle $\mathsf{P} \to
  M$.  The flow $\phi_t$ on $M$ is (by hypothesis) an Anosov flow that
  has an invariant volume form. Because of the density of closed
  orbits in $M$ \cite[Theorem~3]{Anosov}, it is sufficient to show
  that the restriction of $\sigma$ to any closed orbit $\g$ is
  uniquely determined.

  We identify $\gamma$ with $\ZZ \backslash \RR$ and write the
  restriction of $\mathsf{P}$ to $\g$ as $\mathsf{P|}_{\g} = \ZZ
  \backslash (\RR \times G)$ where $\ZZ$ acts on $\RR \times G$ by $n
  \cdot (t,g) = (n +t, h_\gamma^n g)$ for $h_\gamma \in G$ the
  holonomy of $\mathsf{P}$ along $\g$.
  
  Therefore, the restriction $\sigma|_{\g}$ is a section of $\ZZ
  \backslash (\RR \times \mathcal{X})$. Since $\sigma$ is locally
  constant along $\g$ there exists $x = (f^s, f^u)$ in $\mathcal{X}$
  such that the lift of $\sigma|_{\g}$ to $\RR$ is the map $\RR \to
  \RR \times \mathcal{X}\sep{} t \mapsto (t,x)$. Hence $f^s$ and $f^u$
  are $h_\gamma$-invariant. Furthermore the restrictions of $\sigma^*
  E^s$ (resp. $\sigma^* E^u$) to $\g$ are in this case $ \ZZ
  \backslash (\RR \times T_{f^s} \Ff^s)$ (resp. $ \ZZ \backslash (\RR
  \times T_{f^u} \Ff^u)$) so that the contraction property of
  Definition~\ref{defi:anosov}.\eqref{anosovpoint2} exactly translates
  into the assumption of Fact~\ref{fact:attracL}. This implies the
  uniqueness of $(f^s, f^u)$ and hence the uniqueness of
  $\sigma|_{\g}$.
\end{proof}

\begin{remark}
  Without the assumption that the Anosov flow has an invariant volume
  form, it is not always true that the periodic orbits are dense in
  $M$ (see \cite{Franks_Williams_Anomalous}). So
  Proposition~\ref{prop:uniqsigma} may not hold in general.
\end{remark}

From this uniqueness one deduces

\begin{prop}\label{prop:coverM}
  Let $\pi: M' \to M$ be a finite covering of $M$ and let $\mathsf{P}
  \to M$ be a flat $G$-bundle over $M$.

  Then the bundle $\pi^* \mathsf{P}$ over $M'$ is $H$-Anosov if and
  only if $\mathsf{P}$ is $H$-Anosov. In that case the Anosov
  $H$-reduction of $\pi^* \mathsf{P}$ is $\pi^* \mathsf{P}_H$ the
  pullback of the Anosov reduction of $\mathsf{P}$.
\end{prop}

\begin{proof}
  If $\mathsf{P}$ is Anosov, then clearly $\pi^* \mathsf{P}$ is Anosov
  with reduction $\pi^* \mathsf{P}_H$.

  Suppose now $\pi^* \mathsf{P}$ is Anosov with reduction
  $\mathsf{P}_{H}^{\prime}$. We must prove that
  $\mathsf{P}_{H}^{\prime}$ is the pullback of a bundle over $M$.  It
  is enough to treat the case when $M'$ is connected. Since $\pi_1(M)$
  is finitely generated, the finite index subgroup $\pi_1(M')$ is
  contained in a finite index normal subgroup. Thus, up to taking a
  finite cover, we can suppose that $\pi_1(M') < \pi_1(M)$ is normal,
  i.e.\ that the finite covering $M'\to M$ is a Galois covering with
  group $S$.

  For any $\sigma$ in $S$ the two bundles $\mathsf{P}_{H}^{\prime}$
  and $\sigma^* \mathsf{P}_{H}^{\prime}$ are Anosov reductions of
  $\pi^* \mathsf{P}$. Hence by the previous
  proposition~\ref{prop:uniqsigma}, $\sigma^* \mathsf{P}_{H}^{\prime}
  = \mathsf{P}_{H}^{\prime}$; the bundle $\mathsf{P}_{H}^{\prime}$ is
  $S$-invariant. This means that $\mathsf{P}_{H}^{\prime}$ is the
  pullback $\pi^* \mathsf{P}_{H}$ of a bundle over $M$.
\end{proof}

An important feature of Anosov representations is their stability
under deformation.
\begin{prop}
  \cite[Proposition~2.1]{Labourie_anosov}
  \label{prop:openano}
  The set of (G,H)-Anosov representations is \emph{open} in $\hom (
  \pi_1 (M), G)$. Moreover, the $H$-reduction given by the Anosov
  section $\sigma$ depends continuously on the representation.
\end{prop}

\subsubsection{Constructions}
Some simple constructions allow to obtain new Anosov bundles from old
ones.
\begin{lem}
  \label{lem_anosov_constructions}
  \begin{asparaenum}
  \item\label{item1:lemAc} Let $\rho: \pi_1(M) \to G$ be a
    $(G,H)$-Anosov representation, where $H = P^s \cap P^u$ for two
    opposite parabolic subgroups in $G$. Let $Q^s$, $Q^u$ be opposite
    parabolic subgroups in $G$ such that $P^s <Q^s$ and $P^u <
    Q^u$. Then $\rho$ is also a $(G,H')$-Anosov representation, where
    $H' = Q^s \cap Q^u$.
  \item\label{item2:lemAc} Let $\mathsf{P}$ be a $(G,H)$-Anosov bundle
    over $M$ with canonical $H$-reduction $\mathsf{P}_H$ and $E$ a
    flat $L$-bundle over $M$. Then the fibered product
    $\mathsf{P}\times E$ is a $(G\times L, H\times L)$-Anosov bundle
    over $M$, whose canonical $(H\times L)$-reduction is the fibered
    product $\mathsf{P}_H \times E$.
  \item\label{item3:lemAc} Let $G$ be a Lie group of real rank one and
    $\mathsf{P}$ an $(G,H)$-Anosov bundle over $M$ with canonical
    $H$-reduction $\mathsf{P}_H$; $H$ is the centralizer in $G$ of a
    Cartan subgroup $A$ of $G$. Let $f: G \to L$ be a homomorphism of
    Lie groups and $N$ be the centralizer in $L$ of $f(A)$.

    Then $\mathsf{P}\times_G L$ is an $(L,N)$-Anosov bundle and its
    canonical $N$-reduction is the fibered product $\mathsf{P}_H
    \times_H N$.
  \end{asparaenum}
\end{lem}

Points (\ref{item1:lemAc}) and (\ref{item2:lemAc}) of this lemma are
immediate; point (\ref{item3:lemAc}) results from \cite[Proposition
3.1]{Labourie_anosov}. For a more general version concerning Anosov
representations under embeddings of Lie groups $G\to L$ we refer to
\cite{Guichard_Wienhard_DoD}.  Anosov representations are also stable
under finite cover of Lie groups:

\begin{lem}
  \label{lem:finite_cover}
  Let $\pi:G' \to G$ be a finite covering of Lie groups. Let $H = P^s
  \cap P^u$ be the intersection of two opposite parabolic subgroups of
  $G$ and let $H' = \pi^{-1}(H)$.

  Then a representation $\rho: \pi_1(M) \to G'$ is $(G',H')$-Anosov if
  and only if $\pi \circ \rho$ is $(G,H)$-Anosov. Furthermore is
  $\mathsf{P}'$ and $\mathsf{P}$ are the flat principal bundles for
  $\rho$ and $\pi \circ \rho$, then the two Anosov reductions are
  equal under the identification $\mathsf{P}' \times_{G'} G'/H' =
  \mathsf{P}' /H' \cong \mathsf{P}/H = \mathsf{P} \times_{G} G/H$.
\end{lem}

This lemma is obvious since $G'/H' \cong G/H$.

\subsubsection{Definition in terms of the universal cover of M}
\label{sec:terms-univ-cover}
A $(G,H)$-Anosov bundle over $M$ can be equivalently defined in terms
of equivariant maps from the universal cover of $M$ to
$\mathcal{X}\cong G/H$.  Let $\widetilde{M}$ be the universal cover of
$M$.  Then any flat $G$-bundle $\mathsf{P}$ on $M$ can be written as:
\[ \mathsf{P} = \pi_1(M) \backslash (\widetilde{M} \times G) , \quad
\g \cdot ( \tilde{m}, g) = ( \g \cdot \tilde{m}, \rho( \g) g)\] for
some representation $\rho: \pi_1(M) \to G$.

Let $\phi_t$ be the lift of the flow on $M$ to $\widetilde{M}$. This
flow lifts to $\phi_t(\tilde{m}, g) = (\phi_t(\tilde{m}), g)$,
defining a flow on $\widetilde{M} \times G$.

An $H$-reduction $\sigma$ is the same as a $\rho$-equivariant map
\[ \tilde { \sigma} : \widetilde{M} \longrightarrow G/H \cong
\mathcal{X}.\] The section $\sigma$ is flat along flow lines if and
only if the map $ \tilde { \sigma}$ is $\phi_t$-invariant.

The contraction property of the flow is now expressed as follows:
\begin{enumerate}
\item There exists a continuous family $( \| \cdot \|_{\tilde{m}} )_{
    \tilde{m} \in \widetilde{M}} $ such that
  \begin{itemize}
  \item for all $\tilde{m}$, $\| \cdot \|_{\tilde{m}}$ is a norm on $(
    E^s)_{\tilde{ \sigma} (\tilde{m})} \subset T_{ \tilde{
        \sigma}(\tilde{m})} \mathcal{X}$,
  \item and $( \| \cdot \|_{\tilde{m}} )_{ \tilde{m} \in
      \widetilde{M}} $ is $\rho$-equivariant, i.e.\ for all
    $\tilde{m}$ in $\widetilde{M}$, $\g$ in $\pi_1( M)$ and $e$ in $(
    E^s)_{\tilde{ \sigma} (\tilde{m})}$ one has $ \| \rho( \g) \cdot e
    \|_{ \g \cdot \tilde{m} } = \| e \|_{\tilde{m}}$.
  \end{itemize}
\item The flow $\phi_t$ is contracting, i.e.\ there exist $A,a >0$
  such that for any $t>0$ and $\tilde{m}$ in $\widetilde{M}$ and $e$
  in $( E^s)_{\tilde{ \sigma} (\tilde{m})}$ then \[\| e \|_{\phi_t
    \cdot \tilde{m}} \leq A \exp(-at) \| e \|_{\tilde{m}}.\]
    
  (Note that this makes sense because $\tilde{ \sigma} (\phi_t \cdot
  \tilde{m}) = \tilde{ \sigma} (\tilde{m})$, and thus $e$ belongs to
  $( E^s)_{\tilde{ \sigma} (\phi_t \cdot \tilde{m})}$).
\end{enumerate}

\subsubsection{Specialization to $T^1\Sigma$}
\label{sec:spec-t1sigma}\label{sec:afewconseq}

We restrict now to the case when $M = T^1 \Sigma$ is the unit tangent
bundle of a closed oriented connected surface $\Sigma$ of negative
Euler characteristic and $\phi_t$ is the geodesic flow on $T^1 \Sigma$
with respect to some hyperbolic metric on $\Sigma$.

Let $\partial \pi_1(\Sigma)$ be the boundary at infinity of
$\pi_1(\Sigma)$. Then $\partial \pi_1(\Sigma)$ is a topological circle
that comes with a natural orientation and an action of
$\pi_1(\Sigma)$.

There is an equivariant identification
\[ T^1 \widetilde{ \Sigma} \cong \partial \pi_1(\Sigma)^{(3+)}\] of
the unit tangent bundle of $\widetilde{\Sigma}$ with the set of
positively oriented triples in $\partial \pi_1(\Sigma)$. The orbit of
the geodesic flow through the point $(t^s, t, t^u)$ is

\[ \mathcal{G}_{(t^s, t, t^u)} = \mathcal{G}_{(t^s, t^u)} = \{ (r^s,
r, r^u ) \in \partial \pi_1(\Sigma)^{(3+)} \mid \ r^s = t^s, r^u=t^u
\},\] and the set of geodesic leaves is parametrized by $\partial
\pi_1(\Sigma)^{(2)} = \partial \pi_1(\Sigma)^2 \moins \Delta$, the
complement of the diagonal in $\partial \pi_1(\Sigma)^2$. (For more
details we refer the reader to
\cite[Section~1.1]{Guichard_Wienhard_Duke}).

Let $\pi: T^1\Sigma \to \Sigma$ be the natural projection.
\begin{defi}
  A flat $G$-bundle $\mathsf{P}$ over $\Sigma$ is said to be
  \emph{Anosov} if its pullback $\pi^* \mathsf{P}$ over $M = T^1
  \Sigma$ is Anosov.

  A representation $\rho : \pi_1( \Sigma) \to G$ is \emph{Anosov} if
  the composition $ \pi_1( T^1\Sigma) \to \pi_1(\Sigma) \xrightarrow{
    \rho} G$ is an Anosov representation.
\end{defi}

\begin{remark} 
  Since the (unparametrized) geodesic flow can be described
  topologically in terms of triples of points in $\partial
  \pi_1(\Sigma)$, the definition of Anosov representations \emph{does
    not} depend on the choice of the hyperbolic metric on $\Sigma$.
\end{remark}

\begin{remark}
  Note that for a $(G,H)$-Anosov bundle $\mathsf{P}$ over $\Sigma$,
  the pullback $\pi^* \mathsf{P}$ over $T^1\Sigma$ admits a canonical
  $H$-reduction, but this $H$-reduction in general does not come from
  a reduction of $\mathsf{P}$.  Indeed, the invariants defined in
  Section~\ref{sec:invariants} also give obstructions for this to
  happen.
\end{remark}

Let $\rho: \pi_1(\Sigma) \to G$ be an Anosov representation and
$\mathsf{P}$ the corresponding flat $G$-bundle over $\Sigma$. Using
the description in Section~\ref{sec:terms-univ-cover} it follows that
there exists a $\rho$-equivariant map:
\[ \tilde{ \sigma} : T^1 \widetilde{ \Sigma} \longrightarrow
\mathcal{X} \subset \Ff^s \times \Ff^u\] which is invariant by the
geodesic flow.

In particular we get a $\rho$-equivariant map
\[
(\xi^s, \xi^u): \partial \pi_1(\Sigma)^{(2)} \longrightarrow \Ff^s
\times \Ff^u.\]

In view of the contraction property of $\tilde{ \sigma}$
(Definition~\ref{defi:anosov}.(\ref{anosovpoint2})) it is easy to see
that $\xi^s( t^s, t^u)$ (resp. $\xi^u( t^s, t^u)$) depends only of
$t^s$ (resp. $t^u$). Hence, we obtain $\rho$-equivariant maps
$\xi^s: \partial \pi_1(\Sigma) \to \Ff^s$, and $\xi^u: \partial
\pi_1(\Sigma) \to \Ff^u$.

\begin{cor}\label{cor:anosov_holonomy}
  Let $\rho: \pi_1(\Sigma) \to G$ be a $(G,H)$-Anosov
  representation. Then for every $\gamma \in \pi_1(\Sigma)\moins\{1\}$
  the image $\rho(\gamma)$ is conjugate to an element in $H$, having a
  unique pair of attracting/repelling fixed points
  $(\xi^s(t_{\gamma}^{s}), \xi^u(t_{\gamma}^{u})) \in \Ff^s \times
  \Ff^u$, where $(t_{\gamma}^{s}, t_{\gamma}^{u})$ denotes the pair of
  attracting/repelling fixed points of $\gamma$ in $\partial \pi_1(
  \Sigma)$.
\end{cor}

\begin{remark}\label{rem_PconjtoPopp}
  In the situation when $P^s$ is conjugate to $P^u$, there is a
  natural identification between $\Ff^s$ and $\Ff^u$, and
  Proposition~\ref{prop:uniqsigma} implies the equality $\xi^s =
  \xi^u$.  In this case we denote the equivariant map simply by
  $\xi: \partial \pi_1(\Sigma) \to \Ff$.
\end{remark}

\subsection{Maximal representations}\label{sec:prelim_maximal}
\subsubsection{Definition and properties}
Let $G$ be an almost simple noncompact Lie group of Hermitian type,
i.e.\ the symmetric space associated with $G$ is an irreducible
Hermitian symmetric space of noncompact type.  Then $\pi_1(G)$ is a
finite extension of $\ZZ$: $\pi_1(G) / \pi_1(G)^{tor} \cong \ZZ$, and
the obstruction class of $\rho$ (see Introduction) projects to a
characteristic class $\tau(\rho) \in \h^2( \Sigma; \pi_1 (G) /
\pi_1(G)^{tor} ) \cong \ZZ$, called the \emph{Toledo invariant} of the
representation $\rho$.  The Toledo invariant $\tau(\rho)$ is bounded
in absolute value\[ |\tau(\rho)| \leq -C(G)\chi(\Sigma),
\] where $C (G)$ is an explicit constant depending only on $G$.

\begin{defi}\label{defi:maximal}
  A representation $\rho:\pi_1(\Sigma) \to G$ is \emph{maximal} if
  \[
  \tau(\rho) = -C(G) \chi(\Sigma).
  \]
  The space of maximal representation is denoted by $\hommax(
  \pi_1(\Sigma),G)$.  It is a union of connected components of $\hom(
  \pi_1(\Sigma),G)$.
\end{defi}

\begin{remark}
  In the definition of maximal representations we choose the positive
  extremal value of the Toledo invariant.  Yet, the Toledo number
  $\tau(\rho) \in \ZZ$ depends on the identification $\ZZ \cong \h^2(
  \Sigma; \ZZ)$ hence on the orientation of the surface. Thus
  reversing the orientation changes the Toledo number to its
  opposite. Therefore the space where the negative extremal value is
  achieved is isomorphic to the space of maximal representations.
\end{remark}

Maximal representations have been extensively studied in recent years
\cite{Bradlow_GarciaPrada_Gothen,Bradlow_GarciaPrada_Gothen_survey,
  Burger_Iozzi_Labourie_Wienhard, Burger_Iozzi_Wienhard_ann,
  Burger_Iozzi_Wienhard_toledo, Hernandez, Toledo_89,
  Wienhard_mapping}.  They enjoy several interesting properties, e.g.\
maximal representations are discrete embeddings, but more importantly
for our considerations is the following

\begin{thm}\cite{Burger_Iozzi_Labourie_Wienhard,Burger_Iozzi_Wienhard_anosov}\label{thm:maximal_anosov}
  A maximal representation $\rho: \pi_1(\Sigma) \to G$ is an Anosov
  representation. More precisely, $\rho$ is a $(G,H)$-Anosov
  representation, where $H<G$ is the stabilizer of a pair of
  transverse points in the Shilov boundary of the symmetric space
  associated with $G$.
\end{thm}

The Toledo invariant can be defined for surface with boundary
\cite[Section~1.1]{Burger_Iozzi_Wienhard_toledo}, extending the notion
of maximal representation also to this case.  In our construction of
maximal representations we make use of the following gluing theorem,
which follows from the additivity properties of the Toledo invariant
established in \cite[Prop.~3.2]{Burger_Iozzi_Wienhard_toledo}.

\begin{thm}\cite[Th.~1, Def.~2]{Burger_Iozzi_Wienhard_toledo}\label{thm:gluing_maximal}
  Let $\Sigma = \Sigma_1 \cup_{\gamma} \Sigma_2$ be the decomposition
  of $\Sigma$ along a simple closed curve and $\pi_1(\Sigma) =
  \pi_1(\Sigma_1) *_{\langle \gamma \rangle} \pi_1(\Sigma_2)$ the
  corresponding decomposition as amalgamated product. Let $\rho_i:
  \pi_1(\Sigma_i) \to G$ be representations which agree on $\gamma$,
  and let $ \rho = \rho_1*\rho_2: \pi_1(\Sigma) \to G$ the amalgamated
  representation.

  If $\rho_1$ and $\rho_2$ are maximal, then $\rho$ is maximal.

  Conversely, if $\rho$ is maximal, then $\rho_1$ and $\rho_2$ are
  maximal.
\end{thm}

\subsubsection{Maximal representations into
  $\Sp(2n,\RR)$}\label{sec:prelim_maxrep_sp}

Let $\RR^{2n}$ be a symplectic vector space and $ ( e_i)_{1 \leq i
  \leq 2n}$ a symplectic basis, with respect to which the symplectic
form $\omega$ is given by the anti-symmetric matrix:
\[ J= \left(
  \begin{array}{cc}
    0 & \id_n \\ - \id_n & 0
  \end{array}\right). \]
Let $G = \Sp(2n,\RR)$.

Let $L_0^s := \Span ( e_i)_{1 \leq i \leq n}$ be a Lagrangian subspace
and $P^s<\Sp(2n,\RR)$ be the parabolic subgroup stabilizing
$L_0^s$. The stabilizer of the Lagrangian $L_0^u = \Span( e_i)_{ n< i
  \leq 2n}$ is a parabolic subgroup $P^u$, it is opposite to $P^s$.
The subgroup $H = P^s \cap P^u$ is isomorphic to $\GL(n, \RR)$:
\[
\begin{array}{rcl}
  \GL(n, \RR) & \overset{ \sim}{ \longrightarrow} & H \subset G \\
  A & \longmapsto & \left(
    \begin{array}{cc}
      A & 0 \\ 0 & \transpose{}\! A^{-1}
    \end{array}\right).
\end{array}
\]
(Such isomorphisms are in one to one correspondence with symplectic
bases $( \epsilon_i)_{ 1 \leq i \leq 2n }$ ---i.e.\ $\omega(
\epsilon_i, \epsilon_j) = \omega( e_i, e_j)$ for all $i,j$--- for
which $( \epsilon_i)_{ 1 \leq i \leq n}$ is a basis of $L_0^s$).  Note
that $P^u$ is conjugate to $P^s$ in $G$, so that the flag variety
$\Ff^s$ is canonically isomorphic to $\Ff^u$; we will denote this
homogeneous space by $\Ll$.  The space $\Ll$ is the Shilov boundary of
the symmetric space associated with $\Sp(2n,\RR)$; it can be realized
as the space of Lagrangian subspaces in $\RR^{2n}$ and the homogeneous
space $\mathcal{X} \subset \Ll \times \Ll$ is the space of pairs of
transverse Lagrangians.

Let $\rho:\pi_1(\Sigma) \to \Sp(2n,\RR)$ be a maximal representation
and $\mathsf{P}$ the corresponding flat principal $\Sp(2n,\RR)$-bundle
over $T^1\Sigma$ and $E$ the corresponding flat symplectic
$\RR^{2n}$-bundle over $T^1\Sigma$. Then $\rho$ is an $(\Sp(2n,\RR),
\GL(n,\RR))$-Anosov representation
(Theorem~\ref{thm:maximal_anosov}). The canonical
$\GL(n,\RR)$-reduction of $\mathsf{P}$ is equivalent to a continuous
splitting of $E$ into two (non-flat) flow-invariant transverse
Lagrangian subbundles
\[
E = L^s(\rho) \oplus L^u(\rho).
\]
\begin{nota}\label{defi:Lagrangianreduction}
  We call this splitting the {\em Lagrangian reduction} of the flat
  symplectic Anosov $\RR^{2n}$-bundle.
\end{nota}

\begin{remark}
  The symplectic form on $L^s(\rho) \oplus L^u(\rho)$ induces a
  duality $L^u(\rho) \cong L^s(\rho)^*$. As a result we will only
  consider $L^s(\rho)$ when defining invariants for symplectic maximal
  representations.
\end{remark}

The Lagrangian reduction can be directly constructed from the
continuous $\rho$-equivariant curve $\xi:\partial\pi_1(\Sigma) \to
\Ll$: for any triple $v=(t^s, t, t^u) \in \partial
\pi_1(\Sigma)^{(3+)} \cong T^1\widetilde{\Sigma}$, we set
$(L^s(\rho))_v = \xi(t^s)$ and $(L^u(\rho))_v = \xi(t^u)$.  This curve
satisfies an additional positivity property which we now describe.

Let $(L^s, L, L^u)$ be a triple of pairwise transverse Lagrangians in
$\RR^{2n}$; then $L$ can be realized as the graph of $F_L \in
\hom(L^s, L^u)$.
\begin{defi}\label{def:positivetriple}
  A triple $(L^s, L, L^u)$ of pairwise transverse Lagrangians is
  called {\em positive} if the quadratic form $\omega(\cdot,
  F_L(\cdot))$ on $L^s$ is positive definite.
\end{defi}
\begin{defi}
  \label{def:positivecurve}
  A curve $\xi: \partial \pi_1(\Sigma) \to \Ll$ from the boundary at
  infinity of $\pi_1(\Sigma)$ to the space of Lagrangians is said to
  be \emph{positive}, denoted by $\xi>0$, if for every positively
  oriented triple $( t^s, t, t^u)$ in $\partial\pi_1(\Sigma)^{(3+)}$
  the triple of Lagrangians $( \xi(t^s), \xi(t), \xi(t^u))$ is
  positive.
\end{defi}

Important facts about the space of positive curves in $\Ll$ are
established in Appendix~\ref{sec:app_positive}.

\begin{thm}\cite[Theorem~8]{Burger_Iozzi_Wienhard_toledo}
  \label{thm:positivecurve}
  Let $\rho \in \hom( \pi_1(\Sigma),\Sp(2n,\RR))$ be a maximal
  representation and $\xi : \partial \pi_1(\Sigma) \to \Ll$ the
  equivariant limit curve.  Then $\xi$ is a positive curve.
\end{thm}

As a consequence of Proposition~\ref{prop:openano} we have
\begin{fact}\label{fact:curve}
  The positive limit curve $\xi : \partial \pi_1(\Sigma) \to \Ll$
  depends continuously on the representation.
\end{fact}

In the following we will often switch between the three different
viewpoints: $\GL(n,\RR)$-reduction of $\mathsf{P}$, splitting $E =
L^s(\rho) \oplus L^u(\rho)$ or (positive) equivariant curve $\xi
: \partial \pi_1(\Sigma) \to \Ll$.

\section{Examples of representations}\label{sec:examples}

In this section we provide examples of Anosov and maximal
representations. In Section~\ref{sec_anov_rep} we describe examples of
Anosov representations. In Section~\ref{sec_standard} we construct
maximal representations into $\Sp(2n, \RR)$ arising from embeddings of
subgroups. In Section~\ref{sec:amalagamatedRep} we construct other
examples of maximal representations based on
Theorem~\ref{thm:gluing_maximal}.  The maximal representations
constructed in this section will be considered throughout the paper.

\subsection{Anosov representations}
\label{sec_anov_rep}
We give examples of Anosov representations. By
Proposition~\ref{prop:openano} every small deformation of one of these
representations is again an Anosov representation.

\subsubsection{Hyperbolizations} 
Let $\Sigma$ be a connected oriented closed hyperbolic surface and $M=
T^1\Sigma$ its unit tangent bundle equipped with the geodesic flow
$\phi_t$.  Hyperbolizations give rise to discrete embeddings
$\pi_1(\Sigma) \to \PSL(2,\RR)$ which are examples of Anosov
representations (see \cite[Proposition 3.1]{Labourie_anosov}).  More
generally, a discrete embedding of $\pi_1(\Sigma)$ into any finite
cover $L$ of $\PSL(2,\RR)$ is an Anosov representation.  Since
$\PSL(2,\RR)$ has rank one there is no choice for the parabolic
subgroup.

Later we will be interested in particular in Anosov bundles arising
from discrete embeddings $\iota: \pi_1(\Sigma) \to \SL(2,\RR)$. In
that case $H \cong \GL(1,\RR)$ is the subgroup of diagonal matrices
and the $H$-reduction corresponds to a splitting of the flat
$\RR^2$-bundle over $T^1 \Sigma$ into two line bundles $L^s(\iota)
\oplus L^u(\iota)$.

\subsubsection{Hitchin representations}
A representation of $\pi_1(\Sigma)$ into a split real semisimple Lie
group $G$ is said to be a {\em Hitchin representation} if it can be
deformed into a representation $\pi_1(\Sigma) \xrightarrow{\iota} L
\xrightarrow{\tau} G$, where $L$ is a finite cover of $\PSL(2, \RR)$
and $\tau_*: \mathfrak{sl}(2, \RR) \to \mathfrak{g}$ is the principal
$\mathfrak{sl}(2, \RR)$ (see \cite[Section~4]{Hitchin} for more
details), and where the homomorphism $\iota: \pi_1(\Sigma) \to L$ is a
discrete embedding.  For the special case when $G= \SL(n,\RR),$
$\Sp(2m,\RR)$ or $\SO(m,m+1)$, the embedding $\tau: \SL(2,\RR) \to G$
is the $n$-dimensional irreducible representation of $\SL(2, \RR)$,
where $n = 2m$ when $G=\Sp(2m,\RR)$ and $n= 2m+1$ when
$G=\SO(m,m+1)$. Hitchin representations are $(G,H)$-Anosov, where $H$
is the intersection of two opposite minimal parabolic subgroups
\cite[Theorems~4.1, 4.2]{Labourie_anosov},
\cite[Theorem~1.15]{Fock_Goncharov}. For $G= \SL(n,\RR),$
$\Sp(2m,\RR)$ or $\SO(m,m+1)$, $H$ is the subgroup of diagonal
matrices, in particular, the $H$-reduction corresponds to a splitting
of the flat $\RR^n$-bundle over $T^1\Sigma$ into $n$ line bundles.

\subsubsection{Other examples}
\begin{asparaenum}
\item Any quasi-Fuchsian representation $\pi_1(\Sigma)\to \PSL(2,\CC)$
  is Anosov.
\item Embed $\SL(2,\RR)$ into $\PGL(3,\RR)$ as stabilizer of a point
  and consider the representation $\rho: \pi_1(\Sigma) \to \SL(2,\RR)
  \to \PGL(3,\RR)$. Then small deformations of $\rho$ are
  $(\PGL(3,\RR),H)$-Anosov where $H$ is the subgroup of diagonal
  matrices. These representations were studied in
  \cite{Barbot_anosov}.
\item Let $G$ be a semisimple Lie group, $G'<G$ a rank one subgroup,
  $\Lambda < G'$ a cocompact torsionfree lattice and $N =
  \Lambda\backslash G'/K$, where $K<G'$ is a maximal compact
  subgroup. Let $M = T^1 N$ be the unit tangent bundle of $N$. Then
  the composition $\pi_1(M) \to \Lambda \to G' < G$ is a
  $(G,H)$-Anosov representation, where $H$ is the identity component
  of the centralizer in $G$ of a real split Cartan subgroup in $G'$
  (compare with Lemma~\ref{lem_anosov_constructions}), see
  \cite[Prop.~3.1]{Labourie_anosov}.
\item In \cite{Barbot_quasifuchsian,Merigot} a notion of
  quasi-Fuchsian representations for a cocompact lattice $\Lambda <
  \SOcon(1,n)$ into $\SOcon(2,n)$ is introduced, and it is shown that
  these quasi-Fuchsian representations are Anosov representations.
\end{asparaenum}

\subsection{Standard maximal representations}
\label{sec_standard}

In this section we describe the construction of several maximal
representations
\[
\rho: \pi_1(\Sigma) \longrightarrow \Sp(2n,\RR)
\]
to which we will refer as {\em standard representations}. All these
representations come from homomorphisms of $\SL(2,\RR)$ into
$\Sp(2n,\RR)$, possibly twisted by a representation of $\pi_1(\Sigma)$
into the centralizer of the image of $\SL(2,\RR)$ in $\Sp(2n,\RR)$.
By construction, the image of any such representation will be
contained in a proper closed Lie subgroup of
$\Sp(2n,\RR)$.\footnote{More precisely, $\rho(\pi_1(\Sigma))$ will
  preserve a totally geodesic tight disk in the symmetric space
  associated with $\Sp(2n,\RR)$ (the notion of tight disk is not used
  in this paper, the interested reader is referred to
  \cite{Burger_Iozzi_Wienhard_tight}).}
 
Let us fix a discrete embedding $\iota:\pi_1(\Sigma) \to \SL(2,\RR)$.
\subsubsection{Irreducible Fuchsian representations}
Consider $V_0 =\RR_1[X,Y] \cong \RR^2$ the space of homogeneous
polynomials of degree one in the variables $X$ and $Y$, endowed with
the symplectic form determined by
\[
\omega_0(X,Y) = 1.
\]

The induced action of $\Sp(V_0)$ on $V = \sym^{2n-1} V_0 \cong
\RR_{2n-1}[X,Y]\cong \RR^{2n}$ preserves the symplectic form
$\omega_n= \sym^{2n-1} \omega_0 $, explicitly
\[ \omega_n(P_k, P_{l}) = 0 \text{ if } k+l \neq 2n-1 \text{ and }
\omega_n(P_k, P_{2n-1-k}) = \frac{(-1)^{k}}{ (2n-1)!},\] where $P_k =
X^{2n-1-k}Y^k/k!$.

This defines the $2n$-dimensional irreducible representation of
$\Sp(V_0) \cong\SL(2,\RR)$ into $\Sp(V) \cong\Sp(2n,\RR)$,
\[
\phi_{irr}: \SL(2,\RR) \longrightarrow \Sp(2n,\RR),
\]
which, by precomposition with $\iota: \pi_1(\Sigma)\to \SL(2,\RR)$,
gives rise to an {\em irreducible Fuchsian representation}
\[
\rho_{irr}: \pi_1(\Sigma) \longrightarrow \SL(2,\RR) \longrightarrow
\Sp(2n,\RR).
\]

\begin{nota}
  For a line bundle $L$ and a non-zero integer $n$, we use the
  notation $L^n$ for the line bundle that is the tensor product of $n$
  copies of $L$ when $n$ is positive or of $-n$ copies of $L^*$ if $n$
  is negative. By convention, $L^0$ is the trivial line bundle. The
  line bundles $L^n$ and $L^{-n}$ are naturally dual to each other.
\end{nota}

\begin{facts}\label{facts:irreducible}
  \begin{asparaenum}
  \item \label{item:lag_irr} Let $L^s(\iota)$ be the line bundle over
    $T^1\Sigma$ associated with the embedding $\iota:\pi_1(\Sigma) \to
    \SL(2,\RR)$, and $E_\iota$, $E_{ \rho_{irr}}$ the flat symplectic
    bundles over $T^1 \Sigma$. As $E_{ \rho_{irr}}= \sym^{2n-1}
    E_\iota$ and $E_\iota = L^s(\iota) \oplus L^u(\iota) = L^s(\iota)
    \oplus L^s(\iota)^{*} = L^s(\iota) \oplus L^s(\iota)^{-1}$, the
    Lagrangian reduction $L^s(\rho_{irr})$ over $T^1\Sigma$ associated
    with $\rho_{irr}$ is
    \[
    L^s(\rho_{irr}) = L^s(\iota)^{2n-1} \oplus L^s(\iota)^{2n-3}
    \oplus \cdots \oplus L^s(\iota).
    \]
  \item\label{item:splitting} When $n=2$, let us choose the symplectic
    identification $(\RR_3[X,Y], -\omega_2) \cong (\RR^4,\omega)$
    given by $X^3 = e_1, X^2 Y = -e_2 / \sqrt{3} , Y^3 = -e_3 , X Y^2
    = -e_4 / \sqrt{3}$ ($\omega$ was defined in
    Section~\ref{sec:prelim_maxrep_sp}).  With respect to this
    identification the irreducible representation $\phi_{irr}:
    \SL(2,\RR) \to \Sp(4,\RR)$ is given by the following formula
    \[ \phi_{irr} \left(
      \begin{array}{cc}
        a & b \\ c & d
      \end{array} \right) = \left(
      \begin{array}{cccc}
        a^3 & - \sqrt{3} a^2 b & -b^3 & -\sqrt{3} a b^2 \\
        -\sqrt{3} a^2 c & 2abc + a^2 d & \sqrt{3} b^2 d & 2 a b d + b^2 c
        \\
        -c^3 & \sqrt{3} c^2 d & d^3 & \sqrt{3} c d^2 \\
        -\sqrt{3} a c^2 & 2 a c d + b c^2 & \sqrt{3} b d^2 & 2 b c d + a d^2
      \end{array} \right). \]
    In particular
    $\phi_{irr}(\diag( e^m, e^{-m})) = \diag ( e^{3m}, e^m, e^{-3m}, e^{-m})$. 

    This choice has been made so that $(\phi_{irr})_*: \pi_1( \SL(2,
    \RR)) \to \pi_1(\Sp(4, \RR)) $ is the multiplication by $2$. Note
    that the more immediate identification $(\RR_3[X,Y], \omega_2)
    \cong (\RR^4,\omega)$ given by $X^3 = e_1, X^2 Y = -e_2 / \sqrt{3}
    , Y^3 = e_3 , X Y^2 = e_4 / \sqrt{3}$ would produce the morphism
    $\SL(2, \RR) \to \Sp(4, \RR)$ which is the conjugate of
    $\phi_{irr}$ by $\diag(1,1,-1,-1)$ (this last element is
    \emph{not} in $\Sp(4, \RR)$); however this morphism induces the
    multiplication by $-2$ at the level of fundamental groups.
  \end{asparaenum}
\end{facts}

\subsubsection{Diagonal Fuchsian representations} 
Let
\[
\RR^{2n}= W_1 \oplus \cdots \oplus W_{n}
\]
with $W_i = \Span(e_i, e_{n+i})$ be a symplectic splitting of
$\RR^{2n}$.  Identifying $W_i \cong \RR^2$, this splitting gives rise
to an embedding
\[
\psi: \SL(2,\RR)^n \longrightarrow \Sp(W_1) \times \cdots \times
\Sp(W_{n}) \subset \Sp(2n,\RR).
\]
Precomposing with the diagonal embedding of $\SL(2,\RR) \to
\SL(2,\RR)^n$ we obtain the diagonal embedding
\[
\phi_\Delta: \SL(2,\RR) \longrightarrow \Sp(2n,\RR).
\]
Precomposition with $\iota: \pi_1(\Sigma) \to \SL(2,\RR)$ gives rise
to a \emph{diagonal Fuchsian representation}
\[
\rho_\Delta: \pi_1(\Sigma) \longrightarrow \SL(2,\RR) \longrightarrow
\Sp(2n,\RR).
\]

\begin{facts}\label{facts:diagonal}
  \begin{asparaenum}
  \item\label{item:lag_diag} Let $L^s(\iota)$ be the Lagrangian line
    bundle over $T^1\Sigma$ associated with $\iota: \pi_1(\Sigma) \to
    \SL(2,\RR)$, then the Lagrangian reduction $L^s(\rho_\Delta)$ of
    the flat symplectic $\RR^{2n}$-bundle over $T^1\Sigma$ associated
    with $\rho_{\Delta}$ is given by
    \[
    L^s(\rho_\Delta) = L^s(\iota)\oplus \cdots \oplus L^s(\iota).
    \]
  \item\label{item:embedding_psi}When $n=2$ and with respect to the
    symplectic basis $(e_i)_{i=1, \dots,4}$ the map $\psi$ is given by
    the following formula
    \[ \psi\left( \left(
        \begin{array}{cc}
          a & b \\ c & d
        \end{array}\right) , \left(
        \begin{array}{cc}
          \alpha & \beta \\ \gamma & \delta
        \end{array}\right) \right) = \left(
      \begin{array}{cccc}
        a & 0 & b & 0 \\ 0 & \alpha & 0 & \beta \\ c & 0 & d & 0 \\ 0 &
        \gamma & 0 & \delta
      \end{array}\right ).\]
    This choice has been made so that, for the diagonal embedding
    $\phi_\Delta$, the map $(\phi_{\Delta})_*$ is the multiplication
    by $2$.
  \end{asparaenum}
\end{facts}

\subsubsection{Twisted diagonal representations}
We now vary the construction of the previous subsection.  For this
note that the image $\phi_\Delta(\SL(2,\RR)) < \Sp(2n,\RR)$ has a
fairly large centralizer, which is a compact subgroup of $\Sp(2n,\RR)$
isomorphic to $\O(n)$.
\begin{remark}
  For any maximal representation $\rho: \pi_1(\Sigma) \to \Sp(2n,\RR)$
  the centralizer of $\rho(\pi_1(\Sigma))$ is a subgroup of $\O(n)$.
  This is because the centralizer of $\rho(\pi_1(\Sigma))$ fixes
  pointwise the positive curve in the space of Lagrangians. In
  particular it will be contained in the stabilizer of one positive
  triple of Lagrangians which is isomorphic to $\O(n)$.
\end{remark}

That the centralizer of $\phi_\Delta(\SL(2,\RR))$ is precisely $\O(n)$
can be seen in the following way. Let $(W,q)$ be an $n$-dimensional
vector space equipped with a definite quadratic form $q$ and let again
$V_0 = \RR^{2}$ with its standard symplectic form $\omega_0$. The
tensor product $V_0 \otimes W$ inherits the bilinear nondegenerate
form $\omega_0 \otimes q$ which is antisymmetric, so that we can
choose a symplectic identification $\RR^{2n} \cong V_0 \otimes W$.
This gives an embedding
\[
\SL(2,\RR) \times \O(n) \cong \Sp(V_0) \times \O(W,q)
\xrightarrow{\phi_\Delta} \Sp(2n,\RR),
\]
which extends the morphism $\phi_\Delta$ defined above.

Now given $\iota: \pi_1(\Sigma) \to \SL(2, \RR)$ and a representation
$\Theta: \pi_1(\Sigma) \to \O(n)$, we set
\begin{align*}
  \rho_{\Theta}= \iota \otimes \Theta: \pi_1(\Sigma) &\longrightarrow \Sp(2n,\RR)\\
  \gamma &\longmapsto \phi_\Delta( \iota(\gamma), \Theta(\gamma)).
\end{align*}
We will call such a representation a \emph{twisted diagonal
  representation}.
\begin{facts}\label{facts:twisted}
  \begin{asparaenum}
  \item The flat bundle $E$ over $\Sigma$ associated with
    $\rho_{\Theta}:\pi_1(\Sigma) \to \Sp(2n,\RR) \cong \Sp(V_0 \otimes
    W)$ is of the form
    \[ E = E_0 \otimes W,\] where (with a slight abuse of notation)
    $W$ is the flat $n$-plane bundle associated with $\Theta$ and
    $E_0$ the flat $2$-plane bundle over $\Sigma$ associated with
    $\iota$.
  \item\label{item:lag_twisted} Let $L^s(\iota)$ be the line bundle
    over $T^1\Sigma$ associated with $\iota$.  Let $\overline{W}$
    denote the flat $n$-plane bundle over $T^1\Sigma$ given by the
    pullback of $W$.  Then the Lagrangian reduction $L^s(\rho_\Theta)$
    is the tensor product
    \[
    L^s(\rho_\Theta) = L^s(\iota)\otimes \overline{W}.
    \]
  \end{asparaenum}
\end{facts}

\subsubsection{Standard representations for other groups} 
Let $G$ be an almost simple Lie group of Hermitian type of real rank
$n$. Then there exists a unique (up to conjugation by $G$) embedding
$\mathfrak{sl}(2,\RR)^n \to \mathfrak{g}$, which gives, at the level
of Lie groups, a finite-to-one morphism $L^n \to G$ where $L$ is a
finite cover of $\PSL(2, \RR)$. We call the precomposition of such an
embedding with the diagonal embedding $L \to L^n$ a diagonal embedding
\[
\phi_\Delta: L \longto G.
\]
The centralizer of $\phi_\Delta(L)$ in $G$ is always a compact
subgroup $K'$ of $G$.

The composition $\rho_\Delta = \phi_\Delta\circ \iota: \pi_1(\Sigma)
\to G$ is a maximal representation. Given a representation $\Theta:
\pi_1(\Sigma) \to K'$ we can again define a \emph{twisted diagonal
  representation}
\begin{align*}
  \rho_{\Theta}: \pi_1(\Sigma) &\longto G\\
  \gamma &\longmapsto \rho_\Delta(\gamma) \cdot \Theta(\gamma).
\end{align*}

\begin{remark}
  In the general case the subgroup $K'$ can be characterized as being
  the intersection of a maximal compact subgroup $K$ in $G$ with the
  subgroup $H<G$, which is the stabilizer of a pair of transverse
  points in the Shilov boundary of the symmetric space associated with
  $G$. Equivalently, $K'$ is the stabilizer in $G$ of a maximal triple
  of points in the Shilov boundary. (For the definition of maximal
  triples see \cite[Section~2.1.3]{Burger_Iozzi_Wienhard_toledo}).
\end{remark}

\subsection{Amalgamated representations}
\label{sec:amalagamatedRep}

Due to Theorem~\ref{thm:gluing_maximal} we can construct maximal
representations of $\pi_1(\Sigma)$ by amalgamation of maximal
representations of the fundamental groups of subsurfaces.

Let $\Sigma = \Sigma_l \cup_{\gamma} \Sigma_r$ be a decomposition of
$\Sigma$ along a simple closed separating oriented geodesic $\gamma$
into two subsurfaces $\Sigma_l$, lying to the left of $\gamma$, and
$\Sigma_r$, lying to the right of $\gamma$. Then $\pi_1(\Sigma)$ is
isomorphic to $\pi_1(\Sigma_l) *_{\langle\gamma\rangle}
\pi_1(\Sigma_r)$, where we identify $\gamma$ with the element it
defines in $\pi_1(\Sigma)$.

We will call a representation constructed by amalgamation of two
representations $\rho_l: \pi_1(\Sigma_l) \to G$ and
$\rho_r:\pi_1(\Sigma_r) \to G$ with $\rho_l(\gamma) = \rho_r(\gamma)$
an \emph{amalgamated representation} $\rho = \rho_l*\rho_r:
\pi_1(\Sigma) \to G$. By Theorem~\ref{thm:gluing_maximal}, the
amalgamated representation $\rho$ is maximal if and only if $\rho_l$
and $\rho_r$ are maximal.\footnote{Note that it is important that the
  Toledo invariant for both $\rho_l$ and $\rho_r$ are of the same sign
  (and also that the orientations of $\Sigma_l$ and $\Sigma_r$
  ---involved in the definition of the Toledo number--- are compatible
  with the orientation of $\Sigma$). Amalgamating a maximal
  representation with a minimal representation does not give rise to a
  maximal representation.}

\subsubsection{Hybrid representations}
\label{sec:description_hybrid}\label{sec:dec_surf}\label{sec:the_rep_hybrid}
In this section we describe the most important class of maximal
representations $\rho: \pi_1(\Sigma) \to \Sp(4,\RR)$ obtained via
amalgamation.  We call these representation \emph{hybrid
  representations} to distinguish them from general amalgamated
representations.

Let $\iota: \pi_1(\Sigma) \to \SL(2,\RR)$ be a discrete embedding. The
basic idea of the construction of hybrid representations is to
amalgamate the restriction to $\Sigma_l$ of the irreducible Fuchsian
representation $\rho_{irr} = \phi_{irr} \circ \iota$ and the
restriction to $\Sigma_r$ of the diagonal Fuchsian representation
$\rho_{\Delta} = \phi_{\Delta} \circ \iota$. This does not directly
work because the holonomies of $\rho_{irr}$ and $\rho_\Delta$ along
$\gamma$ do not agree, but a slight modification works.

\begin{remark} 
  This idea of using two different embeddings of $\SL(2, \RR)$ to
  construct amalgamated representations was previously used in the
  paper \cite{Goldman_Kapovich_Leeb} to construct representations
  $\pi_1(\Sigma) \to \PU(2,1)$ of every possible Toledo number (see in
  particular the concept of ``hybrid surfaces''
  \cite[Sec.~2.4]{Goldman_Kapovich_Leeb}).
\end{remark}

Assume that $\iota(\gamma) = \diag( e^m, e^{-m})$ with $m>0$. Set
\[ {\rho}_l := \phi_{irr} \circ \iota: \pi_1(\Sigma) \to \Sp(4,\RR),
\]
with $\phi_{irr}$ defined in
Facts~\ref{facts:irreducible}.(\ref{item:splitting}).  Then
${\rho}_l(\gamma) = \diag(e^{3m}, e^m, e^{-3m}, e^{-m})$.

Let $( \tau_{1,t})_{t \in [0,1]}$ and $( \tau_{2,t})_{t \in [0,1]}$ be
continuous paths of discrete embeddings $\pi_1( \Sigma) \to \SL(2,
\RR)$ such that $\tau_{1,0} = \tau_{2,0} = \iota$, and, for all $t \in
[0,1]$,
\[ \tau_{1,t}( \gamma ) = \left(
  \begin{array}{cc}
    e^{l_{1,t}} & \\ & e^{-l_{1,t}}
  \end{array}\right) \text{ and } \tau_{2,t}( \gamma ) = \left(
  \begin{array}{cc}
    e^{l_{2,t}} & \\ & e^{-l_{2,t}}
  \end{array}\right),
\]
where $l_{1,t} > 0$ and $l_{2,t} > 0$, $l_{1,0} = l_{2,0} = m$,
$l_{1,1} = 3m$ and $l_{2,1} = m$. The existence of $\tau_{i,t}$ is a
classical fact from hyperbolic geometry, for the reader's convenience
we include the statement we are using in Lemma~\ref{lem:deformation}.
Set
\[{\rho}_r:= \psi \circ (\tau_{1,1}, \tau_{2,1}): \pi_1(\Sigma)
\longrightarrow \Sp(4,\RR),
\]
with $\psi$ defined by
Facts~\ref{facts:diagonal}.(\ref{item:embedding_psi}).  Then ${\rho}_r
$ is a continuous deformation of $\phi_\Delta \circ \iota$ which
satisfies $\rho_r(\gamma) = \rho_l(\gamma)$.

\begin{defi}\label{defi:hybrid} 
  A {\em hybrid representation} is the amalgamated representation
  \begin{equation*}
    \rho:=  {{\rho}_l}|_{\pi_1(\Sigma_l)} * {{\rho}_r}|_{\pi_1(\Sigma_r)}:
    \pi_1(\Sigma)  = \pi_1(\Sigma_l) *_{\langle\gamma\rangle}
    \pi_1(\Sigma_r) \longrightarrow \Sp(4,\RR).
  \end{equation*}
  If $\chi(\Sigma_l) = k$ we call the representation a $k$-hybrid
  representation.
\end{defi}

Since $ {{\rho}_l}|_{\pi_1(\Sigma_l)}$ and
${{\rho}_r}|_{\pi_1(\Sigma_r)}$ are maximal representations, the
representation $\rho$ is maximal (see
Theorem~\ref{thm:gluing_maximal}).

\begin{remark}\label{rema_choices_embed} 
  The special choices for the embeddings $\phi_{irr}$ and $\psi$
  (Facts~\ref{facts:irreducible}.\eqref{item:splitting} and
  \ref{facts:diagonal}.\eqref{item:embedding_psi}) will be important
  for the calculation of the Euler class of the Anosov reduction of
  the hybrid representation. Obviously one can always change one of
  this two embeddings by conjugation by an element of the centralizer
  of $\rho( \gamma)$, i.e.\ an element, which is of the form
  $\diag(a,b,a^{-1}, b^{-1})$, these representations are also maximal.
\end{remark}

In order to keep track of this situation we define

\begin{defi}
  \label{defi_good_triples}
  Let $\gamma$ be a loop\footnote{The loop $\gamma$ does not need to
    be separating or simple.} on $\Sigma$ and let ${\rho}_l$ and
  ${\rho}_r$ two representations of $\pi_1(\Sigma)$ into $\Sp(4,\RR)$
  with ${\rho}_l(\gamma) = {\rho}_r(\gamma)$ and such that ${\rho}_l$
  is a Hitchin representation and ${\rho}_r$ is a maximal
  representation into $\SL(2, \RR) \times \SL(2, \RR)<\Sp(4,\RR)$.

  The pair $({\rho}_l,{\rho}_r)$ is said to be \emph{positively
    adjusted with respect to $\gamma$} if there exists a symplectic
  basis $(\epsilon_i)_{i=1,\dots,4}$ and continuous deformations
  $({\rho}_{l,t})_{t\in [0,1]}$ and $({\rho}_{r,t})_{t\in [0,1]}$ such
  that:
  \begin{itemize}
  \item ${\rho}_{l,1}= {\rho}_l$ and ${\rho}_{r,1}= {\rho}_r$,
  \item ${\rho}_{l,0} = \phi_{irr} \circ \iota $ is an irreducible
    Fuchsian representation, ${\rho}_{r,0} = \phi_\Delta \circ \iota$
    is a diagonal Fuchsian representation and for each $t$ in $[0,1]$
    ${\rho}_{r,t}$ is a maximal representation into $\SL(2, \RR)
    \times \SL(2, \RR)<\Sp(4,\RR)$,
  \item for each $t$, $\rho_{l,t}( \gamma)$ and $\rho_{r,t}( \gamma)$
    are diagonal in the base $( \epsilon_i)$.
  \end{itemize}

  The pair $({\rho}_l,{\rho}_r)$ is said to be \emph{negatively
    adjusted with respect to $\gamma$} if the pair $(g {\rho}_l
  g^{-1},{\rho}_r)$ is positively adjusted where $g$ is diagonal in
  the base $( \epsilon_i)_{i=1,\dots,4}$ with eigenvalues of opposite
  signs, i.e.\ $g= \diag(a,b,a^{-1}, b^{-1})$ with $ab<0$.
\end{defi}

\begin{remark}
  In fact, it is not really necessary here that the representations
  ${\rho}_l$ and ${\rho}_r$ are representations of $\pi_1( \Sigma)$,
  since only their restrictions to $\pi_1(\Sigma_l)$ or
  $\pi_1(\Sigma_r)$ are considered.
\end{remark}

Since a hybrid representation is maximal the associated flat
symplectic $\RR^4$-bundle $E$ admits a Lagrangian splitting $E =
L^s(\rho) \oplus L^u(\rho)$. We are unable to describe the Lagrangian
bundles $L^s (\rho)$ and $L^u(\rho)$ explicitly as we did above for
standard maximal representations. Nevertheless, computing the
topological invariants (see Section~\ref{sec:computations}) we defined
for Anosov representations, we will be able to determine the
topological type of $L^s(\rho)$. The topological type will indeed only
depend on the Euler characteristic of $\Sigma_l$.

\subsubsection{Hybrid representations: general construction}
\label{sec:gener-hybr-repr}
In the construction above we decompose $\Sigma$ along one simple
closed separating geodesic, so the Euler characteristic of $\Sigma_l$
will be odd. To obtain $k$-hybrid representations of $\pi_1(\Sigma)$
for all $ \chi(\Sigma)+ 1\leq k \leq -1$ we have to consider slightly
more general decompositions of $\Sigma$, in particular $\Sigma_l$ or
$\Sigma_r$ might not be connected.

Let us fix some notation to describe this more general construction.
Let $\Sigma$ be closed oriented surface of genus $g$ and $\Sigma_1
\subset \Sigma$ be a subsurface with Euler characteristic equal to
$k$.

The (non-empty) boundary $\partial \Sigma_1$ is the union of disjoint
circles $\gamma_d$ for $d \in \pi_0( \partial \Sigma_1)$. We orient
the circles so that, for each $d$, the subsurface $\Sigma_1$ is on the
left of $\gamma_d$.  Write the surface $\Sigma\moins \partial
\Sigma_1$ as the union of its connected components:
\[
\Sigma\moins \partial \Sigma_1 = \bigcup_{c \in \pi_0( \Sigma
  \moins \partial \Sigma_1)} \Sigma_c.
\]
For any $d$ in $\pi_0( \partial \Sigma_1)$ the curve $\gamma_d$ bounds
exactly $2$ connected components of $\Sigma \moins \partial \Sigma_1$:
one is included in $\Sigma_1$ and denoted by $\Sigma_{l(d)}$ with
$l(d)$ in $\pi_0(\Sigma_1)$; the other one is included in the
complement of $\Sigma_1$ and denoted by $\Sigma_{r(d)}$ with $r(d)$ in
$\pi_0(\Sigma \moins \Sigma_1)$.  Note that $l(d)$ and $r(d)$ are
elements of $\pi_0( \Sigma \moins \partial \Sigma_1)$. In particular,
${l(d)}$ might equal ${l(d')}$ for some $d\neq d'$; similarly for
${r(d)}$ and ${r(d')}$.

We assume that
\begin{itemize}
\item The graph with vertex set $\pi_0( \Sigma \moins \partial
  \Sigma_1)$ and edges given by the pairs $\{ (l(d), r(d))\}_{d \in
    \pi_0(\partial \Sigma_1)}$ is a tree.
\end{itemize}

The fundamental group $ \pi_1( \Sigma)$ can be described as the
amalgamated product of the groups $\pi_1( \Sigma_c)$, $c$ in
$\pi_0(\Sigma \moins \partial \Sigma_1)$, over the groups $\pi_1(
\gamma_d)$, $d\in \pi_0( \partial \Sigma_1 )$.  The above assumption
ensures that no HNN-extensions appear in this description.

\medskip

With these notations, we can now define general $k$-hybrid
representations.  For each $c$ in $\pi_0 (\Sigma_1)$ we choose a
representation
\[\rho_{c} : \pi_1( \Sigma) \longrightarrow \Sp( 4, \RR)\]
belonging to one of the $2^{2g}$ Hitchin components.  We set with a
slight abuse of notation $ \rho_c = {\rho}_c |_{\pi_1( \Sigma_{c})}$
for each $c$ in $\pi_0 (\Sigma_1)$.

For any $d$ in $\pi_0( \partial \Sigma_1)$, $\rho_{l(d)}(\gamma_d)$
(note that $l(d) \in \pi_0( \Sigma_1)$) is conjugate to a unique
element of the form
\[ \rho_{l(d)}(\gamma_d) \cong \epsilon(d) \left(
  \begin{array}{cccc}
    e^{l_1(d)} & & & \\ & e^{l_2(d)} & & \\ & & e^{-l_1(d)} & \\ & & & e^{-l_2(d)}
  \end{array}\right ) \in \Sp(4, \RR)
\]
with $\epsilon(d) \in \{ \pm 1 \}$, $l_1(d) > l_2(d) > 0$.

The construction of $\rho_{c'}$ for $c'$ in $\pi_0( \Sigma \moins
\Sigma_1 )$ now goes as follow. By Lemma~\ref{lem:deformation} one can
choose a continuous path
\[ {\tau}_{c', t} : \pi_1( \Sigma) \longrightarrow \SL(2, \RR) , \ t
\in [1,2] \] such that $\tau_{c',t} : \pi_1( \Sigma) \to \SL(2, \RR)$
are discrete embeddings for all $t \in [1,2]$ and such that, for any
$d$ in $\pi_0( \partial \Sigma_{c'}) \subset \pi_0( \partial
\Sigma_1)$, (hence $r(d) =c'$) one has
\[ \tau_{c',i}( \gamma_d) \text{ is conjugate to } \epsilon(d) \left(
  \begin{array}{cc}
    e^{l_i(d)} &  \\ & e^{-l_i(d)} 
  \end{array}\right ), \text{ for } i =1,2.
\]
Set ${\rho}_{c'} = \psi \circ ( \tau_{c',1} , \tau_{c',2}) : \pi_1(
\Sigma_{c'}) \rightarrow \Sp( 4, \RR)$, where $\psi$ is defined in
Fact~\ref{facts:diagonal}.(\ref{item:embedding_psi}).

In order to define the amalgamated representation, we need to choose
elements $g_{c}$ in $\Sp(4, \RR)$ for any $c$ in $\pi_0( \Sigma
\moins \partial \Sigma_1)$ such that, for any $d$, $g_{l(d)} \rho_{
  l(d) } ( \gamma_d) g_{l(d)}^{-1} = g_{r(d)} \rho_{ r(d) } (
\gamma_d) g_{r(d)}^{-1}$. As mentioned in
Remark~\ref{rema_choices_embed} these elements should be chosen such
that
\begin{itemize}
\item for any $d$ in $\pi_0( \partial \Sigma_1)$ the pair of
  representations \[ \big(g_{l(d)} {\rho}_{ l(d) } g_{l(d)}^{-1},
  g_{r(d)} {\rho}_{ r(d) } g_{r(d)}^{-1}\big)\] is positively adjusted
  with respect to $\gamma_d$ (Definition~\ref{defi_good_triples}).
\end{itemize}
Such a family $(g_c)$ always exists by our hypothesis that the graph
associated with the decomposition of the surface $\Sigma$ is a
tree. One then constructs the $k$-hybrid representation
\[ \rho : \pi_1( \Sigma) \longrightarrow \Sp(4, \RR)\] by amalgamation
of the representations $g_{l(d)} \rho_{l(d)} g_{l{d}}^{-1}$ and
$g_{r(d)} \rho_{ r(d) } ( \gamma_d) g_{r(d)}^{-1}$.

\begin{remark}
  The hypothesis that the dual graph of $\Sigma \moins \partial
  \Sigma_1$ is a tree is necessary. For example, if this graph has a
  double edge, one would try to construct a Hitchin representation
  whose restriction to the disjoint union of two closed simple curves
  $\gamma_1$ and $\gamma_2$ is contained in some $\SL(2, \RR) \times
  \SL(2, \RR)<\Sp(4,\RR)$. However it is not difficult to see that the
  restriction of a Hitchin representation to the subgroup generated by
  $\gamma_1$ and $\gamma_2$ is irreducible and hence cannot be
  contained in $\SL(2, \RR) \times \SL(2, \RR)< \Sp(4,\RR)$.
\end{remark}

\subsubsection{Other amalgamated representations}\label{sec:other_amalgams}
Let us describe a variant of the construction of hybrid
representation. Assume that $\Sigma$ is decomposed along a simple
closed separating geodesic $\gamma$ into two subsurfaces $\Sigma_l$
and $\Sigma_r$ as above.  On $\pi_1(\Sigma_l)$ we choose again the
irreducible Fuchsian representation $\rho_{irr} = \phi_{irr} \circ
\iota$ into $\Sp(4,\RR)$, for the fundamental group of
$\pi_1(\Sigma_r)$ we choose a maximal representation into $\SL(2,\RR)
\times \SL(2,\RR) \subset \Sp(4,\RR)$ which agrees with the
irreducible representation along $\gamma$, but sends an element
$\alpha \in \pi_1(\Sigma_r)$ corresponding to a non-separating simple
closed geodesic to an element of $\SL(2,\RR) \times \SL(2,\RR) $ with
eigenvalues of different sign. The corresponding amalgamated
representation will be maximal. However, the first Stiefel-Whitney
class of this representation is non-zero by Lemma~\ref{lem:sw1=sign}.
Thus Theorem~\ref{thm_intro:n=2_components} implies that $\rho$ can be
deformed to a twisted diagonal representation.

Analogous constructions can be made to obtain maximal representations
into other Lie groups $G$ of Hermitian type. When $G$ is not locally
isomorphic to $\Sp(4,\RR)$, we expect that all maximal representations
can be deformed to twisted diagonal representations.

\section{Topological invariants}
\label{sec:invariants}\label{sec:TopologicalInvariants}

In this section we introduce the topological invariants for Anosov
representations. First the uniqueness of the \emph{Anosov section} is
restated (Section~\ref{sec:defining-invariants}). Some general
properties of obstruction classes for those Anosov sections are
discussed (Section~\ref{sec:char_class_in_gal}).  Then we define first
and second Stiefel-Whitney classes for symplectic maximal
representations (Section~\ref{sec_fst_scd_SW}). In
Section~\ref{sec:sp4case} we focus on the special case $\Sp(4,\RR)$
and define an Euler class considering a cover of the representation
space (\emph{enhanced} representation space,
Section~\ref{sec:thespaces}).  In Section~\ref{sec:constraints} we
prove a priori constraints on the possible values for the
invariants. Finally in Section~\ref{sec_inv_others} we define and
calculate invariants for other Anosov representations.

\subsection{Uniqueness}
\label{sec:defining-invariants}

Let us denote by $\homHano ( \pi_1(M), G)$ the set of $(G,H)$-Anosov
representations and let $\mathcal{B}_H (M)$ the set of gauge
isomorphism classes of $H$-bundles over $M$.  Summarizing
Proposition~\ref{prop:uniqsigma} and Proposition~\ref{prop:openano} we
have
\begin{prop}\label{prop:invariants}
  For any pair $(G,H)$, there is a well-defined locally constant map
  \[ \homHano( \pi_1(M), G) \longrightarrow \mathcal{B}_H (M),\]
  associating to an Anosov representation its Anosov $H$-reduction.

  This map is natural with respect to taking finite covers of $M$ and
  with respect to the constructions described in
  Lemmas~\ref{lem_anosov_constructions} and \ref{lem:finite_cover} and
  in Proposition~\ref{prop:coverM}.
\end{prop}

In general $\mathcal{B}_H (M)$ could be rather complicated, so instead
of the whole space of gauge isomorphism classes $\mathcal{B}_H (M)$ we
will consider the obstruction classes associated with the $H$-bundle
as topological invariants of an Anosov representation, i.e.\ the
invariants are elements of the cohomology groups of $M$, possibly with
local coefficients.

\begin{remark}
  \label{rem:reduc-P}
  The group $H$ is the Levi component of $P^s$, therefore $P^s$ and
  $H$ are homotopy equivalent. Thus $\mathcal{B}_H =
  \mathcal{B}_{P^s}$, and instead of working with the $H$-reduction we
  could equally well work with the corresponding $P^s$-reduction (or
  similarly with the $P^u$-reduction) of the flat $G$-bundle.
\end{remark}

\subsection{Characteristic classes}
\label{sec:char_class_in_gal}

Obstruction theory (see \cite[Part~III]{Steenrod}) associates
characteristic classes to fiber bundles; they measure the obstruction
to constructing sections of the bundles. In our special case, if
$\mathsf{P}_H$ is a principal $H$-bundle over $M$ we obtain a
characteristic class
\[ o_1(\mathsf{P}_H) \in \h^1(M; \pi_0(H)^{ab}),\] with
$\pi_0(H)^{ab}$ the abelianization of $\pi_0(H)$.  This class can be
explicitly described as follows: the associated principal
$\pi_0(H)$-bundle $\mathsf{P}_H \times_H \pi_0(H)$ is necessarily flat
(since its structure group is discrete) and hence comes from a
morphism $\pi_1(M)\to \pi_0(H)$. This later morphism defines
$o_1(\mathsf{P}_H)$ a class in $\h^1(M; \pi_0(H)^{ab}) =
\h^1(\pi_1(M); \pi_0(H)^{ab}) \cong \hom(\pi_1(M), \pi_0(H)^{ab})$.

\begin{defi}
  The \emph{first obstruction class} $o_1( \rho)$ of a $(G,H)$-Anosov
  representation $\rho: \pi_1(M)\to G$ is the class
  $o_1(\mathsf{P}_H)$ of the Anosov reduction $\mathsf{P}_H$.
\end{defi}

\begin{remark}\label{rem:obstr_morphism}
  A more precise invariant would be the first obstruction morphism
  $\sigma_1(\rho)$ in $\rep(\pi_1(M), \pi_0(H))$. In every case
  considered below $\pi_0(H)$ is abelian, so that the obstruction
  morphism and the obstruction class are equal. This is not, however,
  the general case.
\end{remark}

In order to get a characteristic class of degree $q+1$ one need to
consider bundles whose fibers are $q$-connected
(\cite[Sec.~29]{Steenrod}). We restrict here the discussion to
$q=1$. A way to get a class in a second cohomology group is to
consider the $H/F$-bundle $\mathsf{P}_H / F$ where $F$ is a subgroup
of $H$ such that $H/F$ is connected. We will define a second
cohomology class under the following

\begin{cond}\label{cond:H} 
  The group $H$ is the semidirect product of its identity component
  $H_0$ and a discrete group $F$; $H= F \ltimes H_0$.
\end{cond}
In that situation $F\cong \pi_0(H)$ and $\pi_1(H/F) \cong
\pi_1(H_0)$. The obstruction $o_2(\mathsf{P}_H)$ to finding a section
of the bundle $\mathsf{P}_H/F$ is a class in $\h^2(M; \pi_1(H/F))$ the
second cohomology group of $M$ with coefficients in the
$\pi_1(M)$-module $\pi_1(H/F)$ (for generalities on cohomology with
local coefficients see \cite[Sec.~3.H]{Hatcher_AT} or
\cite[Sec.~31]{Steenrod}). The action of $\pi_1(M)$ on $\pi_1(H/F)$ is
the composition $\pi_1(M)\to \pi_0(H) \to \mathrm{Aut}(\pi_1(H/F))$.

Here is a way to construct this second obstruction class. Using
Condition~\ref{cond:H} it is easy to construct a (noncentral)
extension $\pi_1(H_0)\to \overline{H} \to H$. The obstruction to write
the $H$-bundle $\mathsf{P}_H$ as the reduction of a
$\overline{H}$-bundle is a class in $\h^2(M; \pi_1(H_0))$ (cohomology
with local coefficients since the extension is not central) and is the
second obstruction class $o_2(\rho)$.

The next two results describe the behavior of the first obstruction
under covering and when twisting by a representation into the center.

\begin{prop}\label{prop:first_obstr_cover}
  Let $\pi:G'\to G$ be a finite covering of Lie groups. Suppose that a
  representation $\rho: \pi_1(M) \to G'$ is $H'$-Anosov.

  Then the representation $\pi\circ \rho$ is $H= \pi(H')$-Anosov and
  under the natural map \[ \h^1(M, \pi_0(H')) \to \h^1(M, \pi_0(H))
  \]
  the class $o_1(\pi\circ \rho)$ is the image of $o_1(\rho)$.
\end{prop}

\begin{proof}
  By Lemma~\ref{lem:finite_cover} $\pi\circ\rho$ is $H$-Anosov with
  Anosov reduction $\mathsf{P}_{H'} \times_{H'} H$ where
  $\mathsf{P}_{H'}$ is the Anosov reduction of $\rho$. Hence the
  morphism $\pi_1(M)\to \pi_0(H)^{ab}$ defining $o_1(\pi\circ \rho)$
  is the composition $\pi_1(M)\to \pi_0(H')^{ab} \to \pi_0(H)^{ab}$ of
  $o_1(\rho)$ with the natural projection $\pi_0(H')^{ab} \to
  \pi_0(H)^{ab}$.
\end{proof}

\begin{prop}\label{prop:frst_ob_center}
  Let $\rho: \pi_1(M)\to G$ be a $(G,H)$-Anosov representation. Let
  $Z$ be the center of $G$ and $\eta: \pi_1(M)\to Z$ a representation.

  Then the representation $\eta\cdot \rho: \pi_1(M) \to G \sep \gamma
  \mapsto \eta(\gamma)\rho(\gamma)$ is $(G,H)$-Anosov and its first
  obstruction class is \[o_1(\eta\cdot \rho) = o_1(\rho) +
  \overline{\eta}\] with $\overline{\eta}$ being the composition
  $\pi_1(M)\to Z \to \pi_0(H)^{ab}$.
\end{prop}

\begin{proof}
  The representation $(\eta, \rho): \pi_1(M) \to Z \times G$ is Anosov
  with Anosov reduction $\mathsf{P}_\eta \times \mathsf{P}_H$, where
  $\mathsf{P}_H$ is the Anosov reduction of $\rho$ and
  $\mathsf{P}_\eta$ is the flat $Z$-bundle associated with
  $\eta$. Hence, by Lemma~\ref{lem:finite_cover}, the representation
  $\eta \cdot \rho$ is Anosov with reduction $ (\mathsf{P}_\eta \times
  \mathsf{P}_H) / Z$. Thus the obstruction morphism
  $\sigma_1(\eta\cdot \rho)$ (see Remark~\ref{rem:obstr_morphism}) is
  the composition of $(\eta, \sigma_1(\rho))$ with the map $Z \times
  \pi_0(H) \to \pi_0(H)$. The result for the obstruction class follows
  from this description.
\end{proof}

A direct application of these propositions to the study of connected
components is the following result.  To allow more flexibility we will
work with a Galois covering $\overline{M} \to M$ with group $\Gamma$
and consider representations of $\Gamma$. The group $\Gamma$ is a
quotient of $\pi_1(M)$.

\begin{prop}\label{prop:taking_cover}
  Let $\pi:G'\to G$ be a finite covering of connected Lie groups with
  kernel $W$. Let $H' <G'$ be the intersection of two opposite
  parabolic subgroups and $H=\pi(H')$.

  Let $p: \hom(\Gamma,G')\to \hom(\Gamma,G)$ be the map $\rho \mapsto
  \pi\circ\rho$ and let $\mathcal{C}$ be a connected subset of
  $\homHano(\Gamma,G)$.
  
  Suppose that the natural map $\h^1(\Gamma;W) \cong \hom(\Gamma, W)
  \to \h^1(M;\pi_0(H')^{ab}) \cong \hom(\pi_1(M), \pi_0(H')^{ab})$ is
  an injection then
  \begin{itemize}
  \item either $p^{-1}(\mathcal{C})$ is empty,
  \item or $p^{-1}(\mathcal{C})$ is the union of $|\h^1(\Gamma;W)|$
    components, that are distinguished by the first obstruction
    $o_1$. The image of $o_1: p^{-1}(\mathcal{C}) \to
    \h^1(M;\pi_0(H')^{ab})$ is one coset for the subgroup
    $\h^1(\Gamma;W)$.
  \end{itemize}
\end{prop}

\begin{proof}
  If $p^{-1}(\mathcal{C})$ is nonempty then by connectedness every
  representation in $\mathcal{C}$ lifts to $\hom(\Gamma, G')$. Thus
  $p^{-1}(\mathcal{C}) \to \mathcal{C}$ is a finite covering. Note
  that the fiber containing a representation $\rho$ is precisely the
  set $\eta\cdot \rho$ for $\eta: \Gamma \to W$. Hence this covering
  is of degree $|\h^1(\Gamma;W)|$. So $p^{-1}(\mathcal{C})$ has at
  most that number of components.

  Note also that $p^{-1}(\mathcal{C})$ is included in
  $\hom_{H'\text{-Anosov}}(\Gamma, G')$. Therefore the obstruction
  class defines a map
  \[o_1: p^{-1}(\mathcal{C}) \longto \h^1(M; \pi_0(H')^{ab}).\] Using
  Proposition~\ref{prop:frst_ob_center}, for any $\rho$ in
  $p^{-1}(\mathcal{C})$ and any $\eta$ in $\h^1(\Gamma;W)$ the two
  classes $o_1(\rho)$ and $o_1(\eta\cdot \rho)$ differ by the image of
  $\eta$ in $\h^1(M; \pi_0(H')^{ab})$. The image of $o_1$ is hence one
  orbit of $\h^1(\Gamma;W)$ in $\h^1(M;\pi_0(H')^{ab})$.  The
  hypothesis that $\h^1(\Gamma;W)$ injects into
  $\h^1(M;\pi_0(H')^{ab})$ insures then that the image of $o_1$ has at
  least $|\h^1(\Gamma;W)|$ elements and hence $p^{-1}(\mathcal{C})$
  has at least that number of components.
\end{proof}

The next lemma will provide an easy criterion for the hypothesis of
the above proposition to be satisfied.

\begin{lem}\label{lem:inject}
  Let $\Sigma$ be a closed oriented connected surface. Suppose that
  $A\to B$ is an injective morphism of abelian group. Then the map
  \[ \h^1(\Sigma; A) \longto \h^1(T^1 \Sigma; B)\] is injective.
\end{lem}

\begin{proof}
  The map between the fundamental groups $\pi_1(T^1 \Sigma) \to
  \pi_1(\Sigma)$ is surjective. Together with the injectivity of $A\to
  B$ this implies the injectivity of $\h^1(\Sigma; A) \cong
  \hom(\pi_1(\Sigma), A) \to \h^1(T^1 \Sigma; B) \cong \hom(\pi_1(T^1
  \Sigma), B)$.
\end{proof}

\subsection{First and second Stiefel-Whitney classes}
\label{sec_fst_scd_SW}

The inclusion
\[\hommax( \pi_1(\Sigma), \Sp(2n,\RR)) \subset \homGLano(
\pi_1(\Sigma), \Sp(2n,\RR)), \] which is described in
Section~\ref{sec:prelim_maxrep_sp}, allows us to apply
Proposition~\ref{prop:invariants} to associate to a maximal
representation into $\Sp(2n,\RR)$ the first and second Stiefel-Whitney
classes of a $\GL(n, \RR)$-bundle over $T^1\Sigma$.

\begin{prop}\label{prop:invariants_symp}
  \label{prop:swformax}
  Let $G= \Sp(2n, \RR)$ and $\hommax( \pi_1(\Sigma),G)$ the space of
  maximal representations. Then the obstruction classes of the Anosov
  $\GL(n,\RR)$-reduction give maps:
  \[
  \begin{array}{rcl}
    \hommax( \pi_1(\Sigma),G) & \xrightarrow{sw_1} & \h^1( T^1
    \Sigma; \FF_2 ) \\
    \textup{and } \hommax( \pi_1(\Sigma),G) & \xrightarrow{sw_2} &
    \h^2( T^1 \Sigma; \FF_2 ).
  \end{array}\]
\end{prop}

The following geometric interpretation makes the first Stiefel-Whitney
class $sw_1(\rho)$ easy to compute.  Recall that given a smooth closed
curve $\gamma$ on $\Sigma$, taking the velocity vector at every point
defines a natural lift to a loop on $T^1\Sigma$; this gives rise to a
natural map $\pi_1(\Sigma)\moins\{1\} \to \h_1(T^1\Sigma; \ZZ)\sep{}
\gamma \mapsto [\gamma]$.

\begin{lem}\label{lem:sw1=sign}
  Let $\rho: \pi_1(\Sigma) \to \Sp(2n,\RR)$ be a maximal
  representation and $\xi: \partial \pi_1(\Sigma) \to \Ll$ the
  equivariant limit curve.  Then
  \[sw_1(\rho)([\gamma]) = \mathrm{sign}( \det
  \rho(\gamma)|_{\xi(t^{s}_{\gamma})}),\] where $t^{s}_{\gamma}
  \in \partial\pi_1(\Sigma)$ is the attractive fixed point of $\gamma$
  and $\FF_2$ is identified with $\{\pm 1\}$.
\end{lem}
\begin{proof}
  The first Stiefel-Whitney class $sw_1(\rho)([\gamma])$ is the
  obstruction to the orientability of the bundle $L^s(\rho)|_\g \cong
  \ZZ \backslash (\RR \times \xi(t^{s}_{\gamma}))$ over $\gamma \cong
  \ZZ \backslash \RR$.  Thus, $ sw_1(\rho)([\gamma]) = 1$ if
  $\rho(\gamma) \in \GL(\xi(t^{s}_{\gamma}))$ has positive determinant
  and $ sw_1(\rho)([\gamma]) = -1$ otherwise.
\end{proof}

\subsection{An Euler class}\label{sec:sp4case}
For $n=2$, the invariants obtained in Proposition~\ref{prop:swformax}
do not allow to distinguish the connected components of
\[\hommax(\pi_1(\Sigma), \Sp(4,\RR)) \moins
\homHit(\pi_1(\Sigma), \Sp(4,\RR)).\] In particular, the second
Stiefel-Whitney class does not offer enough information to distinguish
the connected components of
\[\homMaxFourZero.\]
The reason for this is that when $sw_1 (\rho) = 0$ the Lagrangian
bundle $L^s(\rho)$ is orientable.  There is an Euler class whose image
in $\h^2(T^1\Sigma; \FF_2)$ is the second Stiefel-Whitney class
$sw_2(\rho)$. Since an \emph{orientable} vector bundle has no
\emph{canonical orientation}, we need to introduce an \emph{enhanced}
representation space to obtain a well-defined Euler class.

\subsubsection{Enhanced representation spaces}
\label{sec:thespaces}
Let $\rho: \pi_1(\Sigma) \to \Sp(4,\RR)$ be a maximal representation.
Let $\gamma\in \pi_1( \Sigma)\moins\{1\}$ and $L^s(\gamma) \in \Ll$
the attractive fixed Lagrangian of $\rho(\gamma)$.  We denote by
$\Ll_+$ the space of \emph{oriented} Lagrangians and by $\pi: \Ll_+
\to \Ll$ the projection.  Let
\begin{multline*}
  \homMaxLl = \\
  \big\{ (\rho, L) \in \homMaxFourZero \times \Ll \mid L = L^s
  (\gamma) \big\}
\end{multline*}
and
\begin{multline*}
  \homMaxLlPlus =\\
  \big\{ (\rho, L_+) \in \homMaxFourZero \times \Ll_+ \mid \pi(L_+) =
  L^s(\gamma) \big\}.
\end{multline*}
The map $\homMaxLlPlus \to \homMaxLl$ is a $2$-fold cover.

The space $\homMaxLl$ is easily identified with the space
$\homMaxFourZero$, we introduce it to emphasize the fact that $\rho(
\gamma)$ has an attractive Lagrangian.

\begin{lem}
  The natural map
  \begin{equation*}
    \begin{array}{rcl}
      \homMaxLl
      & \longrightarrow & \Ll \\
      (\rho, L) & \longmapsto & L
    \end{array}
  \end{equation*}
  is continuous.
\end{lem}

This lemma follows immediately from the continuity of the eigenspace
of a matrix or from Proposition~\ref{prop:openano}.

\subsubsection{The Euler class}
\label{sec:eulerclassrank2}
As a consequence of the following proposition for every element
$(\rho, L_+)$ in $\homMaxLlPlusNoGrp$ there is a natural associated
\emph{oriented} Lagrangian bundle over $T^1\Sigma$.

\begin{prop}\label{prop:lift}
  Let $\rho\in \hommax( \pi_1( \Sigma), \Sp(2n, \RR))$ and let
  $\xi: \partial \pi_1( \Sigma) \to \Ll$ be the corresponding
  equivariant positive curve. Suppose that $n$ is even.  Then there
  exists a continuous lift of $\xi$ to $\Ll_+$.

  Let $\xi_+ : \partial \pi_1( \Sigma) \to \Ll_+$ be one of the two
  continuous lifts of $\xi$.  Then the map $\xi_+$ is
  $\rho$-equivariant if and only if $sw_1 (\rho) =0$. In this case the
  other lift of $\xi$ is also equivariant.
\end{prop}

\begin{proof}
  The fact that a continuous lift exists depends only on the homotopy
  class of the curve $\xi$. Since the space of continuous and positive
  curves is connected (see Proposition~\ref{prop:conn-posit-curv}),
  the existence of a lift can be checked for one specific maximal
  representation.  Considering a Fuchsian representation gives the
  restriction on $n$.

  Let $\xi_+$ be a continuous lift of $\xi$ and denote by
  $\xi^{\prime}_{+}$ the other lift of the curve $\xi$. Since $\xi$ is
  equivariant, for any $\gamma$ in $\pi_1( \Sigma)$, the following
  alternative holds:
  \[ \gamma \cdot \xi_+ = \xi_+ \text{ or } \gamma \cdot \xi_+ =
  \xi^{\prime}_{+}. \] Furthermore, given any $t$ in $\partial \pi_1(
  \Sigma)$, by connectedness one observes that
  \[ \gamma \cdot \xi_+ = \xi_+ \Longleftrightarrow \xi_+( \gamma
  \cdot t) = \rho(\gamma) \cdot \xi_+(t). \] Using this last equation
  for the attractive fixed point $t^{s}_{\gamma}$ of $\gamma$ in
  $\partial \pi_1 ( \Sigma)$, the equivalence becomes
  \[ \gamma \cdot \xi_+ = \xi_+ \Longleftrightarrow \det
  \rho(\gamma)|_{\xi(t^{s}_{\gamma})} > 0. \] Using now the equality
  $sw_1(\rho)([\gamma]) = \mathrm{sign}( \det
  \rho(\gamma)|_{\xi(t^{s}_{\gamma})})$ (Lemma~\ref{lem:sw1=sign}),
  the proposition follows.
\end{proof}

Proposition~\ref{prop:lift} gives a natural way to lift the
equivariant curve given any element $(\rho, L_+)\in \homMaxLlPlus$:

\begin{defi}\label{defi:canocurve}
  Let $(\rho, L_+) \in \homMaxLlPlus$ and let $\xi : \partial \pi_1(
  \Sigma) \to \Ll$ be the $\rho$-equivariant curve.  The lift
  (uniquely determined and equivariant) $\xi_+$ of $\xi$ such that
  \[ \xi_+( t^{s}_{\gamma}) = L_+ \] is called the \emph{canonical
    oriented equivariant curve} for the pair $(\rho, L_+)$. Here
  $t^{s}_{\gamma}\in\partial \pi_1(\Sigma)$ is the attractive fixed
  point of $\gamma$.
\end{defi}
This canonical oriented equivariant curve defines a $\GL^+(2,
\RR)$-reduction of the $\GL(2, \RR)$-Anosov reduction associated with
$\rho$. It also induces a splitting of the flat bundle $E$ associated
with $\rho$
\[ E = L^{s}_{+}(\rho) \oplus L^{u}_{+}(\rho)\] into two oriented
Lagrangian subbundles.  We call this splitting the {\em oriented
  Lagrangian reduction} of $E$.

\begin{defi}
  The Euler class
  \[ \begin{array}{rcl}
    e_{\gamma}:  \homMaxLlPlus & \longrightarrow & \h^2( T^1\Sigma; \ZZ) \\
    (\rho, L_+) & \longmapsto & e_{\gamma}(\rho, L_+)
  \end{array} \]
  is the Euler class of the $\GL^+(2, \RR)$-reduction given by 
  Definition~\ref{defi:canocurve}.
\end{defi}
The map $e_{\gamma}$ is continuous.

\subsubsection{Connected components}
\label{sec:conncomprank2}
We consider now a subspace of $\homMaxLlPlusNoGrp$ by fixing the
oriented Lagrangian. Let $\LZPlus \in \Ll_+$ be an oriented
Lagrangian; we set

\begin{multline*}
  \homMaxPLnot =\\
  \{ ( \rho , L_+ ) \in \homMaxLlPlus \mid L_+ = \LZPlus \}.
\end{multline*}

\begin{lem}\label{lem:fibr_con}
  The natural map
  \[ \homMaxPLnot \longrightarrow \homMaxFourZero/\Sp(4, \RR)\] is
  onto. Its fibers are connected.
\end{lem}

\begin{proof}
  Let $\rho\in \homMaxFourZero$. Since the first Stiefel-Whitney class
  of $\rho$ vanishes, there exists an attracting oriented Lagrangian
  $L_+$ for $\rho( \gamma)$. Let $g \in \Sp( 4, \RR)$ such that $g
  \cdot L_+ = \LZPlus$, then $( g \rho g^{-1}, \LZPlus)$ belongs to
  $\homMaxPLnot$ and projects to $[\rho]$.

  Now, let $( \rho, \LZPlus)\in \homMaxPLnot$. Then the fiber of the
  projection containing $( \rho, \LZPlus)$ is isomorphic to the
  quotient:
  \[
  \{ g \in \Sp(4, \RR) \mid g \cdot \LZPlus = \LZPlus \} / \{ z \in Z(
  \rho) \mid z\cdot \LZPlus = \LZPlus \},
  \]
  where $Z(\rho)$ denotes the centralizer of $\rho(\pi_1(\Sigma))$ in
  $\Sp(4,\RR)$.  The group $\{ g \in \Sp(4, \RR) \mid g \cdot \LZPlus
  = \LZPlus \}$ is connected. Thus the fiber is connected.
\end{proof}

Lemma~\ref{lem:fibr_con} implies that $\homMaxPLnot$ has as the same
number of connected components as $\homMaxFourZero$.

\begin{defi}
  Let $\gamma\in \pi_1(\Sigma) \moins \{1\}$ and $\LZPlus \in
  \Ll_+$. The \emph{relative Euler class} $e_{\gamma,
    \LZPlus}(\rho)\in \h^2( T^1\Sigma; \ZZ)$ of the class of a maximal
  representation $\rho \in \homMaxFourZero$ is defined as the Euler
  class of one (any) inverse image of $[\rho]$ in the space
  $\homMaxPLnot$.
\end{defi}

\subsection{Constraints on invariants} \label{sec:constraints} Since
the invariants are constructed from bundles that are flat along
geodesic leaves, one expects that not every cohomology class can
arise:
\begin{prop}
  \label{prop:restriction}
  Let $G = \Sp(2n, \RR)$. Then

  \begin{asparaenum}
  \item The image of $sw_1 : \hommax( \pi_1(\Sigma),G) \rightarrow
    \h^1(T^1 \Sigma ; \FF_2 )$ is contained in one coset of
    $\h^1(\Sigma ; \FF_2 )$. More precisely:
    \begin{itemize}
    \item for $n$ even, $sw_1( \rho) \in \h^1( \Sigma; \FF_2) \subset
      \h^1( T^1\Sigma; \FF_2)$,
    \item for $n$ odd, $sw_1( \rho) \in \h^1( T^1\Sigma; \FF_2) \moins
      \h^1( \Sigma; \FF_2) $.
    \end{itemize}
  \item The image of $ sw_2 : \hommax( \pi_1(\Sigma),G) \rightarrow
    \h^2(T^1 \Sigma ; \FF_2 )$ lies in the image of $\h^2( \Sigma ;
    \FF_2 ) \to \h^2(T^1 \Sigma ; \FF_2 )$.
  \item Similarly, when $n=2$, let $\gamma \in \pi_1(\Sigma)$ and
    $\LZPlus \in \Ll_+$, then the image of $e_{\gamma, \LZPlus} :
    \hommaxzero( \pi_1(\Sigma), \Sp(4, \RR)) \rightarrow \h^2(T^1
    \Sigma ; \ZZ )$ lies in the image of $\h^2( \Sigma ; \ZZ ) \to
    \h^2(T^1 \Sigma ; \ZZ )$.
  \end{asparaenum}
\end{prop}

To prove this proposition we use the positivity of the equivariant
curve $\xi : \partial \pi_1(\Sigma) \to \Ll$ of a maximal
representation $\rho: \pi_1(\Sigma) \to \Sp(2n,\RR)$
(Definition~\ref{def:positivecurve}).

\begin{proof}
  For the first property, let $x\in \Sigma$ and consider the exact
  sequence
  \[
  \h^1(\Sigma; \FF_2) \longrightarrow \h^1( T^1 \Sigma; \FF_2)
  \overset{ f_x}{\longrightarrow} \h^1( T_{x}^{1} \Sigma; \FF_2).\] We
  need to check that the image of $sw_1(\rho)$ in $\h^1( T_{x}^{1}
  \Sigma; \FF_2)\cong \FF_2$ does not depend on $\rho$, i.e.\ that the
  gauge isomorphism class of $L^s(\rho)|_{T_{x}^{1} \Sigma}$ is
  independent of $\rho$.

  Let $\tilde{x}\in \widetilde{ \Sigma}$ be a lift of $x$ so that
  $T_{x}^{1} \Sigma \cong T_{\tilde{x}}^{1} \widetilde{\Sigma}
  \cong \partial \pi_1(\Sigma)$, where the last map is given by the
  restriction of $\partial \pi_1(\Sigma)^{(3+)} \to \partial
  \pi_1(\Sigma)\sep{} (t^s,t,t^u) \mapsto t^s$ to $T_{\tilde{x}}^{1}
  \widetilde{\Sigma} \subset T^{1} \widetilde{\Sigma} \cong \partial
  \pi_1(\Sigma)^{(3+)}$. The restriction of the flat $G$-bundle to
  $T_{x}^{1} \Sigma$ is trivial, thus the restriction of the Anosov
  section $\sigma$ to $T_{x}^{1} \Sigma$ can be regarded as a map
  $T_{x}^{1} \Sigma \to \mathcal{X}$.

  By Remark~\ref{rem:reduc-P}, we can work with the $P^s$-reduction,
  i.e.\ we consider only the first component of the map $\sigma|_{
    T_{x}^{1} \Sigma} : T_{x}^{1} \Sigma \to \Ll \times \Ll$.  With
  the above identification $T_{x}^{1} \Sigma \cong \partial
  \pi_1(\Sigma)$ this map is exactly the equivariant limit curve $\xi
  : \partial \pi_1(\Sigma) \to \Ll$.  Since the space of positive
  curves is connected (Proposition~\ref{prop:conn-posit-curv}), $f_x(
  sw_1( \rho) )$ is independent of $\rho$.  The calculation for the
  diagonal Fuchsian representation (see
  Section~\ref{subsec:comp_diag}) gives the desired statement about
  $sw_1(\rho)$.

  \medskip

  For the second statement, in view of
  Lemma~\ref{lem:whenTwoClassIsTor} it is sufficient to prove that for
  any closed curve $\eta \subset \Sigma$ the restriction of $sw_2
  (\rho)$ (or $e_{\gamma, \LZPlus}( \rho)$) to the torus $ T^1
  \Sigma|_{ \eta}$ is zero.

  For this we write the torus $ T^1 \Sigma|_{ \eta}$ as the quotient
  of $\partial \pi_1(\Sigma) \times \RR$ by $\langle \eta \rangle
  \cong \ZZ$ where $\eta \cdot ( \theta, t) = ( \eta \cdot \theta,
  t+1)$.  With this identification the flat bundle can be written as
  $\langle \g \rangle \backslash \big (\partial \pi_1(\Sigma) \times
  \RR \times G \big)$ with $\g \cdot ( \theta, t, g) = ( \gamma \cdot
  \theta, t+1, A g)$ and $A \in H \cong \GL(n, \RR)$. The section of
  the associated $\Ll$-bundle lifts to $\partial \pi_1(\Sigma) \times
  \RR$:
  \[
  \begin{array}{rcl}
    \partial \pi_1(\Sigma) \times \RR & \longrightarrow & \Ll \\
    ( \theta, t) & \longmapsto & \xi( \theta).
  \end{array}
  \]
 
  Since the space of pairs \[\big\{ (A, \xi) \mid A \in H, \xi \text{
    positive, continuous and } A\text{-equivariant} \big\}\] has
  exactly two connected components given by the sign of $\det A$
  (Proposition~\ref{prop:conn_equiv_posit}), we conclude that
  $sw_2(L^s( \rho) |_{T^1 \Sigma|_{ \eta}})$ depends only on this
  sign, hence $sw_2(\rho)$ depends only on $sw_1( \rho)$ (see
  Lemma~\ref{lem:sw1=sign}).  Calculations for the model
  representations show that in fact $sw_2(L^s( \rho) |_{T^1 \Sigma|_{
      \eta}})$ is always zero.

  The proof for the Euler class proceeds along the same lines.
\end{proof}

In view of Proposition~\ref{prop:restriction} we make the following
\begin{defi}\label{defi:admissible}
  A pair \[(\alpha, \beta) \in \h^1(T^1\Sigma; \FF_2)\times
  \h^2(T^1\Sigma; \FF_2)\] is called $n$-\emph{admissible} if
  \begin{itemize}
  \item $n$ is even and $\alpha$ lies in the image of $\h^1( \Sigma ;
    \FF_2 ) \to \h^1(T^1 \Sigma ; \FF_2 )$ and $\beta$ lies in the
    image of $\h^2( \Sigma ; \FF_2 ) \to \h^2(T^1 \Sigma ; \FF_2 )$.
  \item $n$ is odd and $\alpha$ lies in $\h^1(T^1 \Sigma ; \FF_2 )
    \moins \h^1( \Sigma ; \FF_2 )$, and $\beta$ lies in the image of
    $\h^2( \Sigma ; \FF_2 ) \to \h^2(T^1 \Sigma ; \FF_2 )$.
  \end{itemize}
  For $n= 2$, a class $\beta\in \h^2(T^1\Sigma; \ZZ)$ is called
  $2$-admissible if $\beta$ lies in the image of $\h^2( \Sigma ; \ZZ )
  \to \h^2(T^1 \Sigma ; \ZZ )$.
\end{defi}

\subsection{Invariants for other Anosov representations}
\label{sec_inv_others}

\subsubsection{Maximal representations into $\SL(2,\RR)$}
Let $\iota: \pi_1(\Sigma) \to \SL(2,\RR) = \Sp(2,\RR)$ be a maximal
representation. The Anosov section associated with $\rho$ gives a
reduction of the structure group of the flat $\SL(2,\RR)$-principal
bundle $\mathsf{P}$ over $T^1\Sigma$ to $\GL(1,\RR)$. Considering the
$\RR^2$-bundle $E$ associated with $\mathsf{P}$, this reduction
corresponds to a Lagrangian subbundle $L$ of $E$.  As invariants we
get the first Stiefel-Whitney class of the line bundle $L$ over
$T^1\Sigma$,
\[
sw_1(\iota) \in \h^1(T^1\Sigma; \FF_2) \moins\h^1(\Sigma; \FF_2).
\]

This first Stiefel-Whitney class is precisely the spin structure on
$\Sigma$ which corresponds to the chosen lift of $\pi_1(\Sigma)
\subset \PSL(2,\RR)$ to $\SL(2,\RR)$. The invariant $sw_1(\iota)$ can
take $2^{2g}$ different values, distinguishing the $2^{2g}$ connected
components of $\hommax(\pi_1(\Sigma),\SL(2,\RR))$.

\subsubsection{The invariants for Hitchin representations into
  $\Sp(2n,\RR)$}\label{sec:invforHitchin} 
Let $\rho:\pi_1(\Sigma) \to \Sp(2n,\RR)$ be a Hitchin
representation. Then $\rho$ is indeed a $(\Sp(2n,\RR), A)$-Anosov
representation, where $A$ is the subgroup of diagonal matrices
\cite[Theorems~4.1, 4.2]{Labourie_anosov}. The reduction of the
structure group of the flat $\Sp(2n,\RR)$-principal bundle
$\mathsf{P}$ over $T^1\Sigma$ to $A$ corresponds to a splitting of the
associated $\RR^{2n}$-bundle $E$ into the sum of $2n$ isomorphic line
bundles $F_1 \oplus \cdots \oplus F_{2n}$.  The first Stiefel-Whitney
class of the line bundle $F_1$ gives an invariant
\[
sw^A_1(\rho) \in \h^1(T^1\Sigma; \FF_2),
\]
which lies in $\h^1(T^1\Sigma; \FF_2) \moins \h^1(\Sigma; \FF_2)$.
This invariant distinguishes the $2^{2g}$ components of
$\homHit(\pi_1(\Sigma), \Sp(2n,\RR))$.

\subsubsection{General Hitchin components}
\label{sec:gener-hitch-comp}

Let $G^{ad}$ a adjoint split real simple Lie group. The principal
subgroup is a distinguished $ \sigma: \PSL(2, \RR) \to G^{ad}$ that
generalizes the $n$-dimensional irreducible $\PSL(2, \RR)$ in $\PSL(n,
\RR)$. The Hitchin component $\homHit(\pi_1(\Sigma), G^{ad})$ for
$G^{ad}$ is the component of $\hom(\pi_1(\Sigma), G^{ad})$ that
contains uniformizations $\pi_1(\Sigma) \to \PSL(2, \RR)
\xrightarrow{\sigma} G^{ad}$. We refer to Hitchin's paper
\cite{Hitchin} for details. Representations in $\homHit(\pi_1(\Sigma),
G^{ad})$ are called Hitchin representations.

Applying the result \cite[Theorem~1.15]{Fock_Goncharov} one obtains
that Hitchin representations are $H$-Anosov where $H$ is the Levi
component of the Borel subgroup. In fact (by connectedness) the Anosov
reduction $P_{H}$ always admit a reduction to the finite index
subgroup $H_0$. Extending slightly the terminology we will say that
the representations are $H_0$-Anosov.

Now if $G$ is a connected cover of $G^{ad}$ so that the center $Z$ of
$G$ is a quotient of $\pi_1(G^{ad})$: $Z= \pi_1(G^{ad})/ \pi_1(G)$,
one can ask when Hitchin representations lift and how many components
there are above the Hitchin component $\homHit(\pi_1(\Sigma),
G^{ad})$. For this let us denote $\sigma_*: \ZZ \cong \pi_1(\PSL(2,
\RR)) \to \pi_1(G^{ad})$ the morphism induced by the injection of
principal $\PSL(2, \RR)$. Applying
Propositions~\ref{prop:taking_cover} and \ref{prop:frst_ob_center} one
gets

\begin{thm}\label{thm_Hitchin_cover}
  Let $\pi:G\to G^{ad}$ be a finite connected covering of a split real
  adjoint simple Lie group with $\ker \pi=Z$ and let $p:
  \hom(\pi_1(\Sigma), G) \to \hom(\pi_1(\Sigma), G^{ad})$ the induced
  map. The group $\pi^{-1}(H_0)$ is (isomorphic to) the product $Z H_0
  = Z \times H_0$.

  The Hitchin component $\mathcal{H}$ of $G^{ad}$ lifts if and only if
  $\sigma_*(\chi(\Sigma))$ is in $\pi_1(G) <
  \pi_1(G^{ad})$. Furthermore, in that case,
  $\homHit(\pi_1(\Sigma),G)=p^{-1}(\mathcal{H})$ is included in the
  space $\hom_{ZH_0\text{-Anosov}}(\pi_1(\Sigma), G)$ and the map
  \[ o_1 : \homHit(\pi_1(\Sigma),G) \longto \h^1(T^1 \Sigma;
  \pi_0(ZH_0)) = \h^1(T^1 \Sigma; Z)\] induces a bijection between the
  components of $\homHit(\pi_1(\Sigma),G)$ and one orbit of
  $\h^1(\Sigma; Z)$. This orbit is the inverse image of $\sigma_*(1)$
  by the map $\h^1(T^1 \Sigma; Z) \to \h^1(T^{1}_{x} \Sigma; Z)\cong
  Z$. In particular there are $|Z|^{2g}$ components.
\end{thm}

\begin{proof}
  The obstruction to lifting $\rho: \pi_1(\Sigma) \to G^{ad}$ to the
  cover $G \to G^{ad}$ lies in $\h^2(\Sigma; \pi_1(G^{ad})/ \pi_1(G))
  \cong \pi_1(G^{ad})/ \pi_1(G)$ and is the image of $o(\rho) \in
  \pi_1(G^{ad})$ the obstruction to lifting $\rho$ to the universal
  cover of $G^{ad}$.  For a representation $ \pi_1(\Sigma)
  \xrightarrow{\iota} \PSL(2, \RR) \xrightarrow{\sigma} G^{ad}$ the
  obstruction class $o(\sigma \circ \iota)$ is the image by $\sigma_*:
  \ZZ \to \pi_1(G^{ad})$ of the obstruction class $o(\iota)=
  \chi(\Sigma)$, i.e.\ $ o(\sigma \circ \iota) =
  \sigma_*(\chi(\Sigma))$. This gives the condition on
  $\sigma_*(\chi(\Sigma))$.

  By Proposition~\ref{prop:taking_cover} and Lemma~\ref{lem:inject}
  the image of $o_1$ in $\h^1(T^1 \Sigma; Z)$ is one coset of the
  subgroup $\h^1(\Sigma; Z)$. To determine this coset observe first
  that the sequence $0\to \h^1(\Sigma; Z) \to \h^1(T^1 \Sigma; Z)
  \xrightarrow{i^*} \h^1(T^{1}_{x} \Sigma; Z) \to 0$ is exact, where
  $i: T^{1}_{x} \Sigma \to T^{1} \Sigma$ is the injection; this
  follows by example from the fact that the extension $\pi_1(T^1
  \Sigma) \to \pi_1(\Sigma)$ is central with kernel generated by
  $\pi_1(T^{1}_{x} \Sigma)$. Therefore it is enough to determine
  $i^*(o_1(\rho))$ for one particular representation $\rho$.  Let thus
  chose a representation $\rho: \pi_1(\Sigma) \to G$ lifting a
  representation of the form $\sigma \circ \iota$ where $\iota :
  \pi_1(\Sigma) \to \PSL(2, \RR)$ is a discrete embedding. The unit
  tangent bundle $T^1 \Sigma$ is then identified with $\iota(
  \pi_1(\Sigma)) \backslash \PSL(2, \RR)$.  The Anosov reduction
  associated with $\sigma \circ \iota$ is the equivariant map \[
  \begin{array}{rcl}
    \PSL(2, \RR) & \longto & G^{ad}/H_0 \\
    g & \longmapsto & \sigma(g) H_0,
  \end{array}
  \]
  and the Anosov reduction for $\rho$ is
  \[
  \begin{array}{rcl}
    \PSL(2, \RR) & \longto & G/ZH_0 \\
    g & \longmapsto & \pi^{-1}(\sigma(g) H_0).
  \end{array}
  \]
  The image of $o_1(\rho)$ in $\h^1(T^{1}_{x} \Sigma;Z)$ is the
  holonomy of the corresponding principal $\pi_0(ZH_0) \cong Z$-bundle
  restricted to $T^{1}_{x} \Sigma \cong \PSO(2) \subset \PSL(2,
  \RR)$. It is calculated as follow: Let $(g_t)_{t\in [0,1]}$ a loop
  generating $\pi_1(\PSO(2)) = \pi_1( \PSL(2, \RR))$ and let $(h_t)_{t
    \in [0,1]}$ be a continuous curve such that, for all $t$, $h_t$
  belongs to $\pi^{-1}(\sigma(g_t) H_0)$. Then $h_1 h_{0}^{-1}$ is in
  $ZH_0$ and the sought for holonomy is the projection onto $Z$ of
  this element. However one can chose $h_t \in
  \pi^{-1}(\sigma(g_t))$. For this choice $h_1 h_{0}^{-1}$ is in $Z$
  and is the image of the loop $(\sigma(g_t))_{t \in [0,1]}$ under the
  natural map $\pi_1(G^{ad}) \to Z$. Since the loop $(\sigma(g_t))$
  represents precisely $\sigma_*(1)$, the result follows.
\end{proof}

For classical split Lie groups the obstruction class obtained is the
first Stiefel-Whitney class $sw_{1}$ of a line bundle.  We have:

\begin{thm}\label{thm_Hitchinclassical}
  For $G$ the group $\SL(n,\RR)$, $\Sp(n, \RR)$, $\SOcon(n,n)$ or
  $\SOcon(n,n+1)$, the map \[ sw_{1} : \homHit( \pi_1(\Sigma), G)
  \longrightarrow \h^1(T^1\Sigma, \FF_2)\] induces a bijection between
  $\pi_0( \homHit)$ and $\h^1(T^1\Sigma, \FF_2) \moins \h^1(\Sigma,
  \FF_2)$.

  Furthermore, for any $\rho$ in $\homHit$ and any nontrivial $\g$ in
  $\pi_1( \Sigma)$, $sw_{1}(\rho)( [\g])$ in $\FF_2\cong \{ \pm 1\}$
  is the common sign of the eigenvalues of $\rho( \gamma)$.
\end{thm}

\subsubsection{Maximal representations into covers of $\Sp(2n, \RR)$}
\label{sec:cover_sp}

The discussion of Section~\ref{sec:char_class_in_gal} can be used to
analyze maximal representations into covers of $\Sp(2n, \RR)$.

Let $\Gk$ be the $k$-fold connected cover of $\Sp(2n, \RR)$, i.e.\ it
corresponds to the subgroup $k\ZZ \subset \ZZ \cong \pi_1(\Sp(2n,
\RR))$. The kernel $Z_k$ of the projection $\pi : \Gk \to \Sp(2n,
\RR)$ is isomorphic to $\ZZ/ k\ZZ$. The subgroups $H_{(k)} =
\pi^{-1}(\GL(n, \RR) )$ and $H^{+}_{(k)} = \pi^{-1}(\GL^+(n, \RR))$
are $k$-fold covers of $\GL(n, \RR)$ and $\GL^+(n, \RR)$.

\begin{lem}\label{lem:coverGk}
  The covers $H^{+}_{(k)} \to \GL^+(n, \RR)$ and $H_{(k)} \to \GL(n,
  \RR)$ are trivial. The natural map
  \[ Z_k \longto \pi_0(H^{+}_{(k)})\] is an isomorphism. The map $Z_k
  \to \pi_0(H_{(k)})$ is injective with image a subgroup of index $2$.
  The group $\pi_0(H_{(k)})$ is isomorphic to $\ZZ/ 2k\ZZ$ when $n$ is
  odd and to $\ZZ/k\ZZ \times \ZZ/2\ZZ$ when $n$ is even.
\end{lem}

\begin{proof}
  It is enough to do the calculation for $H^{+}_{(\infty)}$ the
  pullback of $\GL^+(n, \RR)$ in the universal cover of $\Sp(2n,
  \RR)$. The identity component of $H^{+}_{(\infty)}$ is then the
  cover of $\GL^+(n, \RR)$ corresponding to the subgroup of
  $\pi_1(\GL^+(n, \RR))$ which is the kernel of the map
  $\pi_1(\GL^+(n, \RR)) \to \pi_1(\Sp(2n, \RR))$. Since this map is
  zero the identity component of $H^{+}_{(\infty)}$ is isomorphic to
  $\GL^+(n, \RR)$ and the cover $H^{+}_{(\infty)} \to \GL^+(n, \RR)$
  is the trivial cover. Hence $\pi_0(H^{+}_{(\infty)})$ is isomorphic
  to $Z \cong \pi_1(\Sp(2n, \RR))$ the kernel of $\widetilde{\Sp}(2n,
  \RR) \to \Sp(2n, \RR)$.

  Clearly $\pi_0(H_{(k)}) /\pi_0(H^{+}_{(k)}) \cong \pi_0(\GL(n,
  \RR))$ so $\pi_0(H^{+}_{(k)})$ is the subgroup of index $2$ in
  $\pi_0(H_{(k)})$. Determining the latter group is then easy.
\end{proof}

The characteristic classes for a $(\Gk, H_{(k)})$-Anosov
representation $\rho$ are then the obstruction classes $o_1(\rho)$ in
$\h^1(T^1 \Sigma, \pi_0(H_{(k)}))$ and $o_2(\rho)$ in $\h^2(T^1
\Sigma, \pi_1(H_{(k)0})) \cong \h^2(T^1 \Sigma, \FF_2)$. They project
to the first and second Stiefel-Whitney classes $sw_1(\pi \circ \rho)$
and $sw_2(\pi \circ \rho)$ (in fact $o_2(\rho) = sw_2(\pi \circ
\rho)$). A pair $(c_1,c_2)$ in $\h^1(T^1 \Sigma, \pi_0(H_{(k)}))
\times \h^2(T^1 \Sigma, \pi_1(H_{(k)0}))$ is called
$n$\emph{-admissible} if it projects to a $n$-admissible pair in
$\h^1(T^1 \Sigma, \FF_2) \times \h^2(T^1 \Sigma, \FF_2 )$
(Definition~\ref{defi:admissible}).

\begin{thm}
  The space $\hommax(\pi_1(\Sigma), \Gk)$ is nonempty if and only if
  $k$ divides $n(g(\Sigma)-1)$. In that case the map
  \begin{multline*}
    (o_1,o_2) : \pi_0\big( \hommax(\pi_1(\Sigma), \Gk) \moins
    \homHit(\pi_1(\Sigma), \Gk) \big) \\ \longrightarrow
    \h^1(T^1\Sigma; \pi_0(H_{(k)})) \times \h^2(T^1\Sigma;
    \pi_1(H_{(k)0}))
  \end{multline*}
  is the bijection onto the space of $n$-admissible pairs. The number
  of components of maximal representations is then:
  \[ \#\pi_o(\hommax(\pi_1(\Sigma), \Gk)) = \left \{
    \begin{array}{ll}
      (3\times 2^{2g} + 2g(\Sigma) - 4)k^{2g} & -\text{if }n=2,\\
      3(2k)^{2g}  & -\text{if }n>2.
    \end{array}
  \right.\]
\end{thm}

\begin{proof}
  The obstruction to lifting a representation $\rho: \pi_1(\Sigma) \to
  \Sp(2n, \RR)$ to the cover $\Gk \to \Sp(2n, \RR)$ is the image of
  the Toledo number $\tau(\rho)$ by the map $\ZZ \cong \h^2(\Sigma;
  \pi_1(\Sp(2n;\RR))) \to \h^2(\Sigma; Z_k ) \cong \ZZ/ k\ZZ $. Hence
  $\rho$ lifts if and only if $\tau(\rho)$ is in $k\ZZ$. Since maximal
  representations have Toledo number $n(g-1)$, this gives the
  restriction on $k$.

  Lemma~\ref{lem:coverGk} and Lemma~\ref{lem:inject} imply that the
  map $\h^1(\Sigma; Z_k) \to \h^1(T^1 \Sigma; \pi_0(H_{(k)}))$ is
  injective. The result follows then from
  Proposition~\ref{prop:taking_cover}.
\end{proof}

It is sometimes possible to calculate the number of components of
maximal representations in $\PSp(2n ,\RR)$. Note that
\[ \pi_1(\PSp(2n, \RR)) = \left\{
  \begin{array}{ll}
    \ZZ & -\text{if }n\text{ is odd},\\
    \ZZ\times \ZZ/2\ZZ & -\text{if }n\text{ is even},
  \end{array}
\right.\] and contains $\pi_1(\Sp(2n, \RR))$ as a subgroup of index
$2$. Moreover the bound for the Toledo number $\tau(\rho)$ of a
representation $\rho : \pi_1(\Sigma) \to \PSp(2n,\RR)$ are
\[ \begin{array}{ll}
  |\tau(\rho)| \leq 2n(g-1) & -\text{if }n\text{ is odd},\\
  |\tau(\rho)| \leq n(g-1) & -\text{if }n\text{ is even}.
\end{array} \]
Hence
maximal representations $\pi_1(\Sigma) \to \PSp(2n, \RR)$ always lift
when $n$ is odd. In that case  $\PGL(n,
\RR)$ is connected and its fundamental group is
\[ \pi_1( \PGL(n, \RR)) = \ZZ/2\ZZ .\] Therefore, for a maximal
representation $\rho: \pi_1(\Sigma) \to \PSp(2n, \RR)$ (with $n$ odd),
only the second obstruction class $o_2(\rho)$ in $\h^2(T^1 \Sigma;
\FF_2)$ is relevant. Since $\rho$ lifts to $\rho': \pi_1(\Sigma) \to
\Sp(2n, \RR)$, $o_2(\rho) = sw_2(\rho')$ satisfies the conditions of
Proposition~\ref{prop:restriction}: $o_2(\rho) \in \h^2(\Sigma;
\FF_2)$.

\begin{thm}
  Let $n$ be an odd integer, $n\geq 3$.

  Then the space $\hommax(\pi_1(\Sigma), \PSp(2n, \RR))$ has $3$
  connected components. Furthermore $o_2: \pi_0\big(
  \hommax(\pi_1(\Sigma), \PSp(2n, \RR)) \moins \homHit(\pi_1(\Sigma),
  \PSp(2n, \RR)) \big) \to \h^2(\Sigma; \FF_2)$ is a bijection.
\end{thm}

The proof is straightforward since Proposition~\ref{prop:taking_cover}
applies.

For the case when $n$ is even, one gets

\begin{prop}
  Let $n$ be an even integer, $n\geq 2$.

  Then the image of $\hommax(\pi_1(\Sigma), \Sp(2n, \RR))$ in
  $\hom(\pi_1(\Sigma), \PSp(2n,\RR) )$ is the union of
  \begin{enumerate}
  \item $1 + (2g-2) + (2^{2g} -1)$ connected components of
    $\hommax(\pi_1(\Sigma), \PSp(2n, \RR))$ when $n=2$,
  \item $1 + 2 + (2^{2g} -1)$ components of $\hommax(\pi_1(\Sigma),
    \PSp(2n, \RR))$ when $n>2$.
  \end{enumerate}
\end{prop}

The above theorem and this proposition imply
Theorem~\ref{thm_intro_comp_PSp} of the introduction.

\begin{proof}
  All the components of $\homHit( \pi_1(\Sigma), \Sp(2n, \RR))$
  projects to the same component in $\hommax( \pi_1(\Sigma), \PSp(2n,
  \RR))$.  Let $\mathcal{C}$ be a connected component of
  $\hommax(\pi_1(\Sigma), \Sp(2n, \RR)) \moins \homHit(\pi_1(\Sigma),
  \Sp(2n, \RR))$. We need to understand how many components are
  identified with $\mathcal{C}$ under the projection $p$, i.e.\ the
  number of components of $\mathcal{D} = p^{-1}(p(\mathcal{C}))$. Note
  that, denoting $Z= \{\pm \id\}$ the center of $\Sp(2n, \RR)$,
  \[\mathcal{D} = \{ \eta \cdot \rho \mid \rho \in \mathcal{C}, \eta
  \in \hom(\pi_1(\Sigma, Z)) \}.\] Let $sw_1$ in $\h^1(T^1\Sigma;
  \FF_2)$ the first Stiefel-Whitney class of any representation in
  $\mathcal{C}$. We shall prove:
  \[ \begin{array}{ll}
    |\pi_0( \mathcal{D})|=2 & -\text{if }sw_1 \neq 0\\
    |\pi_0( \mathcal{D})|=1 & -\text{if }sw_1=0.
  \end{array}
  \]
  By Theorems~\ref{thm_intro:symp_maximal} and
  \ref{thm_intro:n=2_components} it is enough to show that, for any
  $\rho \in \mathcal{C}$ and $\eta \in \hom(\pi_1(\Sigma, Z))$,
  \begin{itemize}
  \item $sw_1(\eta\cdot \rho) = sw_1(\rho)$.
  \item if $sw_1(\rho)=0$, $sw_2(\eta\cdot \rho) = sw_2(\rho)$ (or
    $e_\gamma(\eta\cdot \rho) = e_\gamma(\rho)$ when $n=2$).
  \item if $sw_1(\rho)\neq 0$, there exists $\eta'$ such that
    $sw_2(\eta' \cdot \rho) \neq sw_2(\rho)$.
  \end{itemize}
  If $D_\eta$ is the flat line bundle associated with $\eta$, the
  Lagrangian reduction for $\eta\cdot\rho$ is
  \[ L(\eta \cdot \rho) = D_\eta \otimes L(\rho).\] From
  Proposition~\ref{prop_swtensprod} the Stiefel-Whitney classes are
  \[sw_1(\eta\cdot \rho) = sw_1(\rho) + n sw_1(D_\eta) = sw_1(\rho)\]
  since $n$ is even, and
  \[sw_2(\eta\cdot \rho) = sw_2(\rho) + (n-1) sw_1(D_\eta) \cupprod
  sw_1(\rho) + \frac{n(n-1)}{2} sw_1(D_\eta) \cupprod sw_1(D_\eta).\]
  In the above equality every class is in the cohomology of $\Sigma$.
  Since the map $\h^1(\Sigma; \FF_2) \otimes \h^1(\Sigma; \FF_2) \to
  \h^2(\Sigma; \FF_2) \cong \FF_2\sep a\otimes b \mapsto a \cupprod b$
  is a nondegenerate antisymmetric form on $\h^1(\Sigma; \FF_2)$, one
  has $sw_1(D_\eta) \cupprod sw_1(D_\eta)=0$ and thus
  \[sw_2(\eta\cdot \rho) = sw_2(\rho) + sw_1(D_\eta) \cupprod
  sw_1(\rho).\] When $sw_1(\rho) = 0$ this equality become
  $sw_2(\eta\cdot \rho) = sw_2(\rho)$. When $sw_1(\rho)\neq 0$, by
  nondegeneracy, there exists $\eta'$ with $sw_1(D_{\eta'}) \cupprod
  sw_1(\rho) \neq 0$. Therefore $sw_2(\eta'\cdot \rho) \neq
  sw_2(\rho)$.  When $n=2$ and $sw_1(\rho)$ is zero, the Euler classes
  $e_\gamma(\eta\cdot \rho) $ and $e_\gamma(\rho)$ are the Euler
  classes of the $S^1$-bundles $S_{\eta\cdot\rho}$ and $S_\rho$. The
  above formula for $L(\eta \cdot \rho)$ implies that $S_{\eta \cdot
    \rho} = S_\eta \times_{S^1} S_\rho$ where $S_\eta$ is the flat
  $S^1$-bundle associated with $\eta$. Therefore one has $e_\g(\eta
  \cdot \rho) = e(S_\eta) + e_\g(\rho)$. Since $e(S_\eta)$ depends
  continuously on $\eta$ in $\hom(\pi_1(\Sigma), S^1) \cong
  (S^1)^{2g}$, this Euler class is zero.
\end{proof}

\section{Computations of the Invariants}
\label{sec:computations}

In this section we calculate the invariants for the maximal
representations introduced in Section~\ref{sec:examples}. First the
computation is performed for standard maximal representations
(Sections~\ref{subsec:comp_irr}, \ref{subsec:comp_diag} and
\ref{sec:CompTwistedDiag}). In Section~\ref{sec_cons_for_max} we
deduce Theorems~\ref{thm_intro:symp_maximal} and
\ref{thm_intro:symp_n>2} and Corollary~\ref{cor_intro:deform} of the
introduction.  In Section~\ref{sec:sta_res} we calculate the
invariants for the hybrid representations defined in
Section~\ref{sec:description_hybrid}.

\subsection{Irreducible Fuchsian representations}\label{subsec:comp_irr}
Let $\iota: \pi_1(\Sigma) \to \SL(2,\RR)$ be a discrete embedding and
$\rho_{irr} = \phi_{irr} \circ \iota: \pi_1(\Sigma) \to \Sp(2n,\RR)$
an irreducible Fuchsian representation.  We observed in
Fact~\ref{facts:irreducible}.(\ref{item:lag_irr}) that the Lagrangian
reduction $L^s(\rho_{irr})$ is given by
\[
L^s(\rho_{irr}) = L^s(\iota)^{2n-1} \oplus L^s(\iota)^{2n-3} \oplus
\cdots \oplus L^s(\iota),
\]
where $L^s(\iota)$ is the line bundle associated with $\iota$.

Therefore, by the multiplicative properties of Stiefel-Whitney
classes, the first and second Stiefel-Whitney classes are given by
\[
sw_1(\rho_{irr}) = sw_1(L^s(\rho_{irr})) = \big (\sum_{i=1}^n (2i-1)
\big)sw_1(L^s(\iota)) = n sw_1(L^s(\iota)), \] and
\[
\begin{array}{ll}
  sw_2(\rho_{irr}) = sw_2(L^s(\rho_{irr})) &\displaystyle =
  \frac{n(n-1)}{2} sw_1(L^s(\iota))\cupprod sw_1(L^s(\iota)) \\ 
  &\displaystyle=  \frac{n(n-1)}{2} (g-1) \mod 2, 
\end{array} \]
where the last equality follows from the calculation for $n=2$; in
that case $sw_1(L^s(\iota))\cupprod sw_1(L^s(\iota)) =
sw_2(\rho_{irr}) = e_\gamma(\rho_{irr}) \mod 2 = (g-1) \mod 2$ by
Proposition~\ref{prop:eulerHybrid}. 
Note that in particular $sw_1(\rho_{irr}) = 0$ if $n$ is even. 
\subsection{Diagonal Fuchsian representations}\label{subsec:comp_diag}
Let $\iota: \pi_1(\Sigma) \to \SL(2,\RR)$ be a discrete embedding and
$\rho_{\Delta} = \phi_{\Delta} \circ \iota: \pi_1(\Sigma) \to
\Sp(2n,\RR)$ a diagonal Fuchsian representation.  We observed in
Facts~\ref{facts:diagonal}.(\ref{item:lag_diag}) that the Lagrangian
reduction $L^s(\rho_{\Delta})$ is given by
\[
L^s(\rho_\Delta) = L^s(\iota)\oplus \cdots \oplus L^s(\iota),
\]
where $L^s(\iota)$ is the line bundle associated with $\iota$.

Therefore the first and second Stiefel-Whitney classes are given by
\[
sw_1(\rho_{\Delta}) = sw_1(L^s(\rho_{\Delta})) = n
sw_1(L^s(\iota)), \] and
\[
\begin{array}{ll}
  sw_2(\rho_{\Delta}) = sw_2(L^s(\rho_{\Delta})) 
  &=\displaystyle \frac{n(n-1)}{2} sw_1(L^s(\iota))\cupprod sw_1(L^s(\iota)) \\
  &=\displaystyle  \frac{n(n-1)}{2} (g-1) \mod 2. 
\end{array}\]
Again, $sw_1(\rho_{\Delta}) = 0$ if $n$ is even. 

\subsection{Twisted diagonal representations}
\label{sec:CompTwistedDiag}

The aim of this section is to prove the

\begin{prop}\label{prop:realize_invariants} 
  Let $n\geq 2$. Let $(\alpha, \beta) \in \h^1(T^1\Sigma; \FF_2)\times
  \h^2(T^1\Sigma; \FF_2)$ be an $n$-admissible pair (see
  Definition~\ref{defi:admissible}).  If $n=2$ and $\alpha=0$ we
  suppose furthermore that $\beta=(g-1) \mod 2$.

  Then there exists a twisted diagonal representation
  \[
  \rho_{\Theta}: \pi_1(\Sigma) \to \Sp(2n,\RR)
  \]
  such that $sw_1(\rho_{\Theta}) = \alpha$ and $sw_2(\rho_{\Theta}) =
  \beta$.
\end{prop}

\begin{remark}
  Note that when $n =2$ twisted diagonal representations $\rho_\Theta$
  with $sw_1(\rho_\Theta) = 0$ will always have Euler class
  $e_{\gamma, \LZPlus}(\rho_\Theta) = (g-1) \GenTor \in
  \h^2(T^1\Sigma; \ZZ)$ (Proposition~\ref{prop:eulerHybrid}) so that
  the second Stiefel-Whitney class is $sw_2(\rho_\Theta) = g-1 \mod
  2$.
\end{remark}

Let $\iota: \pi_1(\Sigma) \to \SL(2,\RR)$ be a discrete embedding,
$\Theta: \pi_1(\Sigma) \to \O(n)$ an orthogonal representation and
$\rho_\Theta = \iota \otimes \Theta: \pi_1(\Sigma) \to \Sp(2n,\RR)$
the corresponding twisted diagonal representation.  By
Fact~\ref{facts:twisted}.(\ref{item:lag_twisted}) the Lagrangian
reduction $L^s(\rho_{\Theta})$ is given by
\[
L^s(\rho_\Theta) = L^s(\iota)\otimes \overline{W}_\Theta,
\]
where $L^s(\iota)$ is the line bundle associated with $\iota$ and
$\overline{W}_\Theta$ is the pullback to $T^1\Sigma$ of the flat
$n$-plane bundle $W_\Theta$ over $\Sigma$ associated with $\Theta$.

We have $sw_i(\overline{W}_\Theta) = \pi^* sw_i(W_\Theta)$, where
$\pi^*: \h^i(\Sigma; \FF_2) \to \h^i(T^1\Sigma; \FF_2)$ is induced by
the projection $\pi: T^1\Sigma \to \Sigma$.

Thus, to compute the invariants for twisted diagonal representations
we will need to study the first and second Stiefel-Whitney classes of
orthogonal representations $\Theta: \pi_1(\Sigma) \to \O(n)$.  For
such a $\Theta$ denote by $sw_i(\Theta) = sw_i(W_\Theta) \in
\h^i(\Sigma; \FF_2)$ the Stiefel-Whitney classes.  We first note

\begin{lem}
  \label{lem_SW_Om_to_On}
  Let $m \leq n$ and let $\phi: \O(m) \to \O(n)$ be the injection as
  the subgroup fixing pointwise a $(n-m)$-dimensional subspace. Then
  for any $\Theta: \pi_1( \Sigma) \to \O(m)$ and any $i$ one has \[
  sw_i( \Theta) = sw_i( \phi \circ \Theta). \]
\end{lem}

\subsubsection{The first Stiefel-Whitney class of an orthogonal
  representation}\label{sec:first_sw_on}

Let $ \det: \O(n) \to \{\pm 1\} $ be the determinant homomorphism.
Then the homomorphism
\[
\det\circ \Theta: \pi_1(\Sigma) \longrightarrow \{\pm 1\}
\]
corresponds to the first Stiefel-Whitney class $sw_1(\Theta) \in
\h^1(\Sigma; \FF_2)$ under the identification $\h^1(\Sigma; \FF_2)
\cong \hom(\h^1(\Sigma; \ZZ), \FF_2) \cong \hom(\pi_1(\Sigma), \{ \pm
1\})$.  In particular, the first Stiefel-Whitney class is zero if the
representation has image in $\SO(n)$.

\begin{lem}\label{lem:sw1}
  Let $\alpha \in \h^1(\Sigma; \FF_2)$. Then, for any $n\geq 1$, there
  exists a representation $\Theta_\alpha: \pi_1(\Sigma) \to \O(n)$
  such that $sw_1(\Theta_\alpha) = \alpha$.
\end{lem}

\begin{proof} 
  The result is obvious for $n=1$ since $\det$ is then an
  isomorphism. For general $n$ apply Lemma~\ref{lem_SW_Om_to_On}.
\end{proof}

\subsubsection{The second Stiefel-Whitney class of an orthogonal representation}
Let $\Theta:\pi_1(\Sigma) \to \O(n)$ be an orthogonal representation.
The second Stiefel-Whitney class of $\Theta$ can be described as the
obstruction to lifting the representation $\rho$ to the nontrivial
double cover $\mathrm{Pin}(n)$: \[ 1 \to \{\pm 1\} \to
\mathrm{Pin}(n)\to \O(n) \to 1.
\]
This obstruction class lies in $\h^2( \Sigma, \{ \pm 1 \} ) = \{ \pm 1
\} \cong \FF_2$ and can be explicitly calculated via the following
procedure.  First, fix a standard presentation
\[
\pi_1(\Sigma) = \langle a_1, b_1, \dots, a_g, b_g \, |\, \Pi_{i=1}^g
[a_i, b_i] = 1\rangle.
\]
Second, since the extension $\mathrm{Pin}(n) \to \O(n)$ is central,
the commutator map for $\mathrm{Pin}(n)$ factors through $\O(n)$:
\[ [\cdot, \cdot]^{\widetilde{}}: \O(n) \times \O(n) \longrightarrow
\mathrm{Pin}(n).
\]
The second Stiefel-Whitney class $sw_2(\Theta) \in \h^2(\Sigma;
\FF_2)$ is given by
\[
\Pi_{i=1}^g [\Theta(a_i), \Theta(b_i)]^{\widetilde{}} \in \FF_2.
\]

\begin{lem}\label{lem:all_classes_3} 
  Let $\alpha \in \h^1(\Sigma; \FF_2) $ and $\beta \in\h^2(\Sigma;
  \FF_2)$.
  \begin{asparaitem}
  \item If $\alpha \neq 0$ or $\beta=0$ then, for any $n \geq 2$,
    there exists a representation $\Theta_{\alpha,\beta}:
    \pi_1(\Sigma) \to \O(n)$ such that $sw_1(\Theta_{\alpha,\beta}) =
    \alpha$ and $sw_2(\Theta_{\alpha,\beta}) = \beta$.
  \item If $\alpha = 0$ and $\beta\neq 0$ then, for any $n \geq 3$,
    there exists a representation $\Theta_{\alpha,\beta}$ such that
    $sw_1(\Theta_{\alpha,\beta}) = \alpha$ and
    $sw_2(\Theta_{\alpha,\beta}) = \beta$.
  \end{asparaitem}
  Furthermore, in every case, $\Theta_{\alpha,\beta}$ can be chosen to
  have finite image in $\O(n)$.
\end{lem}
Similar statements may be found in \cite[Sec.~4.2]{Oliveira_2009}.

\begin{remark}
  Note that for $\Theta$ in $\hom(\pi_1( \Sigma), \O(2))$ the equality
  $sw_1( \Theta)=0$ automatically implies $sw_2( \Theta)=0$.
\end{remark}

\begin{proof}
  Since the mapping class group acts transitively on $\h^1(\Sigma;
  \FF_2)\moins \{0\}$ and trivially on $\h^2(\Sigma; \FF_2)$ and in
  view of Lemma~\ref{lem_SW_Om_to_On}, we need to construct the
  following representations:
  \begin{enumerate}
  \item \label{item_pfLAC1} $\Theta: \pi_1(\Sigma) \to \O(1)$ with
    $sw_1( \Theta) = 0$ and $sw_2( \Theta) = 0$.
  \item \label{item_pfLAC2} $\Theta: \pi_1(\Sigma) \to \O(1)$ with
    $sw_1( \Theta) \neq 0$ and $sw_2( \Theta) = 0$.
  \item \label{item_pfLAC3} $\Theta: \pi_1(\Sigma) \to \O(2)$ with
    $sw_1( \Theta) \neq 0$ and $sw_2( \Theta) \neq 0$.
  \item \label{item_pfLAC4} $\Theta: \pi_1(\Sigma) \to \O(3)$ with
    $sw_1( \Theta) = 0$ and $sw_2( \Theta) \neq 0$.
  \end{enumerate}
  By Lemma~\ref{lem:sw1}, the existence of representations required in
  (\ref{item_pfLAC1}) and (\ref{item_pfLAC2}) is obvious.

  \smallskip

  For (\ref{item_pfLAC3}) let us denote by $e^{i\theta}$ elements of
  $\SO(2)$ and by $R$ an element of $O(2) \moins \SO(2)$.  Then
  $\det(e^{i\theta}) = 1$ and $\det (Re^{i\theta}) = -1$.

  We define the representation $\Theta: \pi_1(\Sigma) \to \O(2)$ by
  \[
  \Theta(a_1) = R, \ \Theta(b_1)= e^{i\pi}, \text{ and } \Theta(a_i) =
  \Theta(b_i) = 1 \ \text{ for any } i>1.
  \]
  The relation $\Pi_{i=1}^g [\Theta(a_i), \Theta(b_i)]=1$ is obviously
  satisfied since $e^{i\pi}$ is central in $\O(2)$. Clearly $sw_1(
  \Theta) \neq 0$.

  The group $\mathrm{Pin}(2)$ is in this case isomorphic to $\O(2)$
  and the restriction to $\SO(2)$ of the cover $\O(2) \cong
  \mathrm{Pin}(2) \to \O(2)$ is $\SO(2) \to \SO(2)\sep{} x \mapsto
  x^2$, the image of $R$ being $R$.  We can hence lift $\Theta(b_1)$
  to $e^{i\pi/2}$ and $\Theta(a_1)$ to $R$. Then we have
  \[
  sw_2(\Theta)=\Pi_{i=1}^g [\Theta(a_i), \Theta(b_i)]^{\widetilde{}} =
  e^{-i\pi} = -1.
  \]

  \smallskip

  For (\ref{item_pfLAC4}) we are looking for a representation
  $\Theta:\pi_1(\Sigma) \to \SO(3)$ (i.e.\ $sw_1(\Theta) = 0$) which
  does not lift to $\mathrm{Spin}(3)$.  Let us realize
  $\mathrm{Spin}(3)\cong S^3$ as the quaternions of norm one. The
  covering map $\mathrm{Spin}(3) \to \SO(3)$ is realized by the action
  by conjugation on the space of imaginary quaternions.

  Let us denote by $\{ 1, i, j, k \}$ the standard basis of the
  quaternions.  We define a representation $\Theta: \pi_1(\Sigma) \to
  \SO(3)$ by sending $a_1$ to the projection in $ \SO(3)$ of $i$,
  $b_1$ to the projection of $j$ and all other generators of
  $\pi_1(\Sigma)$ to the trivial element. Since $[\Theta(a_1),
  \Theta(b_1)]^{\widetilde{}} = [i,j] = -1$ this defines a
  homomorphism into $\SO(3)$ which does not lift to
  $\mathrm{Spin}(3)$.
\end{proof}

\subsubsection{Proof of Proposition~\ref{prop:realize_invariants}}
\label{sec:proof-prop-ref}

We have to show that we can realize all $2^{2g}$ different choices for
$sw_1(\rho)$ in the fixed coset of $\h^1(T^1\Sigma; \FF_2)$ and the
$2$ choices for $sw_2(\rho)$ in the image of $\h^2(\Sigma; \FF_2)$ in
$\h^2(T^1\Sigma; \FF_2)$.

Let $\iota: \pi_1(\Sigma) \to \SL(2,\RR)$ be the fixed discrete
embedding, $\Theta: \pi_1(\Sigma) \to \O(n)$ an orthogonal
representation and $\rho_\Theta: \pi_1(\Sigma) \to \Sp(2n,\RR)$ the
corresponding twisted diagonal representation.

The following formulas hold for the first and second Stiefel-Whitney
classes \cite[Corollary~5.4]{Thomas_tensor}
(Proposition~\ref{prop_swtensprod} provides a short proof for the
reader's convenience):
\begin{equation}
  \label{eq:formula}
  sw_1(\rho_{\Theta}) =  sw_1(L^s(\iota)\otimes \overline{W}) = n sw_1(L^s(\iota)) +sw_1(\overline{W})
\end{equation}
\begin{multline*}
  sw_2(\rho_{\Theta}) = sw_2(L^s(\iota)\otimes \overline{W}) =
  \frac{n(n-1)}{2} sw_1(L^s(\iota))\cupprod sw_1(L^s(\iota)) \\ +
  (n-1) sw_1(L^s(\iota))\cupprod
  sw_1(\overline{W})+sw_2(\overline{W}),
\end{multline*}
where $\overline{W}$ is the pullback to $T^1\Sigma$ of the flat
$n$-plane bundle over $\Sigma$ associated with $\Theta$.  Note that
$sw_i(\overline{W})$ is the image of $sw_i(\Theta)$ under the natural
map in cohomology.
 
Lemma~\ref{lem:all_classes_3} implies that by choosing different
$\O(n)$-representations for $\Theta$ we can realize $2^{2g} \times 2$
different choices for $sw_1(\Theta)$ and $sw_2(\Theta)$.  Since $n
sw_1(\iota)$ is fixed, as we vary $sw_1(\overline{W})$ over the
$2^{2g}$ distinct classes in the image of $\h^1(\Sigma; \FF_2)$ in
$\h^1(T^1\Sigma; \FF_2)$ we realize all possible $2^{2g}$ classes in
the $\h^1(\Sigma; \FF_2)$-coset in $\h^1(T^1\Sigma; \FF_2)$ determined
by $n sw_1(\iota)$.  Similarly for any $sw_1(\overline{W})$ we can
realize the two possibilities for $sw_2(\overline{W})$, so we can
realize the two possibilities for $sw_2$ by an appropriately chosen
twisted diagonal maximal representation $\rho_\Theta: \pi_1(\Sigma)
\to \Sp(2n,\RR)$.

\subsection{Consequences for maximal representations}
\label{sec_cons_for_max}
Let us now derive some of the results stated in the introduction.

Every representation in the Hitchin component is purely loxodromic
\cite[Prop.~3.4]{Labourie_anosov}, therefore twisted diagonal
representations are never contained in Hitchin
components. Proposition~\ref{prop:realize_invariants} implies then
that for $n\geq3$ the space
\[\hommax(\pi_1(\Sigma), \Sp(2n,\RR)) \moins
\homHit(\pi_1(\Sigma), \Sp(2n,\RR))\] consists of at least $2\times
2^{2g}$ connected components.  As the total number of connected
components of $\hommax(\pi_1(\Sigma), \Sp(2n,\RR))$, if $n\geq 3$, is
$3\times 2^{2g}$ \cite[Theorem~8.7]{GarciaPrada_Gothen_Mundet} and as
$\homHit(\pi_1(\Sigma), \Sp(2n,\RR))$ has $2^{2g}$ components, we
conclude that the space $\hommax(\pi_1(\Sigma), \Sp(2n,\RR)) \moins
\homHit(\pi_1(\Sigma), \Sp(2n,\RR))$ has $2\times 2^{2g}$ components
and that the first and second Stiefel-Whitney classes of the Anosov
bundle associated with the representation distinguish them; this
proves Theorem~\ref{thm_intro:symp_maximal}.  In the $2^{2g}$ Hitchin
components any representation can be deformed into an irreducible
Fuchsian representation. The remaining $2\times 2^{2g}$ connected
components all contain a twisted diagonal representation. This gives
Theorem~\ref{thm_intro:symp_n>2} and
Corollary~\ref{cor_intro:deform_closed} of the introduction.
Corollary~\ref{cor_intro:deform} then follows directly from the
computations in Section~\ref{subsec:comp_irr} and
Section~\ref{subsec:comp_diag}.

\subsection{Hybrid representations}
\label{sec:sta_res}
The goal of this section is to prove the following
\begin{thm}
  \label{thm:ReducedEulerCC}
  Let $\gamma \in \pi_1(\Sigma)\moins\{1\}$ and $\LZPlus\in \Ll_+$ an
  oriented Lagrangian subspace of $\RR^4$. Then the relative Euler
  class distinguishes the connected components of
  \[ \homMaxFourZero\moins\homHitFour.
  \]
  More precisely, given any $l$ in $\ZZ/ (2g-2)\ZZ \cong \h^2( T^1
  \Sigma; \ZZ)^{tor}$,
  \begin{itemize}
  \item if $l \neq (g-1)$ the set $e^{-1}_{\gamma, \LZPlus}(\{l\})$ is
    a connected component of the space $\homMaxFourZero$. Every
    representation in $e^{-1}_{\gamma, \LZPlus}(\{l\})$ can be
    deformed to a $k$-hybrid representation, where $k = g-1 -l
    \mod{2g-2}$;
  \item the set $e^{-1}_{\gamma, \LZPlus}(\{g-1\})$ is the union of
    the connected component of $\homMaxFourZero$ containing diagonal
    Fuchsian representations in $\Sp(4,\RR)$ and of the Hitchin
    components $\homHitFour$.
  \end{itemize}
\end{thm}
For simplicity we restrict the discussion to the case of $k$-hybrid
representations which are constructed from the decomposition of the
surface $\Sigma = \Sigma_l \cup \Sigma_r$ along a simple closed
separating geodesic $\gamma$ (see
Section~\ref{sec:description_hybrid}).  The general computations
follow by a straightforward extension of this case.

The first step is to calculate the first Stiefel-Whitney class.
\begin{prop}
  \label{prop:swhybrid}
  Let $\rho: \pi_1( \Sigma) \to \Sp(4, \RR)$ be a $k$-hybrid
  representation as defined in Section~\ref{sec:description_hybrid},
  then \[ sw_1( \rho) = 0.\]
\end{prop}

\begin{proof}
  By Proposition~\ref{prop:restriction} we have $sw_1( \rho) \in
  \h^1(\Sigma; \FF_2) \subset \h^1( T^1 \Sigma; \FF_2)$, thus it is
  sufficient to show that $sw_1( \rho)$ is zero on a basis of the
  first homology group of $\Sigma$.
  
  Let $\pi_1(\Sigma) = \langle a_1, b_1, \ldots, a_g, b_g \, |\,
  \Pi_{i=1}^g [a_i, b_i] =1\rangle$ be a standard presentation, then $
  a_1, b_1, \ldots, a_g, b_g $ form a basis of $\h^1(\Sigma; \FF_2)$.
  We can choose such a standard presentation of $\pi_1(\Sigma)$ with
  the property that for any element $h$ of the generating set, either
  $h \in \pi_1(\Sigma_l)$ or $h\in \pi_1(\Sigma_r)$, where the hybrid
  representation $\rho$ was constructed with respect to a
  decomposition $\Sigma = \Sigma_l \cup \Sigma_r$.  Then the sign of
  $\det (\rho(h)|_{L^s(h)})$, where $L^s(h)$ is the attracting
  Lagrangian for $\rho(h)$, is positive for every element $h$ in the
  above generating set. This can be checked independently for the
  irreducible Fuchsian representation $\rho_{irr}$ and (deformations
  of) the diagonal Fuchsian representation $\rho_\Delta$.  In view of
  Lemma~\ref{lem:sw1=sign} this implies $sw_1(\rho) = 0$.
\end{proof}

\begin{prop}
  \label{prop:eulerHybrid}
  Let $\Sigma = \Sigma_l \cup\Sigma_r$ and $\rho = \rho_l * \rho_r:
  \pi_1( \Sigma) \to \Sp(4, \RR)$ be a hybrid representation as
  defined in Section~\ref{sec:the_rep_hybrid}.
  
  Let $\gamma \in \pi_1(\partial \Sigma_l)$ and let
  $(\epsilon_i)_{i=1, \dots, 4}$ one of the basis appearing in
  Definition~\ref{defi_good_triples}. Suppose that $\LZPlus = \langle
  \epsilon_1, \epsilon_2\rangle$ is an attracting fixed point for
  $\rho(\gamma)$.

  Then \[ e_{ \gamma, \LZPlus}( \rho) = ( g-1 - \chi( \Sigma_l))
  \GenTor \in \h^2(T^1\Sigma; \ZZ).\]
\end{prop}

This proposition implies easily Theorem~\ref{thm:ReducedEulerCC}.

\begin{remark} 
  The result is valid also when $\Sigma_l$ of $\Sigma_r$ is empty.
\end{remark}
\begin{remark}\label{rem:negatively_adjusted}
  If the pair $(\rho_l, \rho_r)$ is negatively adjusted with respect
  to $\gamma$ (Definition~\ref{defi_good_triples}), then the Euler
  number is \[ e_{ \gamma, \LZPlus}( \rho) = ( g-1 + \chi( \Sigma_l))
  \GenTor \in \h^2(T^1\Sigma; \ZZ).\]
\end{remark}

\subsubsection{Reduction to the group $\widehat{ \pi_1(\Sigma)}$}
\label{sec:reduction-group-hat}
Our strategy to prove Proposition~\ref{prop:eulerHybrid} is to change
the representation $\rho$ slightly.

We denote by $\widehat{ \pi_1(\Sigma)}$ the group $\{ \pm 1\} \times
\pi_1(\Sigma)$; using the morphism $\iota : \pi_1(\Sigma) \to \SL(2,
\RR)$ there is an embedding $\widehat{ \iota} : \widehat{
  \pi_1(\Sigma)} \to \SL(2, \RR)$, with $\widehat{ \iota}( -1) =
-\id_2$ and $\widehat{ \iota} |_{\pi_1(\Sigma)} = \iota$. Observe that
\[ \widehat{ \iota}\big( \widehat{ \pi_1(\Sigma)} \big) \backslash
\SL(2, \RR) \cong \iota( \pi_1(\Sigma)) \backslash \PSL(2, \RR) \cong
T^1 \Sigma,\] so that the group $\widehat{ \pi_1(\Sigma)}$ is a
quotient of the group $ \pi_1( T^1 \Sigma)$ and hence the notion of
Anosov representations and their invariants (see
Section~\ref{sec:prelim} and Section~\ref{sec:invariants}) are well
defined for $\widehat{\pi_1(\Sigma)}$.

Let $\varepsilon : \widehat{ \pi_1(\Sigma)} \to \{ \pm 1\}$ be the
projection onto the first factor; for any representation $\rho :
\pi_1(\Sigma) \to \Sp(4, \RR)$ we define a representation $\varepsilon
\otimes \rho$ of $\widehat{ \pi_1(\Sigma)}$ into $\Sp(4, \RR)$ by
setting $\varepsilon \otimes \rho(x,\gamma)= \varepsilon(x)
\rho(\gamma)$.  The relations between the invariants of $\rho$ and of
$\varepsilon \otimes \rho$ are discussed in
Appendix~\ref{sec:trick}. In view of these relations
Proposition~\ref{prop:eulerHybrid} follows from:

\begin{prop}
  \label{prop:EulerHybrTwisted}
  Let $\rho$, $\Sigma_l$, $\Sigma_r$, $\gamma$ and $\LZPlus$ be as in
  Proposition~\ref{prop:eulerHybrid}.

  Then
  \[ e_{ \gamma, \LZPlus}( \varepsilon \otimes \rho) = - \chi(
  \Sigma_l) \GenTor. \]
\end{prop}

The rest of this section is devoted to the proof of this proposition.

\subsubsection{Constructing sections}
\label{sec:constr-sect}

In order to calculate the Euler class for $\varepsilon \otimes \rho$
we will construct a lift of $e_{ \gamma, \LZPlus}( \varepsilon \otimes
\rho)$ under the connecting homomorphism $\h^1 (T^1 \Sigma |_\gamma;
\ZZ) \to \h^2( T^1 \Sigma; \ZZ) $ appearing in the Mayer-Vietoris
sequence (Appendix~\ref{subsec:cohomologysigma}) and then calculate
this lift in $\h^1 (T^1 \Sigma |_\gamma; \ZZ) \cong \ZZ^2$.  Such a
lift can be constructed from trivializations of the Lagrangian bundle.

\begin{prop}
  \label{prop:constrsect}
  Let $\rho$, $\Sigma_l$, $\Sigma_r$, $\gamma$ and $\LZPlus$ be as in
  Proposition~\ref{prop:eulerHybrid}. Let $L_+^s(\varepsilon \otimes
  \rho)$ be the oriented Lagrangian reduction for the flat
  $\varepsilon \otimes \rho$-flat bundle over $T^1 \Sigma$.

  Then the restrictions of $L_+^s(\varepsilon \otimes \rho)$ to both
  $T^1 \Sigma |_{\Sigma_l}$ and $T^1 \Sigma |_{\Sigma_r}$ are trivial.
\end{prop}

This proposition is a consequence of the following three lemmas.

\begin{lem}
  \label{lem:homogenousTrivial}
  Let $\iota : \pi_1( \Sigma) \to \SL(2, \RR)$ be a maximal
  representation and let $\phi : \SL(2, \RR) \to \Sp( 4, \RR)$ be a
  homomorphism such that $\rho = \phi \circ \iota$ is maximal.

  Then the oriented Lagrangian bundle $L_+^s(\varepsilon \otimes
  \rho)$ is trivial.
\end{lem}

\begin{proof}
  First we observe that $\varepsilon \otimes \rho = \phi \circ
  \widehat{ \iota}$ with $\widehat{ \iota} = \varepsilon \otimes
  \iota$. In this situation the map defining the oriented Lagrangian
  bundle is the equivariant map
  \[
  \begin{array}{rcl}
    \SL(2, \RR) & \longrightarrow & \mathcal{L}_+ \\
    g & \longmapsto & \phi(g) \cdot \LZPlus,
  \end{array}
  \]
  where $\LZPlus$ is the attracting Lagrangian fixed by
  $\rho(\gamma)$.

  An equivariant map $\SL(2, \RR) \to \Sp(4,\RR)$ that trivializes the
  corresponding bundle is given simply by
  \[ g \longmapsto \phi(g)^{-1}.\qedhere \]
\end{proof}

\begin{remark}
  \label{rem_trivialization}
  \begin{asparaenum}
  \item A \emph{trivialization} of a symplectic bundle ${E}= \pi_1(M)
    \backslash (\widetilde{M} \times \RR^4)$ is an isomorphism with
    the trivial bundle $M\times \RR^4$. At the level of the universal
    cover this is the same as an isomorphism from $\widetilde{M}
    \times \RR^4$ to $\widetilde{M} \times \RR^4$ intertwining the
    action of $\pi_1(M)$ by the representation $\rho: \pi_1(M) \to
    \Sp(4, \RR)$ and the trivial action. This is given by a map
    $\varphi: \widetilde{M} \to \Sp(4, \RR)$ satisfying the
    equivariance equation: $\varphi(\gamma \cdot \tilde{m}) =
    \varphi(\tilde{m}) \rho( \gamma)^{-1}$ for all $\gamma$ in
    $\pi_1(M)$ and $\tilde{m}$ in $\widetilde{M}$.

  \item Let ${L}$ be a subbundle $ {E}$ or, what amounts to the same,
    an invariant subbundle of $\widetilde{M} \times \RR^4$, i.e.\ a
    $\rho$-equivariant map $\zeta: \widetilde{M} \to
    \mathrm{Gr}_2(\RR^4)$. The trivialization $\varphi$ induces
    furthermore a trivialization of ${L}$ if the map $\tilde{m}
    \mapsto \varphi( \tilde{m}) \cdot \zeta( \tilde{m})$ is constant.
  \end{asparaenum}
\end{remark}

\begin{lem}
  \label{lem:homo}
  Let $\rho_0$ and $\rho_1$ be two homotopic maximal
  representations. Then the Lagrangian bundles $L^s( \rho_0)$ and
  $L^s( \rho_1)$ are homotopic and hence they are isomorphic.

  If the first Stiefel-Whitney class $sw_1( \rho_0) = sw_1( \rho_1)$
  is zero, then the corresponding oriented Lagrangian bundles are also
  isomorphic.
\end{lem}

\begin{proof}
  This is a consequence of the fact that the equivariant positive
  curve depends continuously on the representation
  (Fact~\ref{fact:curve}) and the fact that homotopic bundles are
  isomorphic \cite[p.~53]{Steenrod}.
\end{proof}

\begin{lem}
  \label{lem:isomorphicRestriction}
  Let $\rho$ and ${ \rho'}$ be two maximal representations with zero
  first Stiefel-Whitney class and such that $ \rho |_{\pi_1{ \Sigma'}}
  = { \rho'} |_{\pi_1{ \Sigma'}}$ for a subsurface $\Sigma' \subset
  \Sigma$.

  Then the two bundles $L_+^s( \varepsilon \otimes \rho)$ and $L_+^s(
  \varepsilon \otimes {\rho'})$ are homotopic when \emph{restricted}
  to $T^1 \Sigma|_{ \Sigma'}$.
\end{lem}

\begin{proof}
  Note that if $L^s( \varepsilon \otimes \rho)$ is homotopic to $L^s(
  \varepsilon \otimes {\rho'})$ then $L_+^s( \varepsilon \otimes
  \rho)$ is homotopic to $L_+^s( \varepsilon \otimes
  {\rho'})$. Therefore it is sufficient to prove the result without
  considering orientations.  Let $D_\varepsilon$ be the flat line
  bundle associated with $\varepsilon$, so that $L^s( \varepsilon
  \otimes \rho) = D_\varepsilon \otimes L^s( \rho)$ and $L^s(
  \varepsilon \otimes {\rho'}) = D_\varepsilon \otimes L^s(
  {\rho'})$. Thus, we need to show that $L^s( \rho) |_{M'} $ is
  homotopic to $L^s( {\rho'}) |_{M'}$ with $M' = T^1 \Sigma|_{
    \Sigma'} \subset M = T^1 \Sigma$.

  Let $\widetilde{ \Sigma}'$ be the universal cover of $\Sigma'$; it
  can be realized as a $\pi_1( \Sigma')$-invariant subset of
  $\widetilde{ \Sigma}$. More precisely, under an identification
  $\widetilde{ \Sigma} \cong \HH_2$ we can set $\widetilde{ \Sigma}' =
  \mathrm{Conv}( \Lambda_{\pi_1( \Sigma')})$ the convex hull of the
  limit set $\Lambda_{\pi_1 (\Sigma')}$ of $\pi_1( \Sigma')$ in the
  boundary $\partial \HH_2$ of the hyperbolic plane.

  The manifold $\overline{ M} = T^1 \widetilde{ \Sigma}$ is a $\pi_1(
  \Sigma)$-cover of $M$ and we set $\overline{M}' = T^1 \widetilde{
    \Sigma}|_{ \widetilde{ \Sigma}'} \subset T^1 \widetilde{ \Sigma} =
  \overline{ M}$ so that $M' \cong \pi_1( \Sigma') \backslash
  \overline{M}'$.  When we identify the unit tangent bundle
  $\overline{ M}$ with $\partial \pi_1(\Sigma)^{(3+)}$, the set of
  positively oriented triples of $\partial \pi_1(\Sigma)$, we can
  identify $\overline{ M}'$ with the subset of $\partial
  \pi_1(\Sigma)^{(3+)}$ whose projection to $\widetilde{ \Sigma}$
  belongs to $\widetilde{ \Sigma}'$.

  The bundle $L^s( \rho)$ is constructed via the $\rho$-equivariant
  map
  \[\begin{array}{rl}
    \partial \pi_1(\Sigma)^{(3+)} &
    \overset{p}{\longrightarrow} \partial \pi_1(\Sigma)
    \overset{\xi}{\longrightarrow} \mathcal{L} \\
    (t^s, t, t^u) & \longmapsto t^s \longmapsto \xi( t^s)
  \end{array}\]
  where $\xi$ is the positive $\rho$-equivariant curve. Similarly
  $L^s( {\rho'})$ is constructed from the positive ${
    \rho'}$-equivariant curve ${ \xi'}$.

  This means that the restriction $L^s( \rho)|_{M'}$ is constructed
  from the $\rho|_{\pi_1 (\Sigma')}$-equivariant map
  \[ \overline{ M}' \xrightarrow{ p|_{ \overline{M}'}} \partial
  \pi_1(\Sigma) \overset{ \xi}{ \longrightarrow} \mathcal{L}. \]
  Conversely, any $\rho|_{\pi_1 (\Sigma')}$-equivariant continuous map
  $\partial \pi_1(\Sigma) \to \mathcal{L}$ defines a Lagrangian
  reduction of the symplectic $\RR^4$-bundle over $M'$. Hence we get a
  homotopy of the two bundles $L^s( \rho)|_{ M'}$ and $L^s(
  {\rho'})|_{ M'}$ once we have a $\rho|_{\pi_1
    (\Sigma')}$-equivariant homotopy between the two maps
  \begin{align*}
    \overline{ M}' & \xrightarrow{ p|_{ \overline{M}'}} \partial
    \pi_1(\Sigma)
    \overset{ \xi}{ \longrightarrow} \mathcal{L} \\
    \overline{ M}' & \xrightarrow{ p|_{ \overline{M}'}} \partial
    \pi_1(\Sigma) \overset{ {\xi'}}{ \longrightarrow} \mathcal{L}.
  \end{align*}
  For this it is sufficient to construct a $\rho|_{\pi_1
    (\Sigma')}$-equivariant homotopy between the two positive maps
  $\xi$ and ${\xi'}$. This is the content of the next lemma.
\end{proof}

\begin{lem} 
  Let $\rho: \pi_1( \Sigma ) \to \Sp( 4, \RR)$ be a maximal
  representation and let $\Sigma' $ be a subsurface. Then the set
  \[ \mathcal{C} = \{ \xi : \partial \pi_1(\Sigma)\to \Ll \mid \xi
  \text{ is positive, continuous and } \rho|_{\pi_1(
    \Sigma')}\text{-equivariant} \}\] is connected.
\end{lem}
\begin{proof}
  To simplify notation we assume that the boundary $\partial \Sigma' =
  \gamma$ consists of one component.  Note first that the restriction
  to the limit set $\Lambda_{ \pi_1( \Sigma')}$ of any curve $\xi$ in
  $\mathcal{C}$ is completely determined by the action of
  $\pi_1(\Sigma')$. The complement $\partial \pi_1(\Sigma) \moins
  \Lambda_{ \pi_1( \Sigma')}$ is a countable union of open intervals,
  which are transitively exchanged by the action of
  $\pi_1(\Sigma')$. The interval $( t_{\gamma}^{u}, t_{\gamma}^{s})$,
  where $t^{u}_{\gamma}$ and $t^{s}_{\gamma}$ are the fixed points of
  $\gamma$ in $\partial \pi_1(\Sigma)$, is one such interval. The map
  $\xi$ is completely determined by its restriction to this interval.

  Conversely given a positive $\rho|_{ \langle \gamma
    \rangle}$-equivariant continuous curve $\beta : (t^{u}_{\gamma},
  t^{s}_{\gamma}) \to \mathcal{L}_+$, one obtains a continuous
  $\rho|_{ \pi_1( \Sigma')}$-equivariant positive curve $\partial
  \pi_1(\Sigma) \moins \Lambda_{ \pi_1( \Sigma')} \to \mathcal{L}$,
  which one shows to have a continuous extension $\partial
  \pi_1(\Sigma) \to \mathcal{L}$.

  Thus the map:
  \begin{align*}
    \mathcal{C} & \longrightarrow \{ \beta : (t^{u}_{\gamma},
    t^{s}_{\gamma}) \to \mathcal{L} \mid \beta \text{ positive
      continuous and } \rho|_{
      \langle \gamma \rangle}\text{-equivariant} \}\\
    \xi & \longmapsto \xi|_{(t^{u}_{\gamma}, t^{s}_{\gamma})}
  \end{align*}
  is a homeomorphism. By Proposition~\ref{prop:conn_equiv_posit2},
  this space is connected.
\end{proof}

\subsubsection{Calculating the Euler class}

Proposition~\ref{prop:constrsect} gives precisely what is needed to
construct a lift of the Euler class $e_{ \gamma, \LZPlus} (
\varepsilon \otimes \rho)$ under the connecting homomorphism $\h^1(
T^1 \Sigma |_{ \gamma} ; \ZZ)\to \h^2( T^1 \Sigma; \ZZ)$.  We have
(see Appendix~\ref{subsec:cohomologysigma})
\begin{equation}\label{eq:h_1}
  \h^1( T^1 \Sigma |_{ \gamma} ; \ZZ) \cong \hom ( \pi_1( T^1
  \Sigma |_\gamma), \ZZ) \cong \ZZ \times \ZZ.
\end{equation}
The last identification sends $\phi \in \hom( \pi_1( T^1 \Sigma
|_\gamma), \ZZ)$ to $( \phi( T^{1}_{x} \Sigma ), \phi( \gamma) )$,
where the two circles $T^{1}_{x} \Sigma$ and $\gamma$ are naturally
considered as loops in $T^1 \Sigma|_\gamma$.

\begin{prop}
  \label{prop:calc_lift_in_Hone}
  Let $\rho$, $\Sigma_l$, $\Sigma_r$, $\gamma$ and $\LZPlus$ be as in
  Proposition~\ref{prop:eulerHybrid}. Let $L_+^s(\varepsilon \otimes
  \rho)$ be the oriented Lagrangian reduction for the flat
  $\varepsilon \otimes \rho$-bundle over $T^1 \Sigma$.  Suppose that
  $g_l$ and $g_r$ are trivializations of the restrictions of
  $L_+^s(\varepsilon \otimes \rho)$ to $T^1 \Sigma |_{\Sigma_l}$ and
  $T^1 \Sigma |_{\Sigma_r}$.
  
  Then $g_l \circ g_r^{-1} : T^1 \Sigma |_\gamma \times \RR^2 \to T^1
  \Sigma |_\gamma \times \RR^2$ is a gauge transformation of the
  trivial oriented $\RR^2$-bundle over $T^1 \Sigma|_\gamma$ and
  defines a map $h : T^1 \Sigma |_\gamma \to \GL^+( 2, \RR)$.  Let $
  h_* \in \hom( \pi_1 ( T^1\Sigma |_\gamma) , \pi_1 ( \GL^+(2, \RR) )$
  denote the map induced by $h$ at the level of fundamental groups.
  \begin{enumerate}
  \item The image of $h_* \in \hom( \pi_1 ( T^1\Sigma |_\gamma) ,
    \pi_1 ( \GL^+(2, \RR) )) \cong \h^1(T^1\Sigma |_\gamma; \ZZ)$
    under the connection homomorphism $\h^1( T^1 \Sigma |_{ \gamma} ;
    \ZZ)\to \h^2( T^1 \Sigma; \ZZ)$ is the Euler class: $\delta( h_*)=
    e_{ \gamma, \LZPlus }( \varepsilon \otimes \rho)$.
  \item Under the identification $\h^1(T^1\Sigma; \ZZ) \cong \ZZ
    \times \ZZ$ in \eqref{eq:h_1}, $ h_*$ is equal to $(1,0)$, up to
    the image of the map $\h^1( T^1\Sigma |_{\Sigma_l} ; \ZZ) \oplus
    \h^1( T^1\Sigma |_{\Sigma_r} ; \ZZ) \to \h^1( T^1\Sigma |_\gamma ;
    \ZZ)$.
  \end{enumerate}
\end{prop}

\begin{remark}
  The ambiguity given by the image of the map $\h^1( T^1\Sigma
  |_{\Sigma_l} ; \ZZ) \oplus \h^1( T^1\Sigma |_{\Sigma_r} ; \ZZ) \to
  \h^1( T^1\Sigma |_\gamma ; \ZZ)$ accounts for changing the
  trivializations $g_l$ and $g_r$ of the restrictions of
  $L_+^s(\varepsilon \otimes \rho)$ to $T^1 \Sigma |_{\Sigma_l}$ and
  $T^1 \Sigma |_{\Sigma_r}$.
\end{remark}

\begin{proof}[Proof of Proposition~\ref{prop:EulerHybrTwisted}]
  With the notation of Proposition~\ref{prop:calc_lift_in_Hone}
  $\delta (h_*) = e_{ \gamma, \LZPlus }( \varepsilon \otimes
  \rho)$. Since $ h_*$ is equal to $(1,0)$ in the identification
  \eqref{eq:h_1}, Proposition~\ref{prop:mayervietoris} precisely says
  that $\delta (h_*) = ( 2 g( \Sigma_l ) - 1) \GenTor = -\chi(
  \Sigma_l) \GenTor$.
\end{proof}

\begin{proof}[Proof of Proposition~\ref{prop:calc_lift_in_Hone}]
  By Proposition~\ref{prop:constrsect} the restrictions of
  $L_+^s(\varepsilon \otimes \rho)$ to $T^1 \Sigma |_{\Sigma_l}$ and
  $T^1 \Sigma |_{\Sigma_r}$ are indeed trivial.

  The first statement $\delta (h_*) = e_{ \gamma, \LZPlus }(
  \varepsilon \otimes \rho)$ follows from the fact that the Euler
  class is the obstruction to trivialize the oriented Lagrangian
  bundle $L^s_+(\varepsilon \otimes\rho)$ over $T^1\Sigma$ (see
  Proposition~\ref{prop_eulclassMV}).
  
  For the second statement, let $T^{1}_{x} \Sigma$ and $\gamma$ be the
  two generators of the fundamental group of the $2$-torus $T^1 \Sigma
  |_\gamma$. We have to show, identifying $\pi_1( \GL^+(2, \RR) )$
  with $\ZZ$, that
  \[ h_*( T^{1}_{x} \Sigma) = 1 \text{ and } h_*( \gamma) =0.\]
   
  In both cases our strategy will be the same: we write the homotopy
  class we want to calculate as a product of two homotopy classes, the
  first depending only on the restriction to $\Sigma_l$ and the second
  depending on the restriction to $\Sigma_r$. With this we can deform
  the representations and the fiber bundles independently on
  $\Sigma_l$ and on $\Sigma_r$, without changing the homotopy classes
  we are considering. By construction of the hybrid representations
  this means that we can reduce the calculation of the homotopy class
  to the case when the representations are the restrictions of the
  irreducible Fuchsian representation for $\Sigma_l$ and the diagonal
  Fuchsian representation for $\Sigma_r$.

  \smallskip

  We start by proving the second equality: $ h_*( \gamma) =0$.
 
  We identify $\gamma$ with $\ZZ \backslash \RR$ so that the
  restriction of the symplectic bundle to $\gamma$ is identified with
  $\ZZ \backslash (\RR \times \RR^4)$ where $\ZZ$ acts diagonally on
  $\RR \times \RR^4$, $n \cdot ( t, v) = (n+t, (A_\gamma)^n v)$, with
  $A_\gamma = \rho( \gamma)$ being the diagonal element $\diag(
  e^{l_\g}, e^{k_\g}, e^{-l_\g}, e^{-k_\g} )$.  Furthermore, the
  restriction of the oriented Lagrangian reduction $L^{s}_{+}(
  \varepsilon \otimes \rho)$ to the geodesic $\gamma$ is flat and
  hence of the form $\ZZ \backslash ( \RR \times \LZPlus)$.

  A trivialization of the symplectic bundle over $\gamma$ is given by
  an equivariant map $H : \RR \to \Sp(4, \RR)$, i.e.\ $H( t+1) = H(t)
  A_{\gamma}^{-1}$ (see Remark~\ref{rem_trivialization}).  Such a
  trivialization induces a trivialization of the Lagrangian bundle
  $L^s_{+}( \varepsilon \otimes \rho)|_\gamma$ if furthermore for all
  $t$ the element $H(t)$ stabilizes $\LZPlus$.  We now provide a
  ``canonical'' trivialization in our situation.  For this let $A =
  \diag( e^l, e^k, e^{-l}, e^{-k})$ be a diagonal element and let us
  denote by $L_A$ the Lagrangian bundle $\ZZ \backslash ( \RR \times
  \LZPlus)$ where $n \cdot ( t, v) = (n+t, A^n v)$. Then the
  continuous map
  \[ H(A, \cdot) : \RR \longrightarrow \Sp(4, \RR)\sep{} t \longmapsto
  H(A,t) = \diag( e^{-tl}, e^{-tk}, e^{tl}, e^{tk})\] provides a
  trivialization of $L_A$.

  \medskip

  The homotopy class $h_*( \gamma)$ we are calculating comes from two
  trivializations $g_l|_\gamma, g_r|_\gamma : L_+^s( \varepsilon
  \otimes \rho)|_\gamma \to \gamma \times \LZPlus$, and
  \[ g_l|_\gamma \circ (g_{r} |_\gamma)^{-1} (t, v) = (t ,
  h(\gamma(t)) v),\] for $(t,v) \in \gamma \times \LZPlus$.

  With the trivialization $H(A_\gamma, \cdot)$ of $L_+^s( \varepsilon
  \otimes \rho)|_\gamma$ given above, we have
  \[ g_\star |_\gamma \circ H( A_\gamma, \cdot)^{-1} (t,v) = (t,
  M_\star(t) v), \text{ for } \star=l,r.\] This means that in $\pi_1(
  \GL^+(2, \RR)) = \ZZ$ we have the equality
  \[ h_*(\gamma) = [M_l] - [M_r].\]
  
  We now prove that both $[M_l]$ and $[M_r]$ are trivial.

  By construction of the hybrid representation $\rho = \rho_l* \rho_r$
  in Section~\ref{sec:the_rep_hybrid} we know that $A_\gamma =
  \rho_l(\gamma) =\rho_{irr}(\gamma) = \phi_{irr}\circ \iota(\gamma)$,
  where $\iota: \pi_1(\Sigma) \to \SL(2,\RR)$ is a discrete embedding
  and $\phi_{irr}: \SL(2,\RR) \to \Sp(4,\RR)$ is the irreducible
  representation.  The trivialization given in the proof of
  Lemma~\ref{lem:homogenousTrivial} and the formula for $H( \cdot,
  \cdot)$ imply that $M_l$ is (homotopic to) the constant map, thus
  $[M_l]=0$.
  
  To compute $[M_r] = 0$ we have to consider $A_\gamma =
  \rho_r(\gamma) =\rho_{r,1}(\gamma)$, where $(\rho_{r,s})_{s \in
    [0,1]}$ is a continuous path of maximal representations with
  $\rho_{r,0} = \rho_\Delta = \phi_\Delta \circ \iota$ and, for all
  $s$, $\rho_{r,s}(\gamma)$ is diagonal and where $\iota$ is as above
  and $\phi_\Delta: \SL(2, \RR) \to \Sp(4,\RR)$ is the diagonal
  embedding.  Thus the family of changes of trivializations $g_{r,s}
  \circ H( { \rho}_{r,s} ( \gamma),\cdot)^{-1}$, $s\in [0,1]$, provide
  a homotopy from the loop $M_r = M_{r,1}$ to the loop $M_{r,0}$,
  which is the constant map. Therefore $[M_r] = 0$.

  \medskip

  We now turn to the proof of the first equality: $h_*(T_{x}^{1}
  \Sigma) = 1$.

  In contrast to the previous calculation, here no equivariance
  properties are to be satisfied, but the trivializations we choose
  will not be that natural.

  We identify $T_{x}^{1} \Sigma$ with the group $\PSO(2)$, via $T^1
  \widetilde{ \Sigma} \cong \PSL(2, \RR)$, and identify it also with
  the boundary $\partial \pi_1(\Sigma)$, via the projection $T^1
  \widetilde{ \Sigma} \cong \partial \pi_1(\Sigma)^{(3+)} \to \partial
  \pi_1(\Sigma)\sep{} (t^s, t, t^u) \mapsto t^s$. We can suppose that
  under these identifications the attractive fixed point
  $t^s_{\gamma}$ of $\gamma$ is sent to $[\id_2]$ in $\PSO(2)$ whereas
  the repulsive fixed point $t^u_{\gamma}$ is sent to $[J]=\left[
    \left(
      \begin{array}{cc}
        0 & -1 \\ 1 & 0
      \end{array}
    \right)\right]$.

  Since we are working with the representation $\varepsilon \otimes
  \rho$, the flat $\Sp(4, \RR)$-bundle over $T^{1}_{x} \Sigma \cong
  \PSO(2)$ is the quotient
  \[ \SO(2) \times_{\{ \pm 1\}} \Sp(4, \RR) = \{ \pm 1\} \backslash
  (\SO(2) \times \Sp(4, \RR))\] of the trivial bundle over $\SO(2)$ by
  the group $\{ \pm 1\}$, where the action is given by
  \[ (-1) \cdot (s,g) = (-s, -g).\] The oriented Lagrangian reduction
  $L_+^s( \varepsilon \otimes \rho) |_{\PSO(2)}$ is given by the
  positive continuous curve associated with $\rho$:
  \[ \SO(2) \to \PSO(2) \cong \partial \pi_1(\Sigma)
  \overset{\xi_+}{\longrightarrow} \mathcal{L}_+\] into the space of
  oriented Lagrangians.

  A trivialization of the bundle $\SO(2) \times_{\{ \pm 1\}} \Sp(4,
  \RR)$ is then a $\{ \pm 1\}$-equivariant map
  \[ g : \SO(2) \longrightarrow \Sp(4, \RR).\] This trivialization
  induces furthermore a trivialization of the Lagrangian reduction
  $L_+^s( \varepsilon \otimes \rho) |_{\PSO(2)}$ if for all $\alpha$
  in $\SO(2)$
  \[ g(\alpha) \cdot \xi_+(\alpha) = \LZPlus^s.\] We observe that
  \[ \xi_+(\id_2) = \LZPlus^s \text{ and } \xi_+\left(
    \begin{array}{cc}
      0 & -1 \\ 1 & 0
    \end{array}
  \right) = \xi_+(J) = \LZPlus^{u},\] where $\LZPlus^s = \langle e_1,
  e_2 \rangle$ and $\LZPlus^u = \langle e_3, e_4 \rangle$ with $e_1,
  \dots, e_4$ the standard symplectic basis of $\RR^4$.

  \begin{lem}
    Let $\eta : \SO(2) \to \mathcal{L}_+$ be a continuous, $\{ \pm
    1\}$-invariant, positive curve. Hence $\eta$ defines a Lagrangian
    reduction $L_\eta$ of the bundle $\SO(2) \times_{\{ \pm 1\}}
    \Sp(4, \RR)$.

    Suppose that
    \[ \eta(\id_2) = \LZPlus^s \quad \text{and} \quad \eta(J) =
    \LZPlus^u.\] Then for all $\alpha$, $\eta(\alpha)$ and
    $\phi_\Delta(\alpha) \cdot \LZPlus^u$ are transverse
    Lagrangians. This means that there exists a unique symmetric $2$
    by $2$ matrix $M(\alpha)$ such that:
    \[ \left(
      \begin{array}{cc}
        \id_2 & 0 \\ M(\alpha) & \id_2
      \end{array} \right) \phi_\Delta(\alpha)^{-1} \cdot \eta(\alpha) = \LZPlus^s.\]
    The map
    \[ \beta_\eta : \SO(2) \longto \Sp(4, \RR)\sep{} \alpha
    \longmapsto \left(
      \begin{array}{cc}
        \id_2 & 0 \\ M(\alpha) & \id_2
      \end{array} \right) \phi_\Delta(\alpha)^{-1}
    \]
    is a trivialization of $L_\eta$.
  \end{lem}
  \begin{proof} The only point to prove is that $\eta(\alpha)$ and
    $\phi_\Delta(\alpha) \cdot \LZPlus^u$ are transverse. This is
    immediate for $\alpha = \pm \id_2$ and for $\alpha = \pm J$; for
    other $\alpha$, for example when $(\id_2, \alpha, J)$ is
    positively oriented, the positivity of the triples $(\LZPlus^s,
    \eta(\alpha), \LZPlus^u)$ and $(\LZPlus^u,
    \phi_\Delta(\alpha)\cdot \LZPlus^u, \LZPlus^s)$ implies that
    $\eta(\alpha)$ is the graph of $f: \LZPlus^u \to \LZPlus^s$ with
    $\omega(\cdot, f \cdot)$ positive definite and that
    $\phi_\Delta(\alpha)\cdot \LZPlus^u$ is the graph of $g: \LZPlus^u
    \to \LZPlus^s$ with $\omega(\cdot, g \cdot)$ negative
    definite. Since $\omega(v, f v) = \omega(v, g v)$ implies $v=0$,
    the transversality of $\eta(\alpha)$ and $\phi_\Delta(\alpha)\cdot
    \LZPlus^u$ follows.
  \end{proof}
 
  \medskip

  Going back to the proof of Proposition~\ref{prop:calc_lift_in_Hone},
  the trivializations $\beta_\eta$ enable us to write $h_*(T^{1}_{x}
  \Sigma)$ as the difference:
  \[ h_*(T^{1}_{x} \Sigma) = [N_l] - [N_r],\] where, for $\star =l,r$,
  $N_\star$ is defined as the change of trivializations $g_\star
  |_{T^{1}_{x} \Sigma} \circ \beta_{\xi_+}^{-1}$.

  Again $N_l$ and $N_r$ are in fact homotopic to the corresponding
  changes of trivializations we obtain from the representations
  $\phi_{irr} \circ \iota$ and $\phi_{\Delta} \circ \iota$
  respectively. It is then immediate that $N_r$ is homotopic to the
  constant map and that $N_l$ is homotopic to the map
  \begin{align*}
    \PSO(2) & \longrightarrow \mathrm{Stab}( L^s_{0+} ) \\
    [\alpha] & \longmapsto \left(
      \begin{array}{cc}
        \id_2 & 0 \\ -M(\alpha) & \id_2
      \end{array} \right) \phi_\Delta(\alpha)^{-1} \phi_{irr}(\alpha), 
  \end{align*}
  where $M(\alpha)$ is the only $2\times 2$ symmetric matrix such that
  this product belongs to $\mathrm{Stab}( \LZPlus^s )$. A direct
  calculation with the formulas given in
  Facts~\ref{facts:irreducible}.(\ref{item:splitting}) and
  \ref{facts:diagonal}.(\ref{item:embedding_psi}) gives for $\alpha =
  \left (
    \begin{array}{cc}
      \cos( \theta) & -\sin( \theta) \\  
      \sin( \theta) & \cos( \theta)
    \end{array}
  \right) $
  \[ \left(
    \begin{array}{cc}
      \id_2 & 0 \\ -M(\alpha) & \id_2
    \end{array} \right)  \phi_\Delta(\alpha)^{-1} \phi_{irr}(\alpha) = \left(
    \begin{array}{cc}
      A( \theta) & * \\ 0 & \transpose{}\! A( \theta)^{-1}
    \end{array} \right),
  \]
  where \[A(\theta) = \frac{4}{3+\cos^2( 2\theta)}
  \left(\begin{array}{cc} \cos( 2\theta) & -\frac{ \sqrt{3}}{2} \sin(2
      \theta) \\ \frac{ \sqrt{3}}{2} \sin(2 \theta) & \cos(2 \theta)
    \end{array}\right).\]
  Hence a path in $\GL^+(2, \RR)$ representing $N_l$ is $[0, \pi] \to
  \GL^+(2, \RR)\sep{} \theta \mapsto A( \theta)$.  It follows that
  $h_*( T_{x}^{1} \Sigma) = 1$.
\end{proof}

\begin{remark}
  The calculation for negatively adjusted pairs would amount to
  conjugating the map $h$ by $\diag(1,1,-1,-1)$ hence to changing
  $h_*$ in $-h_*$. This leads to the announced value for the Euler
  class (Remark~\ref{rem:negatively_adjusted}).
\end{remark}

\begin{remark}
  The above computation of the Euler class for hybrid representations
  hints towards a more general gluing formula for topological
  invariants of representations for surfaces with boundary.
\end{remark}

\subsection{Zariski density}
\label{sec:zariski-dens-prop}
Here we prove Theorem~\ref{thmintro_Zdensity} of the introduction.
First we state a lemma describing the possible Zariski closures of
maximal representations.

\begin{lem}
  \label{lem:possble_Zcl}
  Let $\rho: \pi_1(\Sigma) \to \Sp(4, \RR)$ be a maximal
  representation. Then the identity component of the Zariski closure
  of $\rho( \pi_1( \Sigma))$ is (up to conjugation)
  \begin{itemize}
  \item $\Sp(4,\RR)$,
  \item $\SL(2, \RR) \times \SL(2,\RR)$,
  \item the diagonal $\SL(2, \RR)$,
  \item $\SL(2, \RR) \times_{\{\pm 1\}} \SO(2)$, the product of the
    diagonal $\SL(2,\RR)$ and the identity component of its
    centralizer,
  \item or the irreducible $\SL(2, \RR)$.
  \end{itemize}
\end{lem}
\begin{remark}
  A similar statement can be found in
  \cite[Prop.~4.8]{Bradlow_GarciaPrada_Gothen_inpreparation}.
\end{remark}
\begin{proof}
  This proof uses basic terminology and concepts in the theory of Lie
  groups as available in \cite{Knapp_LieGrp}.

  Let $L$ be the identity component of the Zariski closure of
  $\rho(\pi_1( \Sigma))$ and suppose that $L \neq G=\Sp(4, \RR)$. Up
  to taking a finite cover one can suppose $\rho(\pi_1( \Sigma))
  \subset L$. The Toledo number $\tau( \rho)$ is the image under the
  map $\pi_1(L) \to \pi_1(G)= \ZZ$ of the obstruction class $o(\rho)$
  of $\rho: \pi_1( \Sigma) \to L$. The group $L$ is the semidirect
  product $R \ltimes U$ of its unipotent radical $U$ and a reductive
  group $R < G$. Since $U$ is contractible, $\pi_1(L) = \pi_1(R)$ and
  the obstruction class $o( \rho_R)$ of the representation $\rho_R:
  \pi_1(\Sigma) \to R = L/U$ equals $o( \rho)$.

  The reductive group $R$ is the almost product of its center $Z$ and
  a semisimple group $S$, i.e.\ $R = Z \cdot S$. Let $S_c$ the product
  of the simple compact factors of $S$ and $S_{nc}$ the product of the
  noncompact ones so that $S$ is the almost product of $S_{nc}$ and
  $S_c$, i.e.\ $S = S_{nc} \cdot S_c$.

  Again, up to taking a finite cover we can suppose that $\rho_R$
  lifts to a representation $(\rho_Z, \rho_{nc}, \rho_c) : \pi_1(
  \Sigma) \to Z \times S_{nc} \times S_c$ and that $Z$ is
  connected. Hence the Toledo number $\tau( \rho)$ is the image of the
  obstruction $(o(\rho_Z), o(\rho_{nc}), o(\rho_c))$ under the map
  $\pi_1(Z) \times \pi_1( S_{nc}) \times \pi(S_c) \to \pi_1(G)$
  induced by the multiplication map $Z \times S_{nc} \times S_c \to
  G$. Since $Z$ is connected and abelian, $\hom(\pi_1( \Sigma),Z)
  \cong Z^{2g}$ is connected and $o( \rho_Z) =0$. Since $\pi_1( S_c)$
  is finite, its image in $\pi_1(G) \cong \ZZ$ is $\{ 0\}$. These last
  remarks imply that $\tau( \rho)$ is the image of $o( \rho_{nc})$
  under $\pi_1( S_{nc}) \to \pi_1(G)$.

  Since $\tau(\rho_{nc}) \neq 0$ it follows that the abelian group
  $\pi_1( S_{nc})$ has a $\ZZ$ factor. The real rank of $S_{nc}$ must
  be $1$ or $2$ since it is obviously bounded by the rank of $G$.

  \smallskip

  Suppose first that the rank of $S_{nc}$ is $1$. The restriction on
  the fundamental group and classification give that $S_{nc}$ is a
  cover of $\PU(1,n)$. Since the dimension of $\mathfrak{g}=
  \mathfrak{sp}(4, \RR)$ is $10$ we have $n \leq 2$. If the
  $8$-dimensional Lie algebra $\mathfrak{su}(1,2)$ were a subalgebra
  of the $10$-dimensional Lie algebra $\mathfrak{g}$, this would give
  a $\mathfrak{su}(1,2)$-module of dimension $2=10-8$ hence a morphism
  $\mathfrak{su}(1,2) \to \mathfrak{gl}(2, \RR)$, this is
  impossible. Hence $\mathfrak{s}_{nc} \cong \mathfrak{su}(1,1) \cong
  \mathfrak{sl}( 2, \RR)$ and $S_{nc}$ is a finite cover of $\PSL(2,
  \RR)$.

  Embeddings of $\mathfrak{sl}( 2, \RR)$ into $\mathfrak{g}$ are in
  correspondence with $\mathfrak{sl}( 2, \RR)$-modules of dimension
  $4$ together with an invariant symplectic form. Their list is the
  following\footnote{Here $V_n$ denote the irreducible $\mathfrak{sl}(
    2, \RR)$-module of dimension $n$, $V_{2k}$ has an invariant
    symplectic form $\omega_{2k}$ and for any vector space $V$, $V
    \oplus V^*$ has a natural symplectic form given by the pairing
    between $V$ and $V^*$. Hence the two modules $(V_2, \omega_2)
    \oplus (V_2, \omega_2)$ and $V_2 \oplus V_{2}^{*}$ although
    isomorphic as $\mathfrak{sl}( 2, \RR)$-modules are not isomorphic
    \emph{symplectic} $\mathfrak{sl}( 2, \RR)$-modules.}
  \begin{enumerate}
  \item \label{itempfZcl1} $(V_4, \omega_4)$, $(V_2, \omega_2) \oplus
    (V_2, \omega_2)$
  \item \label{itempfZcl2} $(V_2, \omega_2) \oplus V_1 \oplus
    V_{1}^{*}$
  \item \label{itempfZcl3} $V_2 \oplus V_{2}^{*}$, $V_1 \oplus
    V_{1}^{*}\oplus V_1 \oplus V_{1}^{*}$.
  \end{enumerate}
  Note that in any case there is a corresponding morphism $\SL(2, \RR)
  \to G$ so that (up to taking a finite cover of $\Sigma$) one can
  always suppose that $S_{nc} = \SL(2, \RR)$. The obstruction
  $o(\rho_{nc})$ is an integer bounded in absolute value by $g-1$ by
  Milnor-Wood inequality \cite[Th.~1]{Milnor}. However in the above
  list, the modules in (\ref{itempfZcl1}) induce multiplication by $2$
  from $\pi_1( \SL(2, \RR)) \cong \ZZ$ to $\pi_1(G) \cong \ZZ$, the
  module in (\ref{itempfZcl2}) gives the identity $\ZZ \to \ZZ$ and
  the module in (\ref{itempfZcl3}) the zero map. Since $\tau(\rho)$ is
  $2(g-1)$ this shows that $o(\rho_{nc}) = g-1$ and that the embedding
  $S_{nc} \to G$ is induced by one of the modules in
  (\ref{itempfZcl1}).

  It remains to analyze what $\mathfrak{l} = \mathrm{Lie}(L)$ could
  be. Clearly it is a sub-$\mathfrak{s}_{nc}$-module of $\mathfrak{g}$
  so we first need to describe $\mathfrak{g}$ as a $\mathfrak{sl}(2,
  \RR)$-module. In the first case of (\ref{itempfZcl1}), one easily
  find $\mathfrak{g}= V_3 \oplus V_7 = \mathfrak{s}_{nc} \oplus V_7$
  and hence $\mathfrak{l} = \mathfrak{r} = \mathfrak{s}_{nc}$ and $L$
  is the irreducible $\SL(2, \RR)$. In the second case of
  (\ref{itempfZcl1}), $\mathfrak{g}= V_3 \oplus V_1 \oplus V_3 \oplus
  V_3 = \mathfrak{s}_{nc} \oplus \mathfrak{z} \oplus V_3 \oplus V_3$
  where $\mathfrak{z} \cong \mathfrak{so}(2)$ is the centralizer of
  $\mathfrak{s}_{nc}$ in $\mathfrak{g}$. In this case one easily sees
  that the only nilpotent subalgebra invariant by $\mathfrak{s}_{nc}$
  is $\mathfrak{z}$, hence $\mathfrak{u}=\{ 0 \}$ and $\mathfrak{l} =
  \mathfrak{r}$. Since (by construction) $\mathfrak{s}_{nc} \subset
  \mathfrak{r} \subset \mathfrak{s}_{nc} \oplus \mathfrak{z}$ there
  are only two possibilities for $L$: either the diagonal $\SL(2,\RR)$
  or the product of this $\SL(2,\RR)$ and the identity component of
  its centralizer, i.e.\ $L= \SL(2, \RR) \times_{\{\pm 1\}} \SO(2)$.

  \smallskip

  We turn now to the case where $S_{nc}$ is of rank $2$. If $S_{nc}$
  were simple and a proper subgroup of $\Sp(4,\RR)$ then it had to be
  isomorphic to $\SL(3,\RR)$. This would imply that $\tau(\rho)
  =0$. Thus $S_{nc}$ is a product $S_1 \times S_2$. Again by a
  dimension count $\mathfrak{s}_1 \cong \mathfrak{s}_2 \cong
  \mathfrak{sl}(2, \RR)$. Since the centralizer of $\mathfrak{s}_1$ in
  $\mathfrak{g}$ is big, inspecting the above list
  (\ref{itempfZcl1})-(\ref{itempfZcl3}) implies that the only possible
  structure of $\mathfrak{s}_1$-module on $\RR^4$ is $(V_2, \omega_2)
  \oplus V_1 \oplus V_{1}^{*}$. The similar observation holds also for
  $\mathfrak{s}_2$. Consequently the embedding $\mathfrak{s}_1 \times
  \mathfrak{s}_2 \subset \mathfrak{g}$ is the embedding associated to
  a decomposition of $\RR^4= V \oplus W$ into two symplectic planes:
  \[ \mathfrak{s}_1 \times \mathfrak{s}_2 \cong \mathfrak{sp}(V)
  \times \mathfrak{sp}(W) \subset \mathfrak{sp}(V \oplus W) \cong
  \mathfrak{g}.\] The decomposition of $\mathfrak{g}$ as a
  $\mathfrak{s}_{nc}$-module is now $\mathfrak{g} = \mathfrak{s}_{nc}
  \oplus V \otimes W$ and the equality $\mathfrak{l} =
  \mathfrak{s}_{nc}$ follows. Hence $L$ is the group $\SL(2, \RR)
  \times \SL(2, \RR) \subset G$.
\end{proof}

\begin{remark}
  As a corollary we find that $L$ is always a reductive group of
  Hermitian type, that its centralizer is compact and that the
  embedding $L \to G$ is tight. This is a special case of a general
  result for maximal representations
  \cite[Theorem~4]{Burger_Iozzi_Wienhard_tight}. However this
  additional a priori knowledge on $L$ would not have simplified the
  proof of the lemma much.  Certainly classification of Lie groups can
  be avoided.
\end{remark}

\begin{proof}[Proof of Theorem~\ref{thmintro_Zdensity}]
  Let $\rho$ in $\homMaxFourZero$ be a representation such that
  $e_\gamma(\rho) \neq g(\Sigma)-1$. If $\Sigma' \to \Sigma$ is a
  $k$-fold cover and $\gamma'$ is the simple curve above $\gamma$,
  then $e_{\gamma'}( \rho|_{\pi_1(\Sigma')}) = k e_\gamma(\rho)$ and
  $g(\Sigma')-1 = k(g(\Sigma)-1)$ so that the assumptions of the
  theorem are still satisfied for $\Sigma'$. Therefore, passing to a
  finite cover, we can suppose that the Zariski closure $L$ of $\rho(
  \pi_1(\Sigma))$ is connected.

  We now inspect the different possibilities for $L$ given by
  Lemma~\ref{lem:possble_Zcl}. If $L$ is the irreducible or the
  diagonal $\SL(2,\RR)$ then, by Proposition~\ref{prop:eulerHybrid},
  we would have $e_\gamma( \rho) = g-1$. If $L= \SL(2, \RR)
  \times_{\{\pm 1\}} \SO(2)$ then $\rho$ could be deformed to a
  representation $\rho'$ into the diagonal $\SL(2,\RR)$ and the Euler
  class would be $e_\gamma(\rho) = e_\gamma(\rho') = g-1$. If $L =
  \SL(2, \RR) \times \SL(2, \RR)$ then $\rho = (\rho_1, \rho_2)$ is a
  pair of maximal representations $\pi_1( \Sigma) \to \SL(2,
  \RR)$. Connectedness of the Teichm\"uller space implies that $\rho$
  can be deformed to a representation $\rho'=(\rho_{1}^{\prime},
  \varepsilon \cdot \rho_{1}^{\prime})$ with $\rho_{1}^{\prime}:
  \pi_1( \Sigma) \to \SL(2, \RR)$ maximal and $\varepsilon: \pi_1(
  \Sigma) \to \{ \pm 1\}$. Up to taking a cover, $\rho'$ is a
  representation in the diagonal $\SL(2, \RR)$ and we would have
  $e_\gamma(\rho) = e_\gamma(\rho') = g-1$. Hence the only possibility
  is $L=\Sp(4, \RR)$ proving the Zariski density of $\rho$.
\end{proof}

\section{Action of the mapping class group}
\label{sec:mapgroup}

This section gives the proof of
Theorem~\ref{thm:components_quotient_intro} of the introduction, see
Corollaries~\ref{cor:comp_mod_n>2} and \ref{cor:comp_mod_n=2} below.

\smallskip

It is known that the action of the mapping class group $\map(\Sigma)$
on
\[
\repmax(\pi_1(\Sigma), \Sp(2n,\RR)) = \hommax(\pi_1(\Sigma),
\Sp(2n,\RR))/\Sp(2n,\RR)
\]
is properly discontinuous \cite[Theorem 1.0.2]{Labourie_energy}
\cite[Theorem 1.1]{Wienhard_mapping}. Furthermore the mapping class
group acts naturally on $\h^i(T^1\Sigma; \FF_2)$ and the maps
\begin{equation*}
  sw_i: \repmax(\pi_1(\Sigma), \Sp(2n,\RR)) \longrightarrow \h^i(T^1\Sigma; \FF_2)
\end{equation*}
are equivariant.

Therefore, understanding the action of $\map(\Sigma)$ on the subspaces
of $\h^i(T^1\Sigma; \FF_2)$ where the Stiefel-Whitney classes take
their values in, allows us to determine the number of connected
components of the quotient of $\repmax(\pi_1(\Sigma), \Sp(2n,\RR))$ by
$\map(\Sigma)$.

\begin{lem}
  Let $A = \ZZ \text{ or } \FF_2$. The mapping class group of $\Sigma$
  acts trivially on the image of $\h^2(\Sigma; A)$ in $\h^2(T^1\Sigma;
  A)$.
\end{lem}
\begin{proof}
  This is immediate from the fact that $\map(\Sigma)$ preserves the
  image of $\h^2(\Sigma; A)$ in $\h^2(T^1\Sigma; A)$ and acts
  trivially on $\h^2(\Sigma; A)$.
\end{proof}

For $\h^1(T^1\Sigma; \FF_2)$ the picture is a bit more complicated.
\begin{prop}\label{prop:mapgroup}
  The action of the mapping class group on $\h^1(T^1\Sigma; \FF_2)$
  preserves the two cosets $\h^1(\Sigma; \FF_2)$ and $\h^1(T^1\Sigma;
  \FF_2) \moins\h^1(\Sigma; \FF_2)$.
  \begin{enumerate}
  \item On $\h^1(\Sigma; \FF_2)$ the action of the mapping class group
    has two orbits, $\{0\}$ and $\h^1(\Sigma; \FF_2) \moins \{0\}$.
  \item On $\h^1(T^1\Sigma; \FF_2) \moins\h^1(\Sigma; \FF_2)$ the
    action of the mapping class group has two orbits.
  \end{enumerate}
\end{prop}
\begin{proof}
  It is obvious that the mapping class group of $\Sigma$ preserves the
  two cosets $\h^1(\Sigma; \FF_2)$ and $\h^1(T^1\Sigma; \FF_2)
  \moins\h^1(\Sigma; \FF_2)$.

  On $\h^1(\Sigma; \FF_2)\cong \FF_2^{2g}$ the mapping class group
  acts by symplectomorphisms and it is a classical fact that it
  generates the whole group $\Sp(2n,\FF_2)$
  \cite[Prop.~7.3]{Farb_Margalit}. This action has two orbits, $0$ and
  $ \FF_2^{2g} \moins \{0\}$.

  The action of the mapping class on the space of spin structures
  $\h^1(T^1\Sigma; \FF_2) \moins\h^1(\Sigma; \FF_2)$ has two orbits
  \cite[Corollary~2]{Johnson}.
\end{proof}

\begin{cor}\label{cor:comp_mod_n>2}
  Let $n\geq 3$. The space $\repmax(\pi_1(\Sigma),
  \Sp(2n,\RR))/\map(\Sigma)$ has $6$ connected components.
\end{cor}
\begin{proof}
  Let us first consider the Hitchin components $\repHit(\pi_1(\Sigma),
  \Sp(2n,\RR))$. This subset is invariant by the action of
  $\map(\Sigma)$. The $2^{2g}$ different Hitchin components are
  indexed by elements $ sw_{1}^{A} \in \h^1(T^1\Sigma; \FF_2)
  \moins\h^1(\Sigma; \FF_2)$ (see
  Section~\ref{sec:invforHitchin}). The action of the mapping class
  group has $2$ orbits on $\h^1(T^1\Sigma; \FF_2) \moins\h^1(\Sigma;
  \FF_2)$. Therefore the quotient $\repHit(\pi_1(\Sigma),
  \Sp(2n,\RR))/\map(\Sigma)$ has $2$ connected components.

  The $2\times 2^{2g}$ connected components of
  \[\repmax(\pi_1(\Sigma), \Sp(2n,\RR)) \moins \repHit(\pi_1(\Sigma),
  \Sp(2n,\RR))\] are indexed by the first and second Stiefel-Whitney
  classes. The mapping class group acts trivially on the image of
  $\h^2(\Sigma; \FF_2)$ in $\h^2(T^1\Sigma; \FF_2)$ and has two orbits
  in the coset of $\h^1(\Sigma; \FF_2)$ in $\h^1(T^1\Sigma;
  \FF_2)$. This implies that it has $4$ orbits in the set of
  $n$-admissible pairs (Definition~\ref{defi:admissible}), hence the
  quotient
  \[\big(\repmax(\pi_1(\Sigma), \Sp(2n,\RR)) \moins
  \repHit(\pi_1(\Sigma), \Sp(2n,\RR))\big) / \map(\Sigma)\] has $4$
  connected components.
\end{proof}

\subsection{The case of  $\Sp(4,\RR)$}\label{sec:mapgroup_n=2}
To define the Euler class of a representation $\rho: \pi_1(\Sigma) \to
\Sp(4,\RR)$ with $sw_1(\rho)=0$ we had to make several choices (see
Section~\ref{sec:sp4case}). In particular, we fixed a nontrivial
element $\gamma \in \pi_1(\Sigma)$ to define the space $\homMaxPLnot$
and the Euler class $e_{\gamma, \LZPlus}$. Denoting $\repMaxFourZero$
the quotient of $\homMaxFourZero$ by the action of $\Sp(4, \RR)$, the
equivariance of
\[ e_{\gamma, \LZPlus} : \repMaxFourZero \longrightarrow \h^2( T^1
\Sigma ; \ZZ)\] only holds for the subgroup $\stab(\gamma) \subset
\map( \Sigma)$ of the mapping class group stabilizing the homotopy
class of $\gamma$.

This implies
\begin{prop}
  \label{prop:MCGonSWzerow}
  Let $\mathcal{C}$ be a component of $\repMaxFourZero \moins
  \repHitFour$.  Then $\mathcal{C}$ is sent to itself by the action of
  the mapping class group of $\Sigma$.
\end{prop}

\begin{proof}
  It is a classical fact that the mapping class group $\map( \Sigma)$
  is generated by Dehn twists along simple closed curves (see for
  example \cite[Theorem~4.2.D]{Ivanov_Mapping}). As a consequence
  $\map( \Sigma)$ is generated by $\{ \stab( \gamma) \}_{\gamma \in
    \pi_1( \Sigma)\moins\{1\}}$. With this said it is enough to show
  the invariance of $\mathcal{ C}$ under $\stab( \g)$. However we know
  that $\repmaxzero \moins \repHit$ is invariant under $\map( \Sigma)$
  and that
  \begin{multline*}
    e_{\gamma, \LZPlus} : \pi_0(\repMaxFourZero \moins \repHitFour) \\
    \longrightarrow \h^2( T^1 \Sigma ; \ZZ)^{tor}
  \end{multline*}
  is a bijection. The equivariance property of $e_{\gamma, \LZPlus}$
  with respect to $\stab(\gamma)$ implies the claim.
\end{proof}

\begin{cor}\label{cor:comp_mod_n=2}
  The space $\repmax(\pi_1(\Sigma), \Sp(4,\RR))/\map(\Sigma)$ has
  $2g+2$ connected components.
\end{cor}
\begin{proof}
  Let us write $\repmax(\pi_1(\Sigma), \Sp(4,\RR))$ as a union of
  subspaces
  \begin{multline*}
    \repmax(\pi_1(\Sigma), \Sp(4,\RR)) =\\
    \repmaxneqzero \bigcup \repHit(\pi_1(\Sigma), \Sp(4,\RR))
    \bigcup \\
    \big((\repMaxFourZero \moins \repHit(\pi_1(\Sigma),
    \Sp(4,\RR))\big).
  \end{multline*}
  The mapping class group preserves this decomposition. Since
  $\map(\Sigma)$ has two orbits on $\h^1(T^1\Sigma;\FF_2) \moins
  \h^1(\Sigma; \FF_2)$, the space $ \repHit(\pi_1(\Sigma),
  \Sp(4,\RR))/\map(\Sigma)$ has $2$ connected components.  The mapping
  class group acts transitively on $\h^1(\Sigma; \FF_2) \moins \{0\}$
  and trivially on the image of $\h^2(\Sigma; \FF_2) \subset
  \h^2(T^1\Sigma; \FF_2)$, thus the space
  \[\repmaxneqzero\]
  has $2$ connected components. By Proposition~\ref{prop:MCGonSWzerow}
  the mapping class group stabilizes the others components.  This
  gives a total of $2g+2$ connected components.
\end{proof}

\section{Holonomy of maximal representations}
\label{sec:HoloOfMaxRep}

In this section we prove Theorem~\ref{thm_intro:holonomy} of the
introduction.

\smallskip

Let $\rho: \pi_1(\Sigma) \to \Sp(2n,\RR)$ be a maximal representation.
As already noted in Corollary~\ref{cor:anosov_holonomy} for any
nontrivial $\gamma$ the element $\rho(\gamma)$ is (conjugate to) an
element of $\GL(n, \RR)$ whose eigenvalues are in absolute value
bigger than $1$.  For representations in the Hitchin components,
Corollary~\ref{cor:anosov_holonomy} implies moreover that
$\rho(\gamma)$ is semi-simple. The following statement shows that we
cannot expect anything similar for maximal representations in general.

\begin{thm}
  \label{thm:holonomy}
  Let $\hH$ be a connected component of \[\repmax(\pi_1(\Sigma),
  \Sp(2n,\RR)) \moins \repHit(\pi_1(\Sigma), \Sp(2n,\RR)),\] and let
  $\g$ be an element in $\pi_1(\Sigma)\moins\{1\}$ corresponding to a
  simple curve.

  If $n=2$, the genus of $\Sigma$ is $2$ and $\g$ is separating, we
  require that $\hH$ is not the connected component determined by
  $sw_1 = 0$ and $e_{\gamma} = 0$.

  Then there exist
  \begin{enumerate}
  \item a representation $\rho \in \hH$ such that the Jordan
    decomposition of $\rho(\g)$ in $\GL(n,\RR)$ has a nontrivial
    parabolic component.
  \item a representation $\rho' \in \hH$ such that the Jordan
    decomposition of $\rho'(\g)$ in $\GL(n,\RR)$ has a nontrivial
    elliptic component.
  \end{enumerate}
\end{thm}

We first establish some preliminary results towards the proof of the
theorem.

\begin{lem}
  \label{lem:surjOfDiff}
  Let $\Sigma$ be a surface with one boundary component $\g = \partial
  \Sigma$. Let $G$ be a semisimple Lie group and $\rho_0: \pi_1 (
  \Sigma) \to G$ a representation whose centralizer in $G$ is finite.

  Then the differential of the map \[
  \begin{array}{rcl}
    \hom( \pi_1( \Sigma) , G) & \overset{ \tau_\g}{\longrightarrow} & G \\
    \rho & \longmapsto & \rho( \g)
  \end{array}
  \] at the point $\rho_0$ is surjective.
\end{lem}

\begin{proof}
  One can always find a set $\{ a_1, \dots, a_k, b_1, \dots, b_k\}$
  freely generating $\pi_1( \Sigma)$ and such that $\gamma = [ a_1,
  b_1] \cdots [ a_k, b_k]$. The result is then simply a reformulation
  of \cite[Proposition~3.7]{Goldman_symplectic}.
\end{proof}

\begin{lem}
  \label{lem:existGoodRep}
  Let $\gamma$ be a simple closed separating curve on a closed surface
  $\Sigma$; denote by $\Sigma_1$ and $\Sigma_2$ the components of
  $\Sigma\moins \g$. Let $\hH$ be a component of $\hommax( \pi_1(
  \Sigma), \Sp( 2n, \RR))$ such that the conditions of
  Theorem~\ref{thm:holonomy} are satisfied.

  Then there exists $\rho$ in $\hH$ such that
  \begin{itemize}
  \item $\rho( \gamma)$ (considered as an element of $\GL(n, \RR) <
    \Sp( 2n, \RR)$) is a multiple of the identity,
  \item the restriction of $\rho$ to $\pi_1( \Sigma_1)$ (resp. $\pi_1(
    \Sigma_2)$) has finite centralizer.
  \end{itemize}
\end{lem}

\begin{proof}
  By Theorem~\ref{thm_intro:symp_n>2} and
  Theorem~\ref{thm_intro:symp_n=2} we only need to prove that there
  are representations satisfying the conclusions of the lemma in a
  neighborhood of a model representation (i.e.\ a standard maximal
  representation or a hybrid representation).

  First consider a diagonal Fuchsian representation $\rho_0=
  \phi_\Delta \circ \iota$.  There exist deformations $\iota_{1,t},
  \dots, \iota_{n,t}$ ($t \in [0,1]$) such that $\iota_{i,t}( \g) =
  \iota( \g)$ and $\iota_{i,0} = \iota$ and, for all $t>0$, the
  representation $\rho_t= ( \iota_{1,t}, \dots, \iota_{n,t})$ sends
  $\pi_1( \Sigma_1)$ and $\pi_1( \Sigma_2)$ into a Zariski dense
  subgroup of $\SL(2, \RR)^n < \Sp(2n, \RR)$ (this construction was
  already used in \cite[Section~9]{Burger_Iozzi_Wienhard_toledo}). As
  the centralizer of $\SL(2, \RR)^n$ inside $\Sp(2n, \RR)$ is finite,
  the statement of the lemma follows.  \medskip

  Now consider a twisted diagonal representation $\rho_0 = \iota
  \otimes \Theta$ which cannot be deformed to a diagonal Fuchsian
  representation.  By Lemma~\ref{lem:all_classes_3} it is sufficient
  to consider the case when $\Theta$ has finite image in $\O(2) <
  \O(n)$ or in $\SO(3) < \SO(n)$. In the first case, the
  representation $\rho_0$ takes values in $\Sp(4, \RR) \times \SL(2,
  \RR)^{n-2} < \Sp(2n, \RR)$; we write $\rho_0 = ( \iota \otimes
  \Theta, \iota, \dots, \iota)$. As above, we can find a deformation
  $\rho_t = ( \iota_{1,t} \otimes \Theta, \iota_{2,t}, \dots
  ,\iota_{n-1, t})$ of $\rho_0$ such that, for all $t$, $\rho_t( \g) =
  \rho( \g)$ and, for all $t>0$, the Zariski closure of $\rho_t(
  \pi_1( \Sigma_1))$ contains $\phi_\Delta( \SL(2, \RR)) \times \SL(2,
  \RR)^{n-2} < \Sp(4, \RR) \times \SL(2, \RR)^{n-2}< \Sp(2n,
  \RR)$. The same is true for $\rho_t( \pi_1( \Sigma_2))$. This
  already means that the centralizer of $\rho_t( \pi_1( \Sigma_1))$ is
  contained in $\O(2) < \O(n) < \Sp(2n, \RR)$; it also implies that
  the image $\Theta( \pi_1( \Sigma_1))$ is contained in the Zariski
  closure of $\rho_t( \pi_1 ( \Sigma_1))$.

  Suppose now that the image of $\pi_1( \Sigma_1)$ by $\Theta$ is not
  contained in $\SO(2)$. Then there exists a reflection $R \in \O(2)
  \moins \SO(2)$ that belongs to the Zariski closure $\overline{
    \rho_t( \pi_1( \Sigma_1))}^Z$. Therefore the centralizer of
  $\overline{ \rho_t( \pi_1( \Sigma_1))}^Z$ (which equals the
  centralizer of $\rho_t( \pi_1( \Sigma_1))$) will be contained in the
  centralizer of $R$ in $\O(2)$ which is finite.

  If the restriction of $\Theta$ to $\pi_1( \Sigma_1)$ is contained in
  $\SO(2)$, we can suppose that this restriction is the trivial
  representation.  In this situation we write $\rho_0$ as amalgamated
  representation $\rho_0 = \rho_1 * \rho_2$, where $\rho_i =
  \rho_0|_{\pi_1(\Sigma_i)}$. Then the restriction of $\Theta$ to
  $\pi_1(\Sigma_2)$ is not contained in $\SO(2)$ (otherwise $\rho_0$
  would be in the same connected component as a diagonal Fuchsian
  representation).  We then deform $\rho_0$ as an amalgamated
  representation: $\rho^{(1)}_t * \rho^{(2)}_t$, where $\rho^{(1)}_t$
  is a deformation of $\rho_1$ considered as a diagonal Fuchsian
  representation (the first case we investigated) and $\rho^{(2)}_t $
  is a deformation of the twisted diagonal representation $\rho_2$.
  The centralizer of $\pi_1( \Sigma_2)$ is finite by the same argument
  as above.

  The case when $\Theta$ is in $\SO(3)$ with finite image can be
  treated in a similar way and is left to the reader (observe that,
  when $\Theta$ does not lift to $\mathrm{Spin}(3)$, the centralizer
  of $\Theta$ is $\{\pm \id_3\}$).

  \smallskip

  Let us now assume that $n=2$ and and $\rho_0$ is a hybrid
  representation in $\homMaxFourZero \moins \homHitFour$.  The
  definition of a hybrid representation $\rho$ involves the choice of
  a subsurface $\Sigma'$ in $\Sigma$, for whose fundamental group we
  choose an irreducible Fuchsian representation. Since we want the
  holonomy around $\gamma$ to be a multiple of the identity in
  $\GL(2,\RR)$ the curve $\gamma$ has to be contained in $\Sigma
  \moins \Sigma'$ and not homotopic to a boundary component of
  $\Sigma'$. This requires the Euler characteristic $\chi(\Sigma')$ to
  be different from $3-2g$.

  The way hybrid representations are constructed in
  Section~\ref{sec:amalagamatedRep}, and with the hypothesis on $\g$
  and $\Sigma'$, one can ensure that $\rho( \gamma)$ is a multiple of
  the identity in $\GL(2, \RR)< \Sp(4, \RR)$ and that the Zariski
  closure $\overline{ \rho( \pi_1( \Sigma_1))}^Z$ (resp. $\overline{
    \rho( \pi_1( \Sigma_2))}^Z$) contains $\SL(2, \RR) \times \SL(2,
  \RR)$ or an irreducible $\SL(2, \RR)$ in $\Sp(4,\RR)$. Therefore,
  the hybrid representation $\rho$ satisfies all the desired
  conclusions.  Considering not only hybrid representations
  constructed from positively adjusted pairs as in
  Section~\ref{sec:description_hybrid}, but also such constructed from
  negatively adjusted pairs (Definition~\ref{defi_good_triples}) we
  get representations with Euler class $g-1 \pm\chi( \Sigma') \in
  \ZZ/(2g-2)\ZZ$ (see Proposition~\ref{prop:eulerHybrid} and
  Remark~\ref{rem:negatively_adjusted}). Varying the subsurface
  $\Sigma'$, every Euler characteristic $\chi( \Sigma')$ different
  form $3-2g$ can be attained, hence, when $g>2$, by
  Theorem~\ref{thm_intro:n=2_components} we obtain representations in
  any connected component of
  \[
  \homMaxFourZero \moins \homHitFour.
  \]
  Only in the case when $g=2$ the above construction does not give
  representations with Euler class $e_\gamma = 0$.
\end{proof}

\begin{proof}[Proof of Theorem~\ref{thm:holonomy}]
  Suppose that $\gamma$ is a separating curve and let $\rho_0$ in
  $\hH$ be a representation satisfying the conclusions of
  Lemma~\ref{lem:existGoodRep}. Let us denote by $\Sigma_1$ and
  $\Sigma_2$ the components of $\Sigma \moins \g$. We call $\sigma_1$
  and $\sigma_2$ the two evaluation maps:
  \[
  \begin{array}{rcl}
    \sigma_i : \hom( \pi_1( \Sigma_i), \Sp(2n, \RR)) & \longrightarrow &
    \Sp(2n, \RR) \\
    \rho & \longmapsto & \rho( \g).
  \end{array}
  \]
  The representation space for $\pi_1( \Sigma)$ is the fiber product
  of the representation space for $\pi_1( \Sigma_1)$ and $\pi_1(
  \Sigma_2)$ over $\Sp(2n, \RR)$:
  \begin{multline*}
    \hom( \pi_1( \Sigma), \Sp(2n, \RR)) = \\
    \big \{ (\rho_1, \rho_2) \in \hom( \pi_1( \Sigma_2), \Sp(2n, \RR))
    \times \hom( \pi_1( \Sigma_2), \Sp(2n, \RR)) \mid \sigma_1(
    \rho_1) = \sigma_2( \rho_2) \big \}.
  \end{multline*}
  By Lemma~\ref{lem:surjOfDiff} the map $\sigma_1$ (resp. $\sigma_2$)
  is locally surjective in a neighborhood of $\rho_0|_{ \pi_1(
    \Sigma_1)}$ (resp. $\rho_0|_{ \pi_1( \Sigma_2)}$). This implies
  that the map
  \[ \begin{array}{rcl} \sigma : \hom( \pi_1( \Sigma), \Sp(2n, \RR)) &
    \longrightarrow &
    \Sp(2n, \RR) \\
    \rho & \longmapsto & \rho( \g) \end{array} \] is locally
  surjective in a neighborhood of $\rho_0$. As $\sigma( \rho_0) =
  \rho_0( \g)$ is a multiple of the identity in $\GL(n , \RR)$ we
  obtain representations $\rho$ and $\rho'$ in a neighborhood of
  $\rho_0$ having the desired properties.

  \smallskip

  When $\gamma$ is not separating, the proof follows a similar
  strategy. Let $\eta$ a simple closed curve separating on $\Sigma$
  such that $\Sigma \moins \eta = \Sigma_1 \cup \Sigma_2$ with
  $\Sigma_1$ a once punctured torus containing $\gamma$. In
  $\mathcal{H}$ one finds easily a representation $\rho_0$ such that
  $\rho_0(\gamma)$ is a multiple of $\id_n$ in $\GL(n, \RR)$ and such
  that $\rho_0 |_{\pi_1(\Sigma_2)}$ satisfies the hypothesis of
  Lemma~\ref{lem:surjOfDiff}. Therefore any small enough deformation
  of $\rho_0 |_{\pi_1(\Sigma_1)}$ can be extended to a deformation of
  $\rho_0$. Since $\gamma$ can be made an element of a subset freely
  generating the free group $\pi_1(\Sigma_1)$, the result follows.
\end{proof}

\appendix

\section{Maximal representations}\label{sec:app_max}

\subsection{The space of positive curves}
\label{sec:app_positive}
In this section we establish certain connectedness properties of the
space of positive curves into the Lagrangian Grassmannian.

We will use the notation from Section~\ref{sec:prelim_maximal}:
$\RR^{2n}$ is a symplectic vector space, with symplectic basis $(e_i)
_{i=1, \dots, 2n}$; $\xX \subset \mathcal{L} \times\mathcal{L}$ is the
space of pairwise transverse Lagrangian subspaces of $\RR^{2n}$,
$L_0^s = \Span ( e_i)_{1 \leq i \leq n}$, $L_0^u = \Span(
e_i)_{n+1\leq i \leq 2n}$ are two transverse Lagrangian subspaces of $
\RR^{2n}$, $P^s, P^u \subset \Sp(2n,\RR)$ are their stabilizers. The
unipotent radical of $P^s$ is
\begin{equation*}
  U^s = \left \{ u^s(M) = \left (
      \begin{array}{cc}
        \id_n & M \\ 0 & \id_n
      \end{array} \right ) \mid M \in \mathrm{M}(n , \RR),
    \transpose{} M = M
  \right \}.
\end{equation*}
A Lagrangian $L$ can be written as $u^s( M) \cdot L_0^u$ for some $M$
if and only if $L$ and $L_0^s$ are transverse, in which case $M$ is
uniquely determined by $L$. The triple of Lagrangians $(L_0^s, L,
L_0^u)$ is positive (Definition~\ref{def:positivetriple}) if and only
if the symmetric matrix $M$ such that $L = u^s( M) \cdot L_0^u$ is
positive definite.

Recall that a curve $\xi: S^1 \to \mathcal{L}$ is said to be
\emph{positive} if it sends every positive triple of $S^1$ to a
positive triple of Lagrangians.

\begin{prop}
  \label{prop:conn-posit-curv}
  The space $\mathcal{P}$ of \emph{continuous and positive} curves
  from $S^1$ to $\mathcal{L}$ is connected. In fact, fixing two points
  $x^s \neq x^u$ in $S^1$, the fibers of the map
  \begin{equation*}
    \begin{array}{ccc}
      \mathcal{P} & \longrightarrow & \mathcal{X} \subset \mathcal{L}
      \times \mathcal{L} \\
      \xi & \longmapsto & ( \xi( x^s), \xi( x^u)) 
    \end{array}
  \end{equation*}
  are contractible.
\end{prop}

\begin{proof}
  Consider the set
  \begin{equation*}
    \mathcal{P}_0 = \{ \xi \in \mathcal{P} \mid \xi( x^s) = L^{s}_{0}, \xi(x^u) = L^{u}_{0} \}.
  \end{equation*}
  Since $\Sp(2n,\RR)$ acts transitively on $\mathcal{X}$, it is enough
  to show that $\mathcal{P}_0$ is contractible.  For every $\xi$ in
  $\mathcal{P}_0$ and every $x \neq x^s$, $\xi(x)$ and $\xi( x^s)$ are
  transverse, therefore we can regard $\mathcal{P}_0$ as a subset of
  the space of maps from $S^1 \moins \{ x^s \}$ to $\mathrm{Sym}(n,
  \RR)$. It is precisely the set of maps $\tilde{ \xi} : S^1 \moins \{
  x^s \} \to \mathrm{Sym}(n, \RR)$ such that
  \begin{itemize}
  \item $\lim_{ x \to x^s } u^s( \tilde{ \xi}( x) ) \cdot L_{0}^u =
    L_{0}^s$
  \item $\tilde{ \xi}$ is continuous
  \item for all $x>x'$ the symmetric matrix $\tilde{ \xi}( x) -
    \tilde{ \xi}( x')$ is positive definite.
  \item $\tilde{ \xi}( x^u) = 0$.
  \end{itemize}
  This set of maps is a convex subset of the space of all maps from ${
    S^1 \moins \{ x^s \}}$ into $\mathrm{Sym}(n, \RR)$ (this follows
  from the fact that the set of positive definite matrices is
  convex). The contractibility of $\mathcal{P}_0$ follows.
\end{proof}

\begin{prop}
  \label{prop:conn_equiv_posit}
  Let $\gamma$ be a nontrivial element of $\pi_1(\Sigma)$. Then the
  space of pairs
  \begin{equation*}
    \mathcal{P}^\gamma = \big \{ (\rho, \xi) \in \hom( \langle \gamma \rangle, \Sp(2n,\RR))
    \times \mathcal{P} \mid \xi \text{ is } \rho\text{-equivariant} \big \}
  \end{equation*}
  has two connected components. If $t^{s}_{\gamma}$ and
  $t^{u}_{\gamma}$ are the fixed points of $\gamma$ in $\partial
  \pi_1(\Sigma)$, then the connected components are detected by which
  component of $ \mathrm{Stab}( \xi( t^{s}_{\gamma} ) ) \cap
  \mathrm{Stab}( \xi( t^{u}_{\gamma}) ) \cong \GL(n, \RR)$ contains
  $\rho( \gamma)$.
\end{prop}

\begin{proof}
  There are at least two connected components since the sign of $\det(
  \rho( \gamma) |_{\xi( t^{s}_{\gamma}) })$ varies continuously.

  It is sufficient to understand the connected components of the
  fibers of the map
  \begin{equation*}
    \begin{array}{rcl}
      \phi :\mathcal{P}^\gamma & \longrightarrow & \mathcal{X} \\
      ( \rho, \xi) & \longmapsto & ( \xi( t^{s}_{\gamma} ), \xi( t^{u}_{\gamma} ) ). 
    \end{array}
  \end{equation*}
  Again it is enough to calculate the components of
  $\mathcal{P}^{\gamma}_{0} := \phi^{-1}( L_{0}^s, L_{0}^u)$. The
  points $t^{s}_{\gamma}$ and $t^{u}_{\gamma}$ divide the circle
  $\partial \pi_1(\Sigma)$ in two intervals $I_{su}$ and $I_{us}$:
  they are chosen so that $x$ belongs to $I_{su}$ (respectively
  $I_{us}$) if and only if the triple $(t^{s}_{\gamma}, x ,
  t^{u}_{\gamma})$ (respectively $(t^{u}_{\gamma}, x ,
  t^{s}_{\gamma})$) is positively oriented. These two intervals are
  homeomorphic to $\RR$ and isomorphisms are chosen so that the action
  of $\gamma$ is conjugate to $t \mapsto t+1$ on $\RR$.

  It is not difficult to show that a curve $\xi : \partial
  \pi_1(\Sigma) \to \mathcal{L}$, such that $\xi( t^{s}_\gamma ) =
  L_{0}^{s}$ and $\xi( t^{u}_\gamma ) = L_{0}^{u}$, is positive if and
  only if the following conditions are satisfied:
  \begin{itemize}
  \item for all $x$ in $I_{su}$ the triple $(L_{0}^s, \xi(x),
    L_{0}^u)$ is positive.
  \item for all $x$ in $I_{us}$ the triple $(L_{0}^u, \xi(x),
    L_{0}^s)$ is positive.
  \item the restriction of $\xi$ to $I_{su}$ is positive (i.e.\ it
    sends positive triples to positive triples).
  \item the restriction of $\xi$ to $I_{us}$ is positive.
  \end{itemize}
  
  Therefore we can consider the two intervals $I_{su}$ and $I_{us}$
  separately. Using the parametrization by symmetric matrices, it is
  sufficient to show that the set
  \begin{multline*}
    \mathcal{S} = \big \{ (A, \tilde{ \xi}) \in \GL(n, \RR) \times
    C^0( \RR, \mathrm{Sym}_{>0}(n, \RR)) \mid \\ \tilde{ \xi}( t+1) =
    A \tilde{ \xi}( t) \transpose{}\! A, \text{ and } \forall s<t,\,
    \tilde{ \xi}( t) - \tilde{ \xi}( s) > 0 \big \}
  \end{multline*}
  has two connected components that are distinguished by the sign of
  $\det A$. (Note that the $\rho$-equivariance of $\xi$ guarantees
  that $\lim_{t \to \infty} u^s( \tilde{\xi}(t)) \cdot L_{0}^u =
  L_{0}^s$ and $\lim_{t \to -\infty} u^s( \tilde{\xi}(t)) \cdot
  L_{0}^u = L_{0}^s$). Taking into account the natural action of
  $\GL(n, \RR)$ on $\mathcal{S}$ reduces the question to determining
  the connected components of the subset $\mathcal{S}_0 := \{ (A,
  \tilde{ \xi} ) \in \mathcal{S} \mid \tilde{ \xi}(0) = \id_n \}$. The
  map
  \begin{equation*}
    \begin{array}{rcl}
      \mathcal{S}_0 & \longrightarrow & \GL(n, \RR) \\
      ( A, \tilde{ \xi} ) & \longmapsto & A.
    \end{array}
  \end{equation*}
  has convex, hence contractible fibers. Its image is
  \begin{equation*}
    \big \{ A \in \GL(n, \RR) \mid A \transpose{}\! A - \id_n \in
    \mathrm{Sym}_{>0}(n, \RR) \big \}.
  \end{equation*}
  Using the Cartan decomposition of $\GL(n,\RR)$, it is easy to show
  that this set has precisely two connected components given by the
  sign of $\det A$.
\end{proof}

Note that the proof gives the following

\begin{prop}
  \label{prop:conn_equiv_posit2}
  Let $\gamma$ be a nontrivial element of $\pi_1(\Sigma)$ and $\rho $
  a representation $\langle \gamma \rangle \to \Sp(2n,\RR))$. Then the
  space
  \[ \mathcal{P}^\rho = \big \{ \xi \in \mathcal{P} \mid \xi \text{ is
  } \rho\text{-equivariant} \big \} \] is connected.
\end{prop}

\subsection{Deforming maximal representations in $\SL(2,\RR)$}
The following fact follows from classical Fricke-Klein theory using
Fenchel-Nielsen coordinates.
\begin{lem}\label{lem:deformation}
  Let $(\gamma_i)_{i=1,\ldots, k}$ be a family of pairwise
  nonhomotopic simple closed curves on $\Sigma$ and denote by
  $\gamma_i \in \pi_1(\Sigma)$ the corresponding elements of the
  fundamental group. Let $\iota_0: \pi_1(\Sigma) \to \SL(2,\RR)$ a
  discrete embedding with $\iota_0(\gamma_i) = \varepsilon_i g_{i,0}
  \left(\begin{array}{cc} e^{\lambda_{i,0}} & 0 \\ 0
      &e^{-\lambda_{i,0}}
    \end{array}\right) g_{i,0}^{-1}$, $\varepsilon_i \in \{ \pm 1\}$,
  $\lambda_{i,0} \in \RR \moins \{ 0 \}$, $g_{i,0} \in \SL(2,
  \RR)$. Let $(\lambda_{i,t})_{t\in [0,1]}$ be continuous paths in
  $\RR \moins \{ 0 \}$.
 
  Then there exists a continuous path of discrete embeddings
  $(\iota_t)_{t\in [0,1]}$, and continuous paths $g_{i,t}$ such that
  for any $t \in [0,1]$, $\iota_t(\gamma_i) = \varepsilon_i g_{i,t}
  \left(\begin{array}{cc} e^{\lambda_{i,t}} & 0 \\ 0
      &e^{-\lambda_{i,t}}
    \end{array}\right) g_{i,t}^{-1}$.
\end{lem}

\subsection{Twisting representations}
\label{sec:trick}

In this section we explain the strategy which we used to calculate the
topological invariants for maximal representations.

\subsubsection{The group $\widehat{\pi_1(\Sigma)}$ }
\label{sec:group}
We fix a discrete embedding of $\pi_1(\Sigma)$ into $\PSL(2, \RR)$:
$\pi_1(\Sigma) < \PSL(2, \RR)$. Let $\pi : \SL(2, \RR) \to \PSL(2,
\RR)$ be the projection.

We set $\widehat{\pi_1(\Sigma)} = \pi^{-1}( \pi_1(\Sigma)) \subset
\SL(2,\RR)$. The group $\widehat{\pi_1(\Sigma)}$ is a two-to-one cover
of $\pi_1(\Sigma)$, which is isomorphic to $ \{ \pm 1 \} \times
\pi_1(\Sigma)$. The isomorphism can be chosen so that it intertwines
$\pi$ with the second projection $\{ \pm 1 \} \times \pi_1(\Sigma) \to
\pi_1(\Sigma)$. Any choice of such an isomorphism amounts to choosing
a lift of $\pi_1(\Sigma)<\PSL(2, \RR)$ to $\SL(2, \RR)$; such lifts
are in one-to-one correspondence with spin-structures on $\Sigma$.
  
For the rest of this section we fix such an isomorphism
$\widehat{\pi_1(\Sigma)} = \{ \pm 1 \} \times \pi_1(\Sigma)$.
  
\subsubsection{Maximal representation of $\widehat{ \pi_1(\Sigma)}$}
\label{sec:maxrepGammaBar}

\begin{defi}
  \label{defi_maxhat}
  A representation $\widehat{\rho} : \widehat{\pi_1(\Sigma)} = \{ \pm
  1 \} \times \pi_1(\Sigma) \to \Sp(4, \RR)$ is said to be
  \emph{maximal} if the restriction $\widehat{\rho}|_{\pi_1(\Sigma)}$
  is maximal (see Definition~\ref{defi:maximal}).
  
  The set of maximal representation is denoted by $\hommax( \widehat{
    \pi_1(\Sigma)}, \Sp(4, \RR))$.
\end{defi}

Let $\widehat{\rho} : \widehat{ \pi_1(\Sigma)} \to \Sp( 4, \RR)$ be a
maximal representation and $\varepsilon : \widehat{ \pi_1(\Sigma)} \to
\{ \pm 1 \}$ be any representation, then the representation
$\varepsilon \cdot \widehat{\rho}$, defined by $\g \mapsto
\varepsilon( \g) \widehat{\rho}( \g)$, is also maximal.

If $\rho : \pi_1(\Sigma) \to \Sp(4, \RR)$ is a maximal representation,
then $\widehat{ \rho} = \rho \circ pr_2 : \widehat{\pi_1(\Sigma)} \to
\Sp(4, \RR)$ is a maximal representation, where $pr_2 :
\widehat{\pi_1(\Sigma)} = \{ \pm 1 \} \times \pi_1(\Sigma) \to
\pi_1(\Sigma)$ denotes the projection onto the second factor.

Since $T^1 \Sigma \cong \widehat{ \pi_1(\Sigma)} \backslash \SL(2,
\RR)$ the notion of \emph{Anosov representations} and \emph{Anosov
  reductions} (see Section~\ref{sec_anosov}) can be easily extended to
representation of $\widehat{\pi_1(\Sigma)}$. The following lemma is an
immediate consequence of Theorem~\ref{thm:maximal_anosov}.

\begin{lem}
  Every maximal representation $\widehat{\rho}: \widehat{
    \pi_1(\Sigma)}\to \Sp(4,\RR)$ is Anosov.
\end{lem}

Similarly to the discussion in
Section~\ref{sec:TopologicalInvariants}, the Anosov reduction leads to
a map
\begin{equation*}
  \hommax( \widehat{ \pi_1(\Sigma)}, \Sp(4, \RR)) \xrightarrow{sw_1} \h^1( T^1 \Sigma ; \FF_2 )
\end{equation*}
Fixing a nontorsion element $\widehat{\gamma}$ of
$\widehat{\pi_1(\Sigma)}$, we introduce the space \[\homMaxHatLlPlus\]
of pairs $(\rho, L_+)$ consisting of a maximal representation $\rho$
of $\widehat{\pi_1(\Sigma)}$ whose first Stiefel-Whitney class is zero
and an attracting oriented Lagrangian $L_+$ for $\rho(
\widehat{\gamma})$. For those pairs, following the discussion in
Section~\ref{sec:eulerclassrank2}, we define an Euler class
\begin{equation*}
  \homMaxHatLlPlus \overset{ e_{\widehat{\gamma}} }{
    \longrightarrow} \h^1( T^1 \Sigma ; \ZZ ).
\end{equation*}

\subsubsection{Relations between the invariants}
\label{sec:relbetweenInv}
In this section we describe the relations between topological
invariants of maximal representations of $\pi_1(\Sigma)$ and
$\widehat{\pi_1(\Sigma)}$. More precisely:
\begin{itemize}
\item If $\rho: \pi_1(\Sigma) \to \Sp( 4, \RR)$ is a maximal
  representation, we compare the invariants of $\rho$ and $\widehat{
    \rho} = \rho \circ pr_2$.
\item If $\widehat{\rho} : \widehat{\pi_1(\Sigma)} \to \Sp(4, \RR)$ is
  a maximal representation and $\varepsilon : \widehat{\pi_1(\Sigma)}
  \to \{\pm 1\}$ a homomorphism, we compare the invariants of
  $\widehat{\rho}$ and $\varepsilon \cdot \widehat{\rho}$.
\end{itemize}

\begin{lem}
  Let $\rho$ be a maximal representation of $\pi_1(\Sigma)$ and let
  $\widehat{\rho} = \rho \circ pr_2$. Then
  \begin{equation*}
    sw_1( \rho) = sw_1( \widehat{\rho}).
  \end{equation*}
  
  When $sw_1( \rho) = 0$, let $\gamma$ be a nontrivial element of
  $\pi_1(\Sigma)$, $\widehat{\gamma}$ one of the two elements of
  $\widehat{\pi_1(\Sigma)}$ projecting to $\gamma$ and let $L_+$ be an
  attracting oriented Lagrangian for $\rho( \gamma) = \widehat{ \rho}(
  \widehat{ \gamma})$. Then
  \begin{equation*}
    e_{\gamma}( \rho, L_+) = e_{ \widehat{\gamma}}( \widehat{\rho}, L_+). 
  \end{equation*}
\end{lem}

\begin{proof}
  The (oriented) Lagrangian reductions associated with $\rho$ and for
  $\widehat{\rho}$ are exactly the same, hence their characteristic
  classes coincide.
\end{proof}

\begin{lem}
  Let $\widehat{\rho}$ be a maximal representation of
  $\widehat{\pi_1(\Sigma)}$ and $\varepsilon : \widehat{\pi_1(\Sigma)}
  \to \{ \pm 1\}$ a representation. Then the first Stiefel-Whitney
  class of $\widehat{\rho}$ and $\varepsilon \cdot \widehat{\rho}$
  coincide:
  \begin{equation*}
    sw_1( \widehat{\rho}) = sw_1(\varepsilon \cdot \widehat{\rho}).
  \end{equation*}
\end{lem}

\begin{proof}
  This lemma follows immediately from
  Proposition~\ref{prop:frst_ob_center} since, in this case, the map
  from $Z=\{\pm \id\}$ to the group $\pi_0(\GL(2,\RR))$ is zero
  ($-\id$ is in $\GL^+(2,\RR)$).
\end{proof}

\begin{prop}
  Let $\widehat{ \rho}: \widehat{ \pi_1(\Sigma)} \to \Sp(4, \RR)$ be a
  maximal representation with $sw_1(\widehat{\rho}) = 0$.  Let
  $\widehat{ \gamma}\in\widehat{ \pi_1(\Sigma)} $ be a nontorsion
  element and $L_+$ an attracting oriented Lagrangian for $\widehat{
    \rho}( \widehat{ \gamma})$. Let $\varepsilon : \widehat{
    \pi_1(\Sigma)} \to \{ \pm 1\}$ be a homomorphism. Then $L_+$ is an
  attracting oriented Lagrangian for $( \varepsilon \cdot \widehat{
    \rho})( \widehat{ \gamma})$ and the Euler class (relative to
  $\widehat{ \gamma}$) for the pairs $( \varepsilon \cdot \widehat{
    \rho}, L_+)$ and $( \widehat{ \rho}, L_+)$ are
  \begin{equation*}
    e_{\widehat{\gamma}}( \varepsilon \cdot \widehat{\rho} , L_+) =
    e_{\widehat{\gamma}}(\widehat{\rho} , L_+) \in  \h^2(T^1\Sigma;
    \ZZ)  \text{ if } \varepsilon( -1) = 1,
  \end{equation*}
  and
  \begin{equation*}
    e_{\widehat{\gamma}}( \varepsilon \cdot \widehat{\rho}, L_+) =
    e_{\widehat{\gamma}}(\widehat{\rho}, L_+) + (g-1)\GenTor \in
    \h^2(T^1\Sigma; \ZZ) \text{ if } \varepsilon( -1) = -1.
  \end{equation*}
\end{prop}

\begin{proof}
  Let $L_1$ and $L_2$ be the Lagrangian reductions associated with
  $\widehat{ \rho}$ and $\varepsilon \cdot \widehat{ \rho}$. Denote by
  $L_{1+}$ and $L_{2+}$ the corresponding oriented Lagrangian bundles
  determined by the choice of $\widehat{ \gamma}$ and $L_{+}$.  If
  $D_\varepsilon$ is the flat real line bundle over $T^1 \Sigma$
  associated with the representation $\varepsilon$, we have
  \begin{equation*}
    L_{2+} = D_\varepsilon \otimes L_{1+}.
  \end{equation*}
  (Because $L_1$ has even dimension there is a canonical orientation
  on $D_\varepsilon \otimes L_{1+}$, even if $D_\varepsilon$ is
  neither oriented nor necessarily orientable)

  Let $S_1$ and $S_2$ be the associated $S^1$-bundles corresponding to
  $L_{1+}$ and $L_{2+}$ and let $S_\varepsilon$ be the flat
  $S^1$-bundle associated with the representation $\varepsilon :
  \widehat{ \pi_1(\Sigma)} \to \{ \pm 1\} \subset S^1$, then the above
  equality can be restated as
  \begin{equation*}
    S_2 = S_\varepsilon \times_{S^1} S_1.
  \end{equation*}
  This implies for the Euler classes:
  \begin{equation*}
    e( S_2) = e( S_\varepsilon ) + e(S_1).
  \end{equation*}
  Since $e(S_2) = e_{ \widehat{\gamma}}( \varepsilon \cdot \widehat{
    \rho})$ and $e(S_1) = e_{ \widehat{\gamma}}(\widehat{ \rho})$, the
  proposition will follow from the following lemma.
\end{proof}

\begin{lem}
  \label{lem:eulerclassSonerep}
  Let $\varepsilon : \widehat{ \pi_1(\Sigma)} \to S^1$ be a
  representation and let $S_\varepsilon$ be the associated flat
  $S^1$-bundle. Then
  \begin{equation*}
    e(S_\varepsilon) = 0 \text{ if } \varepsilon( -1) = 1,
  \end{equation*}
  and
  \begin{equation*}
    e(S_\varepsilon) = (g-1)\GenTor \text{ if } \varepsilon( -1) = -1.
  \end{equation*}
\end{lem}
\begin{proof}
  First we note that $e( S_\varepsilon)$ varies continuously with
  $\varepsilon$. Hence $e$ only depends on the connected component of
  $\hom( \widehat{ \pi_1(\Sigma)}, S^1)$ containing $\varepsilon$.
  Since $\hom( \widehat{ \pi_1(\Sigma)} , S^1) = \hom( \{ \pm 1\}
  \times \pi_1(\Sigma), S^1) = \hom( \{ \pm 1 \} \times \ZZ^{2g}, S^1
  ) \cong \{ \pm 1\} \times ( S^1)^{2g}$, this space has two connected
  components distinguished precisely by the value of $\varepsilon(
  -1)$.

  We only need to calculate the Euler class for two specific
  representations. The first is the trivial representation for which
  the result is obvious.  The second one is the projection $pr_1 :
  \widehat{ \pi_1(\Sigma)} \cong \{ \pm 1 \} \times \pi_1(\Sigma) \to
  \{\pm 1\}$ onto the first factor. Consider the (non-flat)
  $S^1$-bundle over the surface: $\pi_1(\Sigma) \backslash \SL(2, \RR)
  \to \pi_1(\Sigma) \backslash \HH$, its Euler class is given by the
  Toledo number of the injection $\pi_1(\Sigma) \to \SL(2, \RR)$,
  which is $(g-1)$. One checks that the $S^1$-bundle $S_\varepsilon$
  is the pullback of this $S^1$-bundle by the natural projection
  $\widehat{ \pi_1(\Sigma)} \backslash \SL(2, \RR) \cong T^1 \Sigma
  \to \pi_1(\Sigma) \backslash \HH \cong \Sigma$. This implies the
  claim.
\end{proof}

\section{Cohomology}\label{sec:cohomology}

\subsection{The cohomology of $T^1\Sigma$}\label{subsec:cohomologysigma}
In this section we compute the cohomology of the unit tangent bundle
$T^1 \Sigma$ with $\ZZ$ and $\FF_2$ coefficients and study the
connecting homomorphism in the Mayer-Vietoris sequence. The results
are used in Section~\ref{sec:constraints} and
Section~\ref{sec:sta_res}.
\begin{prop}
  \label{prop:cohom_unit}
  Let $\Sigma$ be a closed, connected, oriented surface of genus
  $g>1$.  The cohomology groups of $T^1 \Sigma$ with $\ZZ$
  coefficients are:
  \begin{itemize}
  \item $\h^0( T^1 \Sigma; \ZZ) = \ZZ$,
  \item $\h^1( T^1 \Sigma; \ZZ) = \ZZ^{2g}$,
  \item $\h^2( T^1 \Sigma; \ZZ) = \ZZ^{2g} \times \ZZ/(2g-2)\ZZ$,
  \item $\h^3( T^1 \Sigma; \ZZ) = \ZZ$.
  \end{itemize}
  The cohomology groups of $T^1 \Sigma$ with $\FF_2$ coefficients are:
  \begin{itemize}
  \item $\h^0( T^1 \Sigma; \FF_2) = \FF_2$,
  \item $\h^1( T^1 \Sigma; \FF_2) = \FF_{2}^{2g+1}$,
  \item $\h^2( T^1 \Sigma; \FF_2) = \FF_{2}^{2g+1}$,
  \item $\h^3( T^1 \Sigma; \FF_2) = \FF_2$.
  \end{itemize}
\end{prop}

\begin{proof}
  Let $A$ be the ring $\ZZ$ or $\FF_2$.

  The unit tangent bundle $T^1 \Sigma \to \Sigma$ is a principal
  $S^1$-bundle whose Euler class $e$ is $(2-2g)$ in $\ZZ \cong \h^2(
  \Sigma ; \ZZ)$. The Gysin exact sequence with $A$-coefficients for
  this bundle is
  \begin{multline}\label{eq_gysin}
    0 \longrightarrow \h^0( \Sigma; A) \longrightarrow \h^0( T^1
    \Sigma; A) \longrightarrow 0 \longrightarrow \h^1( \Sigma; A) \\
    \longrightarrow \h^1( T^1 \Sigma; A) \longrightarrow \h^0( \Sigma;
    A) \xrightarrow{ \cupprod e_A} \h^2( \Sigma; A) \longrightarrow
    \h^2( T^1 \Sigma; A) \\ \longrightarrow \h^1( \Sigma; A)
    \longrightarrow 0 \longrightarrow \h^3( T^1 \Sigma; A)
    \longrightarrow \h^2( \Sigma; A) \longrightarrow 0,
  \end{multline}
  where $e_A$ is the image of $e$ under the natural map $ \h^2( \Sigma
  ; \ZZ) \to \h^2( \Sigma ; A)$.  The conclusion for $\h^0$ and $\h^3$
  follows immediately from this. When $A$ is $\ZZ$, $\h^0( \Sigma;
  \ZZ) \xrightarrow{ \cupprod e} \h^2( \Sigma; \ZZ)$ is injective and
  we get the exact sequences:
  \begin{equation*}
    0 \longrightarrow \ZZ^{2g} \longrightarrow \h^1( T^1 \Sigma; \ZZ)
    \longrightarrow 0,
  \end{equation*}
  and
  \begin{equation*}
    0 \longrightarrow \ZZ / (2g-2)\ZZ \longrightarrow \h^2( T^1 \Sigma; \ZZ)
    \longrightarrow \ZZ^{2g} \longrightarrow 0.
  \end{equation*}
  From this the result for $\h^1$ and $\h^2$ follows easily. When $A$
  is $\FF_2$, the connecting map $\h^0( \Sigma; A) \xrightarrow{
    \cupprod e_A} \h^2( \Sigma; A)$ is zero.
\end{proof}

The above proof gives a \emph{canonical} isomorphism
\begin{equation*}
  \h^2( T^1 \Sigma; \ZZ)^{tor} \cong \ZZ / (2g-2)\ZZ
\end{equation*}
between the torsion of $\h^2( T^1 \Sigma; \ZZ)$ and $\ZZ / (2g-2)\ZZ$.
In particular, $\GenTor$ is the \emph{canonical} generator of $ \h^2(
T^1 \Sigma; \ZZ)^{tor}$.

Let $\gamma$ be a simple closed oriented separating curve on the
surface $\Sigma$, i.e.\ $\Sigma \moins \gamma$ has two connected
components, $\Sigma_l$ denotes the component on the left of $\gamma$
and $\Sigma_r$ the component on the right (this uses the orientations
of $\gamma$ and $\Sigma$). This induces a decomposition of the unit
tangent bundle: $T^1 \Sigma$ is the union of $T^1 \Sigma|_{ \Sigma_l}$
and $T^1 \Sigma|_{ \Sigma_r}$ identified along $T^1 \Sigma|_{
  \gamma}$. The Mayer-Vietoris sequence for this decomposition reads
as
\begin{multline*}
  0 \longrightarrow \h^0( T^1 \Sigma ; \ZZ) \longrightarrow \h^0( T^1
  \Sigma|_{\Sigma_l} ; \ZZ) \oplus \h^0( T^1 \Sigma|_{\Sigma_r} ; \ZZ)
  \\ \longrightarrow \h^0( T^1 \Sigma|_{\gamma} ; \ZZ) \longrightarrow
  \h^1( T^1 \Sigma ; \ZZ) \longrightarrow \h^1( T^1 \Sigma|_{\Sigma_l}
  ; \ZZ) \oplus \h^1( T^1 \Sigma|_{\Sigma_r} ; \ZZ) \\ \longrightarrow
  \h^1( T^1 \Sigma|_{\gamma} ; \ZZ) \longrightarrow \h^2( T^1 \Sigma ;
  \ZZ) \longrightarrow \h^2( T^1 \Sigma|_{\Sigma_l} ; \ZZ) \oplus
  \h^2( T^1 \Sigma|_{\Sigma_r} ; \ZZ) \\ \longrightarrow \h^2( T^1
  \Sigma|_{\gamma} ; \ZZ) \longrightarrow \h^3( T^1 \Sigma ; \ZZ)
  \longrightarrow 0.
\end{multline*}
This sequence can also be used to compute the cohomology of the unit
tangent bundle. We concentrate on the connecting morphism
\begin{equation*}
  \delta:   \h^1( T^1 \Sigma|_{\gamma} ; \ZZ) \longrightarrow \h^2( T^1 \Sigma ; \ZZ)
\end{equation*}
and its kernel.

We realize $\gamma$ as a $C^1$ loop on $\Sigma$, it then has a natural
lift to the unit tangent bundle $T^1 \Sigma$ which we denote again by
$\gamma$. This lift induces a trivialization $T^1 \Sigma|_{ \gamma}
\cong S^1 \times \gamma$ and hence isomorphisms
\begin{equation*}
  \h^1( T^1 \Sigma|_{ \gamma} ; \ZZ ) \cong \h^1( S^1 ; \ZZ) \oplus
  \h^1( \gamma ; \ZZ ) \cong \ZZ \oplus \ZZ.
\end{equation*}
The first identification is the map in cohomology corresponding to the
projections $T^1 \Sigma|_{ \gamma} \to S^1$ and $T^1 \Sigma|_{ \gamma}
\to \gamma$ whereas the second identification involves the
orientations on $S^1$ and $\gamma$ (the orientation on $S^1 \cong
T_{x}^{1} \Sigma$ is induced by the orientation on $\Sigma$).

\begin{prop}
  \label{prop:mayervietoris}
  Let $\gamma$ be an oriented closed simple separating geodesic on the
  surface $\Sigma$.
  
  Then the orientation class $o_\gamma\in \h^1( \gamma; \ZZ) \cong
  \ZZ$ is sent to $\GenTor$ by the connecting homomorphism of the
  Mayer-Vietoris sequence:
  \begin{equation*}
    \h^1( \gamma; \ZZ) \subset \h^1( T^1 \Sigma|_{ \gamma} ; \ZZ)
    \overset{ \delta }{ \longrightarrow } \h^2( T^1 \Sigma ; \ZZ).
  \end{equation*}
  
  The kernel of $\delta$ is generated by the elements:
  \begin{equation*}
    (1, 1-2 g( \Sigma_l) ) \text{ and } (-1 , 1- 2g( \Sigma_r) ) \in
    \ZZ \times \ZZ \cong \h^1( T^1 \Sigma|_{ \gamma} ; \ZZ ).
  \end{equation*}
\end{prop}

\begin{proof}
  The connecting homomorphisms for the decompositions of the surface
  and the unit tangent bundle fit in a commutative diagram:
  \[ \xymatrix@-1em{ \h^1( \gamma; \ZZ) \ar[r]^{\delta} \ar[d] & \h^2(
    \Sigma ; \ZZ) \ar[d] \\
    \h^1( T^1 \Sigma|_{\gamma}; \ZZ) \ar[r]^{\delta} & \h^2( T^1
    \Sigma ; \ZZ). }\] So the first result follows from the equality
  $\delta( o_\gamma ) = o_\Sigma$, where $o_\Sigma$ is the orientation
  class in $\h^2( \Sigma; \ZZ)$. This equality is easy to
  establish. In fact the Mayer-Vietoris sequence for the surface:
  \begin{equation*}
    \h^1( \gamma; \ZZ) \longrightarrow \h^2( \Sigma ; \ZZ)
    \longrightarrow \h^2( \Sigma_l ; \ZZ) \oplus \h^2( \Sigma_r ; \ZZ)
  \end{equation*}
  already shows that $\delta : \h^1( \gamma; \ZZ) \to \h^2( \Sigma ;
  \ZZ)$ is surjective so that $\delta( o_\gamma) = \pm o_\Sigma$. The
  sign conventions are precisely arranged so that $\delta( o_\gamma) =
  o_\Sigma$.
  
  Due to the exactness of the Mayer-Vietoris sequence the kernel of
  $\delta$ is the image of
  \begin{equation*}
    \h^1( T^1 \Sigma|_{\Sigma_l} ; \ZZ) \oplus \h^1( T^1
    \Sigma|_{\Sigma_r} ; \ZZ) \longrightarrow \h^1( T^1
    \Sigma|_{\gamma} ; \ZZ).
  \end{equation*}
  It is therefore enough to show that the image of
  \begin{equation}\label{eq:imagesigma1}
    \h^1( T^1 \Sigma|_{\Sigma_l} ; \ZZ) \longrightarrow \h^1( T^1
    \Sigma|_{\gamma} ; \ZZ) \cong \ZZ \oplus \ZZ
  \end{equation}
  is generated by $( 1, 1-2g( \Sigma_l))$. (The calculation for
  $\Sigma_r$ is similar). The commutative square
  \[\xymatrix@-1em{ \h^1( \Sigma_l ; \ZZ ) \ar[d] \ar[r] & \h^1(
    \gamma;
    \ZZ) \ar[d] \\
    \h^1( T^1 \Sigma|_{\Sigma_l} ; \ZZ) \ar[r] & \h^1( T^1
    \Sigma|_{\gamma} ; \ZZ)}\] implies that the composition $\h^1(
  \Sigma_l ; \ZZ ) \subset \h^1( T^1 \Sigma|_{\Sigma_l} ; \ZZ) \to
  \h^1( T^1 \Sigma|_{\gamma} ; \ZZ)$ is zero, because the map $ \h^1(
  \Sigma_l ; \ZZ ) \to \h^1( \gamma; \ZZ)$ is zero as $\gamma$ is a
  boundary in $\Sigma_l$. The restriction $T^1 \Sigma|_{ \Sigma_l}$ is
  the trivial bundle $S^1 \times \Sigma_l$ so that
  \begin{equation*}
    \h^1( T^1 \Sigma|_{\Sigma_l} ; \ZZ) \cong \h^1( S^1 ; \ZZ) \times
    \h^1( \Sigma_l ; \ZZ).
  \end{equation*}
  This means that the image of the above map \eqref{eq:imagesigma1}
  has rank $1$ and is the image of
  \begin{equation*}
    \ZZ \cong \h^1( S^1 ; \ZZ) \longrightarrow \h^1( T^1 \Sigma|_{\gamma} ; \ZZ)
    \cong \ZZ \times \ZZ.
  \end{equation*}
  The first component of this last map is the identity $\ZZ \cong
  \h^1( S^1 ; \ZZ) \to \h^1( S^1 ; \ZZ) \cong \ZZ$ so that the image
  of \eqref{eq:imagesigma1} is generated by $(1, n)$ for some integer
  $n$. To calculate this integer $n$ let us consider the closed
  surface $\overline{ \Sigma}_1 = \Sigma_l \cup_{ \gamma} D^2$
  obtained by gluing a disk along $\gamma$. The genus of $\overline{
    \Sigma}_1$ is $g( \overline{ \Sigma}_1) = g( \Sigma_l)$ and the
  two $S^1$-bundles $T^1 \Sigma|_{\Sigma_l }$ and $T^1 \overline{
    \Sigma}_{1 | \Sigma_l }$ are isomorphic. From the Mayer-Vietoris
  sequence for the decomposition of $T^1 \overline{ \Sigma}_1$ we get
  \begin{multline*}
    \h^1( T^1 \Sigma|_{\Sigma_l } ; \ZZ) \oplus \h^1( T^1 \overline{
      \Sigma}_{1 | D^2 } ; \ZZ ) \longrightarrow \h^1( T^1
    \Sigma|_{\gamma } ; \ZZ) \longrightarrow \h^2( T^1 \overline{
      \Sigma}_1; \ZZ) \\ \longrightarrow \h^2( T^1 \Sigma|_{\Sigma_l }
    ; \ZZ) \oplus \h^2( T^1 \overline{ \Sigma}_{1 | D^2 } ; \ZZ )
    \longrightarrow \h^2( T^1 \Sigma|_{\gamma } ; \ZZ).
  \end{multline*}
  A small calculation shows that
  \[\ZZ \cong \h^1( T^1 \overline{
    \Sigma}_{1 | D^2 } ; \ZZ ) \to \h^1( T^1 \Sigma|_{\gamma } ; \ZZ)
  \cong \ZZ \times \ZZ
  \]
  sends $1$ to $( -1 , 1)$. The exact sequence reduces to
  \begin{equation*}
    0 \longrightarrow \ZZ^2 / \langle (1, n), (-1,1) \rangle
    \longrightarrow \h^2( T^1 \overline{ \Sigma}_1; \ZZ) \longrightarrow
    \ZZ^{2 g( \Sigma_l ) } \longrightarrow 0.
  \end{equation*} 
  This exact sequence implies that the torsion of $\h^2( T^1
  \overline{ \Sigma}_1; \ZZ)$ is isomorphic to $\ZZ / (n+1)\ZZ$. This
  torsion being isomorphic to $\ZZ / (2 g( \Sigma_l) - 2 )\ZZ$ by
  Proposition~\ref{prop:cohom_unit} the value of $n$ is $1 - 2g(
  \Sigma_l)$ or $2 g( \Sigma_l) -3$.
  
  Doubling $\Sigma_l$ along $\gamma$ to obtain a closed surface, a
  similar argument shows that the torsion of the double of $\Sigma_l$
  is isomorphic to $\ZZ / (2n)\ZZ$. Thus the only possible value for
  $n$ is $1- 2g( \Sigma_l)$ because the genus of the double of
  $\Sigma_l$ is $2g( \Sigma_l)$.
\end{proof}

\begin{lem}
  \label{lem:whenTwoClassIsTor}
  Let $\{ \eta_i\}$ be curves in $\Sigma$ whose images generate the
  homology group $\h_1( \Sigma ; \ZZ)$; we denote by $f_i : \eta_i \to
  \Sigma$ the inclusion. With a slight abuse of notation we write $T^1
  \Sigma |_{\eta_i}$ for the pulled back circle bundle $f_{i}^{*} T^1
  \Sigma$ and denote again by $f_i : T^1 \Sigma |_{ \eta_i} \to T^1
  \Sigma$ the corresponding map.
  
  Let $c$ be a class in $\h^2( T^1 \Sigma; A)$ such that, for all $i$,
  \[ c |_{T^1 \Sigma |_{ \eta_i}} := f_{i}^{*} ( c) = 0.\] Then $c$
  belongs to $\mathrm{ Im}( \h^2( \Sigma; A) \to \h^2( T^1 \Sigma;
  A))$.
\end{lem}

\begin{proof}
  Using the Gysin exact sequence with $A$ coefficients we only need to
  show that the image $a$ of $c$ in $\h^1( \Sigma; A)$ is zero. As $\{
  \eta_i\}$ generate the homology it will be the case if $f_{i}^{*}
  (a) = 0$ for all $i$. Observe that the Gysin sequences for $T^1
  \Sigma$ and for $T^1 \Sigma |_{\eta_i}$ fit in a commutative diagram
  \[
  \xymatrix@-1em{
    \h^2( T^1 \Sigma; A) \ar[r] \ar[d] & \h^1( \Sigma; A) \ar[d] \\
    \h^2( T^1 \Sigma |_{\eta_i}; A) \ar[r] & \h^1( \eta_i; A).}\]

  So the property follows from the hypothesis $f_{i}^{*}(c) = 0$.
\end{proof}

\subsection{Stiefel-Whitney classes for tensor products}

\begin{prop}\label{prop_swtensprod}
  Let $L$ a line bundle and $W$ a $n$-plane bundle over a
  (paracompact) base $B$. Then the first and second Stiefel-Whitney
  classes of the tensor product $L\otimes W$ are:
  \[ sw_1(L\otimes W) = n sw_1(L) +sw_1({W})\]
  \[
  sw_2(L\otimes W) = \frac{n(n-1)}{2} sw_1(L)\cupprod sw_1(L) + (n-1)
  sw_1(L)\cupprod sw_1({W})+sw_2({W}).\]
\end{prop}

\begin{proof}
  There exists $f:B_1 \to B$ such that $f^*W$ splits as the sum of
  $n$-line bundles and $f^*: \h^*(B) \to \h^*(B_1)$ is injective
  \cite[Ch.~16, Prop.~5.2]{Husemoller_FB}. Hence, up to pulling back,
  one can suppose that $W$ is the sum of $n$ line bundles $L_1, \dots,
  L_n$. Let $\eta$ the first Stiefel-Whitney class of $L$ and $\eta_i$
  the first Stiefel-Whitney class of $L_i$. Then the first
  Stiefel-Whitney class of $L\otimes L_i$ is $\eta+ \eta_i$
  (\cite[Ch.~16, Th.~3.4]{Husemoller_FB}).
 
  Since the total Stiefel-Whitney class is multiplicative under sums
  of bundles (\cite[Ch.~16, Sec.~3.1]{Husemoller_FB}) one finds the
  following formulas for the first and second Stiefel-Whitney classes
  of $W$ and $L\otimes W$:
  \begin{align*}
    sw_1(W) = \sum_{i=1,\dots, n} \eta_i,\quad &
    sw_1(L\otimes W) = \sum_{i=1,\dots, n} (\eta +\eta_i) \\
    sw_2(W) = \sum_{1\leq i < j \leq n} \eta_i \cupprod \eta_j,\quad &
    sw_2(L\otimes W) = \sum_{1\leq i < j \leq n} (\eta+\eta_i)
    \cupprod (\eta+\eta_j).
  \end{align*}
  By expanding the sums, the claim follows.
\end{proof}

\subsection{Euler class in Mayer-Vietoris sequence}

\begin{prop}
  \label{prop_eulclassMV}
  Let $F$ be an oriented $n$-plane bundle over a base $B$. Suppose
  that $B = B_{1} \cup_C B_2$ with $F|_{B_{1}}$ and $F|_{B_2}$ being
  trivial and that Mayer-Vietoris sequences hold for this
  decomposition of $B$. Let $h: C \to \GL^+(n, \RR)$ be the change of
  trivializations, $k: C \to \RR^n \moins \{0\}$ an orbital map
  associated with $h$ and $t$ in $\h^{n-1}(\RR^n \moins \{ 0\}; \ZZ)$
  the generator.

  Then the Euler class of $F$ is $\delta(k^*(t))$ where $\delta:
  \h^{n-1}(C; \ZZ) \to \h^n(B, \ZZ)$ is the connecting morphism in the
  Mayer-Vietoris sequence and $k^*: \h^{n-1}(\RR^n \moins \{ 0\}; \ZZ)
  \to \h^{n-1}(C; \ZZ)$ the map induced by $k$.

  In particular the image of $k^*(t)$ under the natural map from
  $\h^{n-1}(C; \ZZ)$ to $\hom(\h_{n-1}(C; \ZZ), \ZZ) $ is $k_* :
  \h_{n-1}(C; \ZZ) \to \h_{n-1}(\RR^n \moins \{ 0\}; \ZZ) \cong \ZZ$.
\end{prop}

\begin{proof}
  Below every cohomology group is understood with $\ZZ$-coefficients.

  Let $F^0$ the complement of the zero section in $F$.  The Thom class
  $u$ of $F$ is the unique class of $\h^n(F, F^0)$ such that $u$ is
  sent to the generator in $\h^n(F_b, F_b \moins \{0\})$ ($F_b \cong
  \RR^n$ is one fiber of $F$) and the Euler class $e$ is the image of
  $u$ under the natural map $\h^n(F, F^0) \to \h^n(F) \cong \h^n(B)$
  \cite[Ch.~16, Sec.~7]{Husemoller_FB}. We should first construct a
  class $x$ in $\h^{n-1}(F|_C)$ such that $\delta (x)=e$ in $\h^n(E)
  \cong \h^n(B)$ and then identify $x$ with $k^*(t)$. Consider the
  commutative diagram:
  \[
  \xymatrix@-1em{
    &  & \h^{n-1}(F|_{C}) \ar[d] \\
    & \h^{n-1}(F^0|_{B_{1}}) \oplus \h^{n-1}(F^0|_{B_2}) \ar[d]
    \ar[r] & \h^{n-1}(F^0|_{C}) \ar[d] \\
    \h^n(F,F^0) \ar[d] \ar[r] & \h^n(F|_{B_{1}},F^0|_{B_{1}}) \oplus
    \h^n(F|_{B_2},F^0|_{B_2}) \ar[d]
    \ar[r] & \h^n(F|_C,F^0|_{C}) \\
    \h^n(F) \ar[r] & \h^n(F|_{B_{1}}) \oplus \h^n(F|_{B_2}).  }\] The
  horizontal lines are Mayer-Vietoris exact sequences whereas the
  vertical lines are exact sequences of pairs. Each of those sequences
  arises from a short exact sequences of differential
  complexes\footnote{In this proof, we should not use any particular
    realization of the complexes calculating the cohomology. Rather we
    will only use the \emph{existence} of such complexes.}: $0 \to
  (C^*(F),d) \to (C^*(F|_{B_{1}}),d) \oplus (C^*(F|_{B_2}),d) \to
  (C^*(F|_{C}),d) \to 0$, etc. In the above diagram, consider the
  following classes:
  \[
  \xymatrix@-1em{
    &  & x \ar@{|->}[d] \\
    & v_1 \oplus v_2 \ar@{|->}[d]
    \ar@{|->}[r] & w \ar@{|->}[d] \\
    u \ar@{|->}[d] \ar@{|->}[r] & u_1 \oplus u_2 \ar@{|->}[d]
    \ar@{|->}[r] & 0 \\
    e \ar@{|->}[r] & 0 \oplus 0.  }\] The diagram is self-explanatory:
  $u_1 \oplus u_2$ is the image of $u$ and $u_1$ and $u_2$ are the
  Thom classes of $F|_{B_{1}}$ and $F|_{B_2}$.  Since these bundles
  are trivial, the corresponding Euler classes are zero and $u_1$ and
  $u_2$ lift to $v_1 \in \h^{n-1}(F^0|_{B_{1}})$ and $v_2 \in
  \h^{n-1}(F^0|_{B_2})$ respectively. The class $w$ in
  $\h^{n-1}(F^0|_C)$ is the image of $v_1 \oplus v_2$. It projects to
  $0$ in $\h^n(F|_C,F^0|_{C})$. Hence it comes from a class $x$ in
  $\h^{n-1}(F|_C)$.

  We claim that $\delta(x) \in \h^n(F)$ is the Euler class of $F$.
  Let $V_1 \in C^{n-1}(F^0|_{B_{1}})$ and $V_2 \in
  C^{n-1}(F^0|_{B_2})$ represent $v_1$ and $v_2$. Let also
  $\widetilde{V}_1 \in C^{n-1}(F|_{B_{1}})$ and $\widetilde{V}_2 \in
  C^{n-1}(F|_{B_2})$ be cochains extending $V_1$ and $V_2$, i.e.\
  $V_1$ is the image of $\widetilde{V}_1$ under the surjective map
  $C^{n-1}(F|_{B_{1}}) \to C^{n-1}(F^0|_{B_{1}})$. Then $U_1 = d
  \widetilde{V}_1$ and $U_2 = d \widetilde{V}_2$ are representatives
  of $u_1$ and $u_2$ and a representative of $u$ is $U=U_1 \oplus U_2$
  where we identify $C^n(F,F^0)$ with its image in
  $C^n(F|_{B_{1}},F^0|_{B_{1}}) \oplus
  C^n(F|_{B_2},F^0|_{B_2})$. Under the injection $C^n(F,F^0)\to
  C^n(F)$ $U$ represents the Euler class $e$.

  A cochain representing $w$ in $C^{n-1}(F^0|_C)$ is $W= V_2|_C -
  V_1|_C$ (the map from the complexes associated with $B_{1}$, or
  $B_2$, to the complexes associated with $C$ are simply denoted by
  $A\mapsto A|_C$). A cochain representing $x$ in $C^{n-1}(F|_C)$ is
  then $X= \widetilde{V}_2|_C - \widetilde{V}_1|_C$. Hence $X$ is the
  image of $\widetilde{V}_1 \oplus \widetilde{V}_2$ by the map
  $C^{n-1}(F|_{B_{1}}) \oplus C^{n-1}(F|_{B_2}) \to
  C^{n-1}(F|_C)$. Thus the element $d\widetilde{V}_1 \oplus d
  \widetilde{V}_2$ in $C^n(F|_{B_{1}}) \oplus C^n(F|_{B_2})$ lies in
  the image of the injective map $C^n(F)\to C^n(F|_{B_{1}}) \oplus
  C^n(F|_{B_2})$. By construction of the connecting morphism,
  $\delta(x)$ is represented by $d\widetilde{V}_1 \oplus d
  \widetilde{V}_2 = U_1 \oplus U_2 = U$. Therefore $e$ and $\delta(x)$
  are represented by the same cocycle, so they are equal. Note that an
  obvious diagram chasing shows that $\delta(x)$ does not depend on
  the choice of the lifts $v_1$ and $v_2$.

  The trivializations can now be used to get an explicit cycle
  $x$. For this we note that the trivialization $\phi_i: F|_{B_{i}}
  \to B_{i} \times \RR^n$ gives an isomorphism
  \[\phi_{i}^*: \h^*(B_{i} \times \RR^n, B_{i} \times (\RR^n \moins
  \{0\})) \longrightarrow \h^*(F|_{B_{i}}).\] Note that
  \[ \h^{*+n}(B_{i} \times \RR^n, B_{i} \times (\RR^n \moins \{0\}))
  \cong \h^*(B_{i}) \otimes \h^n(\RR^n, \RR^n \moins \{0\}).\] Under
  the isomorphism $\phi_{i}^*$, the class $u_i$ is the image of $1
  \otimes z$, $z$ being the generator in $\h^n(\RR^n, \RR^n \moins
  \{0\})$. The connecting homomorphism $\delta: \h^{n-1}(\RR^n \moins
  \{ 0\}) \to \h^n(\RR^n, \RR^n \moins \{0\})$ is an isomorphism that
  sends $t$ to $z$. The class $v_i$ lifting $u_i$ can be chosen to be
  the image of $1\otimes t$ by $\phi_{i}^{*}$. Hence $v_i$ is the
  image of $t$ by $\phi_{i}^{*} \circ p^*= (p \circ \phi_i)^*$ where
  $p: B_{i} \times (\RR^n \moins \{ 0\}) \to (\RR^n \moins \{ 0\}) $
  is the second projection.  By definition of $h$, for any $f$ in
  $F^0|_C$, one has the following equality $p \circ \phi_2(f) =
  (h\circ \pi(f)) \cdot (p \circ \phi_1(f))$ where $\pi: F^0 \to B$ is
  the projection. The next lemma implies that $w= v_2|_C - v_1 |_C $
  is equal to $(k \circ \pi)^*(t)$.  Since the map $\h^{n-1}(F|_C)
  \cong \h^{n-1}(C) \to \h^{n-1}(F^0|_C)$ is precisely $\pi^*$, one
  can set $x=k^*(t)$ for the lift of $w$.
\end{proof}

\begin{lem}
  Let $h: D \to \GL^+(n, \RR)$ and $\phi: D \to \RR^n \moins \{0\}$ be
  two continuous maps. Denote by $k: D \to \RR^n \moins \{0\} \sep d
  \mapsto h(d) \cdot v_0$ an orbital application and by $h\cdot \phi$
  the map $d \mapsto h(d) \cdot \phi(d)$.

  Then the following equality holds:
  \[ (h\cdot \phi)^* = k^* + \phi^* : \h^{n-1}(\RR^n \moins \{0\};
  \ZZ) \longto \h^{n-1}(D; \ZZ).\]
\end{lem}

\begin{proof}
  Denote by $o$ the map $\GL^+(n,\RR) \to \RR^n \moins \{0\} \sep g
  \mapsto g \cdot v_0$ and by $ev$ the map $\GL^+(n,\RR) \times (\RR^n
  \moins \{0\}) \to \RR^n \moins \{0\} \sep (g,v) \mapsto g \cdot
  v$. Hence $k= o\circ h$ and $h\cdot \phi = ev \circ (h, \phi)$. By
  the Künneth formula, the $(n-1)$-cohomology group of $\GL^+(n,\RR)
  \times (\RR^n \moins \{0\})$ decomposes as
  \[ \h^{n-1}(\GL^+(n,\RR) \times (\RR^n \moins \{0\}); \ZZ) \cong
  \h^{n-1}(\GL^+(n,\RR); \ZZ) \oplus \h^{n-1}( \RR^n \moins \{0\};
  \ZZ).\] and the map $ \h^{n-1}(\GL^+(n,\RR); \ZZ) \oplus \h^{n-1}(
  (\RR^n \moins \{0\}); \ZZ) \to \h^{n-1}(\GL^+(n,\RR) \times (\RR^n
  \moins \{0\}); \ZZ)$ is given by the two projections; its inverse is
  given by the inclusions of $\GL^+(n,\RR) \times\{v_0\}$ and $\{g_0\}
  \times (\RR^n \moins \{0\})$.  In this decomposition $ev^*$ is the
  map $(o^*, \id)$ and $(h,\phi)^*$ is $h^*+\phi^*$. Hence $(h\cdot
  \phi)^* = (h,\phi)^* \circ ev^* = h^* \circ o^* + \phi^* = k^* +
  \phi^*$.
\end{proof}

\def\cprime{$'$} \providecommand{\bysame}{\leavevmode\hbox
  to3em{\hrulefill}\thinspace}
\providecommand{\MR}{\relax\ifhmode\unskip\space\fi MR }
\providecommand{\MRhref}[2]{%
  \href{http://www.ams.org/mathscinet-getitem?mr=#1}{#2} }
\providecommand{\href}[2]{#2}

\end{document}